\documentclass[final,onefignum,onetabnum]{siamonline190516}

\usepackage{lipsum}
\usepackage{amsfonts}
\usepackage{epstopdf}
\ifpdf
  \DeclareGraphicsExtensions{.eps,.pdf,.png,.jpg}
\else
  \DeclareGraphicsExtensions{.eps}
\fi

\usepackage{datetime}
\newdateformat{monthyeardate}{%
  \monthname[\THEMONTH], \THEYEAR}

\usepackage{academicons}

\usepackage{xcolor}

\usepackage{amsmath}
\allowdisplaybreaks
\usepackage{amssymb}
\usepackage{commath}
\usepackage{mathtools}
\usepackage{bbm}
\usepackage{bm}

\usepackage{color}
\usepackage{graphicx}
\usepackage[small]{caption}
\usepackage{subcaption}

\usepackage{relsize}
\usepackage{adjustbox}
\usepackage{algorithm}
\usepackage[noend]{algpseudocode}
\usepackage{booktabs}
\usepackage{tikz}
\usepackage{pifont}% http://ctan.org/pkg/pifont for cmark

\usepackage{enumitem}
\setlist[enumerate]{leftmargin=.5in}
\setlist[itemize]{leftmargin=.5in}

\newsiamremark{remark}{Remark}
\newsiamremark{hypothesis}{Hypothesis}
\crefname{hypothesis}{Hypothesis}{Hypotheses}
\newsiamthm{claim}{Claim}
\headers{Towards Stable RBF Methods}{Glaubitz, Le M\'el\'edo, and \"Offner}

\title{Towards Stable Radial Basis Function Methods for Linear Advection Problems\thanks{\monthyeardate\today \corresponding{Jan Glaubitz (\email{Jan.Glaubitz@Dartmouth.edu}
)} 
\funding{This work is partially supported by the German Research Foundation (DFG, Deutsche Forschungsgemeinschaft) \#GL 927/1-1 (Glaubitz), SNF \#175784 (Le M\'el\'edo and \"Offner), and UZH Postdoc Forschungskredit \#FK-19-104 (\"Offner).}}}

\author{Jan Glaubitz\thanks{Department of Mathematics, Dartmouth College, Hanover, NH 03755, USA}
\and 
Elise Le M\'el\'edo\thanks{Institut f\"ur Mathematik, Universit\"at Z\"urich, CH-8057 Z\"urich, Switzerland} 
\and 
Philipp \"Offner\thanks{Institut f\"ur Mathematik, Johannes Gutenberg Universität, 55099 Mainz, Germany}
}

\usepackage{amsopn}

\usepackage[normalem]{ulem}

\DeclareMathOperator{\diag}{diag}

\DeclareMathOperator*{\cond}{cond}
\newcommand{\scp}[2]{\left\langle{#1,\, #2}\right\rangle} 
\renewcommand{\d}{\mathrm{d}} 
\newcommand{\intd}{\, \mathrm{d}}
\newcommand{\N}{\mathbb{N}}

\newcommand{\R}{\mathbb{R}} 
\newcommand{\fnum}{f^{\mathrm{num}}}
\newcommand{\cmark}{\ding{51}}%
\ifpdf
\hypersetup{
  pdftitle={Glaubitz2021Towards},
  pdfauthor={Jan Glaubitz}
}
\fi

\begin{document}

\maketitle

% REQUIRED
\begin{abstract}
	In this work, we investigate (energy) stability of global radial basis function (RBF) methods for linear advection problems. 
Classically, boundary conditions (BC) are enforced strongly in RBF methods. 
By now it is well-known that this can lead to stability problems, however. 
Here, we follow a different path and propose two novel RBF approaches which are based on a weak enforcement of BCs.
By using the concept of flux reconstruction and simultaneous approximation terms (SATs), respectively, we are able to prove that both new RBF schemes are strongly (energy) stable. 
Numerical results in one and two spatial dimensions for both scalar equations and systems are presented, supporting our theoretical analysis. 
\end{abstract}

% REQUIRED

% REQUIRED

\section{Introduction} 
\label{sec:introduction}

Since their introduction in Hardy's work \cite{hardy1971multiquadric} on cartography in 1971, RBFs have become a powerful tool in numerical analysis, including multivariate interpolation and approximation theory. 
This is because they are, for instance, easy to implement, allow arbitrary scattered data, and can be highly accurate.
Hence, RBFs are also often used in the numerical treatment of partial differential equations (PDEs) \cite{kansa1990multiquadrics, fasshauer1996solving, hon1998efficient, kansa2000circumventing, larsson2003numerical, platte2004computing, fornberg2015primer} and are considered as a viable alternative to more traditional methods such as finite difference (FD), finite volume (FV), finite element (FE), and spectral schemes.
In this work, we focus on global RBF methods for linear advection problems of the form
\begin{equation}\label{eq:linear-adv}
  \partial_t \boldsymbol{u} + \sum_{r=1}^d A_r \partial_{x_r} \boldsymbol{u} = 0, 
  \quad \boldsymbol{x} = (x_1,\dots,x_d) \in \Omega \subset \R^d, \ t> 0, 
\end{equation}
subject to appropriate boundary and initial conditions (ICs); see \S \ref{sec:problem} for more details. 
Here, the vector of unknowns ${\boldsymbol{u} = (u_1,\dots,u_m)^T}$ contains the \emph{conserved variables}.
Note that under certain assumptions on the matrices ${A_r \in \R^{m \times m}}$, ${r=1,\dots,d}$, this setting yields the subclass of hyperbolic conservation laws with linear flux functions \cite{lax1973hyperbolic, dafermos2010hyperbolic, whitham2011linear}.

\subsection{State of the Art}

Even though RBF methods have a long-standing history in the context of numerical PDEs, their stability theory can still be considered as under-developed, especially when compared to more traditional methods.   
For instance, to the best of our knowledge, energy stability for RBF methods has only been considered in the recent works \cite{glaubitz2019stability} and \cite[Chapter 7.2]{glaubitz2020shock}. 
There, it was shown that usual RBF methods --- for which BCs are enforced strongly --- are not energy stable and might therefore produce physically unreasonable numerical solutions. 
Furthermore, it appears that the differentiation matrices of RBF methods encountered in time-dependent PDEs often have eigenvalues with a positive real part resulting in unstable methods; see \cite{platte2006eigenvalue}.
Hence, in the presence of rounding errors, these methods are less accurate \cite{kansa2000circumventing, platte2004computing, schaback1995error} and can become unstable in time unless a dissipative time integration method \cite{martel2016stability, platte2006eigenvalue}, artificial dissipation 
\cite{flyer2016enhancing, ranocha2018stability, glaubitz2016artificial, glaubitz2019smooth, glaubitz2019analysis}, 
or some other stabilizing technique 
\cite{scarnati2018using, gelb2019numerical, glaubitz2019high, hesthaven2008filtering, glaubitz2018application, glaubitz2019shock, don2016hybrid}  
is used.
So far, this issue was only overcome for problems which are free of BCs \cite{martel2016stability}. 
Especially, it was proven in \cite{martel2016stability} that RBF collocation methods are time-stable (in the sense of eigenvalues for linear problems) for all conditionally positive definite RBFs and node distributions when no BCs  are present. 
RBF methods might therefore be well suited for periodic domains, such as circles and spheres, but not for applications with inflow-outflow conditions or general initial-boundary-value (IBV) problems where periodicity of the computational domain $\Omega$ cannot be assumed.

\subsection{Our Contribution}

The present work strives to establish stable RBF methods also in the presence of BCs.
In particular, we propose and investigate two new approaches that involve a weak enforcement of BCs. 
These approaches yield two novel RBF schemes which are provably energy stable for linear advection equations \eqref{eq:linear-adv}.
The first method draws inspiration from flux reconstruction (FR) schemes \cite{huynh2007flux}, a subclass of spectral element (SE) methods, and is subsequently referred to as the \emph{FR-RBF} method. 
The second method, on the other hand, is inspired by SATs. 
These originate from FD methods \cite{kreiss1989initial, gustafsson1995time, gustafsson2007high} and, together with summation by parts (SBP) operators \cite{svard2014review,fernandez2014review}, can be used to construct energy stable methods for many problems. 
Recently, this technique has been demonstrated to be viable also in the context of FE methods \cite{abgrall2019analysis, abgrall2019analysis_2}. 
Here, we adapt this approach to construct energy stable RBF methods for linear advection equations. Henceforth, this second class of energy stable RBF methods will be called \emph{SAT-RBF methods}.
Both methods are thoroughly investigated --- theoretically as well as numerically --- and their  advantages but also pitfalls are highlighted.
In our opinion, these methods should not be considered as some kind of 'ultimate' RBF schemes for linear advection equations. 
Rather, we hope that the present investigation will pave the way towards the development of a more mature stability theory for RBF methods. 
Moreover, the present work reveals some new connections between RBF methods and classical FD and FE schemes.
Finally, it should be stressed that we focus on global RBF methods. 
While the extension to local RBF methods, e.\,g.\ RBF-FD methods \cite[Chapter 5]{fornberg2015primer}, would be highly desirable, such an investigation would exceed the scope of this manuscript. 
Yet, we intend to address energy stability for local RBF methods in future works.

\subsection{Outline}

The rest of this manuscript is organized as follows. 
In \S \ref{sec:problem}, we detail the problem statement and main goal of this work, which is the construction of energy stable RBF methods for linear advection equations. 
\S \ref{sec:RBF_methods} then provides a short recap of all necessary preliminaries on RBF methods. 
The heart of this work are \S \ref{sec:RBF_FR} and \S \ref{sec:RBF_SAT}, where we introduce and analyze two new approaches to weakly enforce BCs in RBF schemes. 
This results in the FR-RBF and SAT-RBF method.
Both schemes are proven to be conservative and energy stable for linear advection equations.
Yet, it should be pointed out that problems related to numerical stability arise in the construction of the FR-RBF method, as it can already be observed for scalar linear advection problems in one spatial dimension. 
We therefore only address the extension of the SAT-RBF method to multiple spatial dimensions and systems. 
Some other extensions are discussed in \S \ref{sec:extensions}. 
These include (entropy) stability for nonlinear problems, local RBF-FD methods, and replacing exact integrals by discrete quadrature/cubature formulas.
Finally, \S \ref{sec:num} provides numerical results for both methods, along with a comparison with the usual RBF method (where the BCs are enforced strongly). 
The tests are performed for scalar equations and systems in one and two spatial dimensions. 
Concluding thoughts and some final remarks are offered in \S \ref{sec:summary}.  
\section{Problem Statement} 
\label{sec:problem} 

Let $\Omega \subset \R^d$ be a bounded domain with boundary $\partial \Omega$.
We consider linear advection equations of the form \eqref{eq:linear-adv} equipped with suitable IC and BC: 
\begin{equation}\label{eq:model_problem}
\begin{aligned}
	\partial_t \boldsymbol{u} + \sum_{r=1}^d A_r \partial_{x_r} \boldsymbol{u} & = 0, 
  		\quad && \boldsymbol{x} \in \Omega, \ t> 0, \\ 
  	\boldsymbol{u} & = \boldsymbol{u}_{\text{init}}, 
		\quad && \boldsymbol{x} \in \Omega, \ t = 0, \\ 
  	\boldsymbol{u} & = \boldsymbol{g} , 
		\quad && \bm{x} \in \partial \Omega, \ t > 0,  
\end{aligned}
\end{equation} 
where $\bm{u}_{\text{init}}$ denotes the IC and $\bm{g}$ describes the BC. 
Following \cite{nordstrom2017roadmap}, the conditions on $\bm{u}_{\text{init}}$ and $\bm{g}$  are chosen such that \eqref{eq:model_problem} is well-posed. 
For ease of notation, we focus on the one-dimensional case with positive constant velocity ${A_1 = a > 0}$ for the moment. 
Then, \eqref{eq:model_problem} reduces to
\begin{equation}\label{eq:linear_ad}
\begin{aligned}
    \partial_t u+ a \partial_x u & = 0, \quad && x_L < x < x_R, \ t>0,\\
    u(x,0) & = u_{\text{init}}(x), \quad   && x_L\leq x\leq x_R,\\
    u(x_L,t) & = g(t), \quad &&  t > 0.
\end{aligned}
\end{equation}
Note that in this case the BC is only defined at the left boundary. 
In general, the BC will be enforced at the inflow part of the boundary, denoted by $\partial \Omega^-_.$
As it is described in \S \ref{sec:RBF_methods}, the spatial and temporal discretization of \eqref{eq:linear_ad} is often decoupled, yielding an ordinary differential equation (ODE), called the \emph{semidiscrete equation}, which is then integrated in time.
Similarly to discretizations used in the context of FE methods, for a fixed time $t$, the function ${u(\cdot,t):[x_L,x_R] \to \R}$ is approximated by a suitable combination of 
basis functions $\varphi_1,\dots,\varphi_N$ whose coefficients are to be computed. 

Let us denote this semidiscretization at a fixed time $t$ by ${u_N(t):[x_L,x_R] \to \R}$. 
Thus, ${u(x,t) \approx u_N(x,t)}$ for all ${x \in [x_L,x_R]}$ and fixed $t\geq0$ and \eqref{eq:linear_ad} becomes 
\begin{equation}\label{eq:linear_ad:semidisc}
\begin{aligned}
    \partial_t u_N + a \partial_x u_N & = 0, \quad && x_L < x < x_R, \ t>0,\\
    u_N(x,0) & = (u_{\text{init}})_N(x), \quad   && x_L\leq x\leq x_R,\\
    u_N(x_L,t) & = g(t), \quad &&  t > 0,
\end{aligned}
\end{equation}
where $(u_{\text{init}})_N$ denotes is the approximation of the initial condition $u_{\text{init}}$.
In this context, stability of a semidiscretization is usually defined as follows; see \cite{nordstrom2017roadmap}.

\begin{definition}\label{def:energy_stable}
 The semidiscrete method \eqref{eq:linear_ad:semidisc} is called \emph{strongly energy stable} if 
 \begin{equation}\label{es:energy_stable}
  	\norm{u_N(t)}^2 \leq K(t) \left( \norm{(u_{\text{init}})_N}^2 + \max_{s \in [0,t]} |g(s)|^2  \right)
 \end{equation}
 holds for some $K(t)$ bounded for any finite $t$ and independent of $u_{\text{init}}$, $g$, 
 and $N$ (the grid's resolution). 
 Furthermore, $\norm{\cdot}$ is the norm corresponding to some inner product on the linear space $V_N$ spanned by the basis functions; that is, $V_N = \operatorname{span}\{ \varphi_1,\dots,\varphi_N \}$. 
\end{definition}

Here, we consider the $L^2$ norm, given by
\begin{equation}
	\norm{f}^2 = \int_{\Omega} f^{2}(x) \intd x, \quad f \in L^2(\Omega), 
\end{equation} 
which is a usual choice for $\norm{\cdot}$. 
Other often considered norms include 
certain Sobolev \cite{vincent2011new} and discrete norms 
\cite{gassner2013skew,ranocha2016summation,ranocha2018stability,offner2018stability,glaubitz2019stableDG}. 
In the context of RBF methods, appropriate discrete norms would correspond to stable high-order quadrature/cubature rules for (potentially) scattered data points \cite{huybrechs2009stable,glaubitz2020stable,glaubitz2020stableCF,glaubitz2020constructing}.
For nonlinear problems, sometimes also the total variation is considered; see  
\cite{harten1984class,cockburn1989tvb,cockburn1991runge,toro2013riemann,glaubitz2019shock} 
and references therein.

\begin{remark}
Definition \ref{def:energy_stable} is adapted to problem \eqref{eq:linear_ad} where only one boundary term is fixed. However, if an additional forcing function was considered on the right hand side of \eqref{eq:linear_ad}, we would need to include the maximum of this function in \eqref{es:energy_stable} in the same fashion as for $g$; see \cite{svard2014review} for more details. 
\end{remark}

In particular, strong energy stability ensures that the numerical solution is not unconditionally increasing. 
It is therefore one of the basic requirements a method should fulfill. 
In \S \ref{sec:RBF_FR} and \S \ref{sec:RBF_SAT}, we propose and investigate two new approaches to ensure strong energy stability for RBF methods.
\section{Preliminaries on Radial Basis Function Methods} 
\label{sec:RBF_methods} 

In this section, we collect all the necessary preliminaries regarding RBF methods. 
For more details, we recommend the monographs \cite{buhmann2003radial,wendland2004scattered,fasshauer2007meshfree,fornberg2015primer} as well as the reviews \cite{fornberg2015solving,iske2003radial}.
It should be stressed that in recent years fundamental advances have occurred, and that many of these are only covered in \cite{fornberg2015primer}.

\subsection{Method of Lines}
\label{sub:method-of-lines}

Many RBF methods for time-dependent PDEs build up on the \emph{method of lines} \cite{leveque2002finite}. In this approach, problem \eqref{eq:linear-adv} initially remains continuous in time and only a spatial discretization is considered. 
This results in a system of ODEs
\begin{equation}\label{eq:semidis-eq}
  \frac{\d}{\d t} \mathbf{u}_j = \operatorname{L}(\mathbf{u})_j, 
  \quad j=1,\dots,m,
\end{equation}
usually referred to as the \emph{semidiscrete equation}. 
Here, $\mathbf{u}$ is given by ${\mathbf{u} = (\mathbf{u}_1,\dots,\mathbf{u}_m)}$, where 
${\mathbf{u}_j = \mathbf{u}_j(t)}$ denotes the vector of coefficients of the spatial semidiscretization $(u_j)_N$ of the $j$th component ${u_j = u_j(t,\boldsymbol{x})}$. 
Furthermore, $\operatorname{L}$ represents the operator for the spatial semidiscretization of the right hand side of \eqref{eq:linear-adv}.
Once such a spatial semidiscretization has been defined, the semidiscrete equation \eqref{eq:semidis-eq} is evolved in time by some usual time integration method. 
A popular choice with favorable stability properties are strong stability preserving (SSP) Runge--Kutta (RK) methods \cite{shu1988total, gottlieb1998total, levy1998semidiscrete, gottlieb2001strong, ketcheson2008highly}.
For all numerical tests presented in this work, we used the explicit SSP-RK method of third order using three stages (SSPRK(3,3)), given as follows.

\begin{definition}[SSPRK(3,3)]\label{def:SSP}
Let $\mathbf{u}^n$ be the solution at time $t^n$.  The solution $\mathbf{u}^{n+1}$ at time $t^{n+1}$ is computed as
\begin{equation}\label{eq:SSPRK33}
\begin{aligned}
  \mathbf{u}^{(1)} & = \mathbf{u}^n + \Delta t \operatorname{L}\left( \mathbf{u}^n \right), \\ 
  \mathbf{u}^{(2)} & = \frac{3}{4} \mathbf{u}^n + \frac{1}{4} \mathbf{u}^{(1)} + \frac{1}{4} \Delta t 
\operatorname{L}\left( \mathbf{u}^{(1)} \right), \\ 
  \mathbf{u}^{n+1} & = \frac{1}{3} \mathbf{u}^n + \frac{2}{3} \mathbf{u}^{(2)} + \frac{2}{3} \Delta t 
\operatorname{L}\left( \mathbf{u}^{(2)} \right). 
\end{aligned}
\end{equation}
\end{definition}

The time step size $\Delta t$ in (\ref{eq:SSPRK33}) is computed as $\Delta t = \frac{Ch}{\lambda_{\text{max}}}$ with $C = 0.1$, where $\lambda_{\text{max}}$ denotes the largest characteristic velocity. 
Moreover, $h$ is the smallest distance between any two distinct centers. 
That is, ${h = \min_{i\neq j} \norm{\mathbf{x}_i - \mathbf{x}_j}_2}$. 
See \S \ref{sub:RBF-interpol} for more details.

For the general linear advection equation \eqref{eq:linear-adv}, the largest characteristic velocity $\lambda_{\text{max}}$ is given by the largest absolute eigenvalue among all matrices $A_1,\dots,A_d$.\footnote{Usually, these matrices are assumed to be diagonalizable. This yields solutions that can be decomposed into traveling waves, which speed and direction is respectively given by the eigenvalues and corresponding (normalized) eigenvector. See, for instance, \cite[Chapter 11.1]{evans2010partial} or \cite[Remark 2.2]{glaubitz2020shock}.}

\subsection{Radial Basis Function Interpolation}
\label{sub:RBF-interpol}

Let us now consider the approximation of a function $u: \Omega \to \R$ with $\Omega \subset \R^d$ by RBF interpolants. 
Given a set of $N$ distinct points $\mathbf{x}_n \in \Omega$, $n=1,\dots,N$, called \emph{centers} (or \emph{nodes}), the \emph{RBF interpolant of $u$ w.\,r.\,t.\ $\phi$} is given by 
\begin{equation}\label{eq:pure-RBF-interpol}
  u_N(\boldsymbol{x}) = \sum_{n=1}^N \alpha_n \phi \left( \, \norm{\boldsymbol{x}-\mathbf{x}_n}_2 \right),
\end{equation}
where $\phi:[0,\infty) \to \R$ only depends on the norm (radius) $r = \norm{\boldsymbol{x}-\mathbf{x}_n}_2$ and is therefore referred to as an \emph{RBF} (or \emph{kernel}). 
The coefficients $\alpha_n$, $n=1,\dots,N$, are then uniquely determined by the \emph{interpolation conditions} 
\begin{equation}\label{eq:interpol-cond}
  u_N(\mathbf{x}_n) = u(\mathbf{x}_n), \quad n=1,\dots,N.
\end{equation}
These yield a system of linear equations,  
\begin{equation}\label{eq:LS-pure-RBF-interpol}
  \underbrace{\begin{pmatrix} 
    \phi\left( \norm{\mathbf{x}_1-\mathbf{x}_1}_2 \right) & \dots & \phi\left( \norm{\mathbf{x}_1-\mathbf{x}_N}_2 \right) 
\\ 
    \vdots & & \vdots \\ 
    \phi\left( \norm{\mathbf{x}_N-\mathbf{x}_1}_2 \right) & \dots & \phi\left( \norm{\mathbf{x}_N-\mathbf{x}_N}_2 \right)
  \end{pmatrix}}_{=: \Phi}
  \underbrace{\begin{pmatrix} 
    \alpha_1 \\ \vdots \\ \alpha_N
  \end{pmatrix}}_{=: \boldsymbol{\alpha}}
  = 
  \underbrace{\begin{pmatrix} 
    u(x_1) \\ \vdots \\ u(x_N)
  \end{pmatrix}}_{=: \mathbf{u}},
\end{equation}
which can be solved for the vector of coefficients $\boldsymbol{\alpha} \in \R^N$ if the matrix $\Phi$ is invertible. Some popular examples of RBFs are listed in Table \ref{tab:RBFs} and even more can be found in the literature \cite{iske2003radial, buhmann2003radial, wendland2004scattered, fasshauer2007meshfree, fornberg2015solving}. 

\begin{table}[htb!]
  \centering 
  \renewcommand{\arraystretch}{1.3}
  \begin{tabular}{c|c|c|c}
    RBF & $\phi(r)$ & parameter & order \\ \hline 
    Gaussian & $\exp( -(\varepsilon r)^2)$ & $\varepsilon>0$ & 0 \\ 
    Multiquadrics & $\sqrt{(\varepsilon r)^2 + 1}$ & $\varepsilon>0$ & $1$ \\ 
    Polyharmonic splines & $r^{2k-1}$ & $k \in \N$ & $k$ \\ 
    & $r^{2k} \log r$ & $k \in \N$ & $k+1$ \\
  \end{tabular} 
  \caption{Some popular RBFs}
  \label{tab:RBFs}
\end{table}

In many situations the RBF interpolant \eqref{eq:pure-RBF-interpol} is desired to include polynomials up to a certain degree, together with matching constraints on the expansion coefficients. 
Let $\{p_i\}_{i=1}^Q$ be a basis of the space of (algebraic) polynomials of degree less than $m$, denoted by $\mathbb{P}_{m-1}(\R^d)$.\footnote{Note that $Q$ denotes the dimension of $\mathbb{P}_{m-1}(\R^d)$,
given by $Q = \binom{d+m-1}{d}$.} 
Then, including polynomials of degree less than $m$ into the RBF interpolant \eqref{eq:pure-RBF-interpol} results in 
\begin{equation}\label{eq:RBF-interpol}
  u_N(\boldsymbol{x}) 
    = \sum_{n=1}^N \alpha_n \phi \left( \, \norm{\boldsymbol{x}-\mathbf{x}_n}_2 \right) 
    + \sum_{i=1}^Q \beta_i p_i(\boldsymbol{x}).
\end{equation}
Moreover, the matching constraints  
\begin{equation}\label{eq:matching-cond}
  \sum_{n=1}^N \alpha_n p_i(\mathbf{x}_n) = 0, \quad i = 1,\dots,Q,
\end{equation}
have to be satisfied. 
These are supposed to ensure that \eqref{eq:RBF-interpol} is still uniquely determined by the interpolation conditions \eqref{eq:interpol-cond}. 
Let us denote 
\begin{equation}\label{eq:P} 
  P = 
  \begin{pmatrix} 
    p_1\left( \mathbf{x}_1 \right) & \dots & p_1\left( \mathbf{x}_N \right) \\ 
    \vdots & & \vdots \\ 
    p_Q\left( \mathbf{x}_1 \right) & \dots & p_Q\left( \mathbf{x}_N \right)
  \end{pmatrix}, \quad 
  \boldsymbol{\beta} = 
  \begin{pmatrix} 
    \beta_1 \\ 
    \vdots \\ 
    \beta_Q
  \end{pmatrix}.
\end{equation}
Then, \eqref{eq:matching-cond} can be rewritten as $P \boldsymbol{\alpha} = \mathbf{0}$ and the whole set of the expansion coefficients ${\alpha_1,\dots,\alpha_N}$ and ${\beta_1,\dots,\beta_Q}$ in \eqref{eq:RBF-interpol} can be recovered from the system of linear equations 
\begin{equation}\label{eq:LS-RBF-interpol}
  \underbrace{\begin{pmatrix} 
    \Phi & P^T \\ 
    P & 0
  \end{pmatrix}}_{=: V}
  \begin{pmatrix} 
    \boldsymbol{\alpha} \\ \boldsymbol{\beta}
  \end{pmatrix}
  = 
  \begin{pmatrix} 
    \mathbf{u} \\ \mathbf{0}
  \end{pmatrix},
\end{equation}
where $V$ is referred to as the \emph{Vandermonde matrix}.
There are various good reasons for including polynomials in RBF interpolants 
\cite{schaback1995error,buhmann2000radial,flyer2016role,flyer2016enhancing}:

\begin{enumerate}
  \item 
  Polynomial terms can ensure that \eqref{eq:LS-RBF-interpol} is uniquely solvable when conditionally positive definite 
RBFs are used, assuming that the set of centers $X := \{\mathbf{x}_n\}_{n=1}^N$ is 
$\mathbb{P}_{m-1}(\R^d)$-unisolvent.\footnote{We say that an RBF $\varphi$ is \emph{conditionally positive definite 
of order $m$} on $\R^d$ if $\boldsymbol{\alpha}^T \Phi \boldsymbol{\alpha} > 0$ holds for every set of distinct 
centers $X \subset \R^d$ and all $\boldsymbol{\alpha} \in \R^N\setminus\{\mathbf{0}\}$ that satisfy 
\eqref{eq:matching-cond}.} 
  See, for instance, \cite[Chapter 7]{fasshauer2007meshfree} or \cite[Chapter 3.1]{glaubitz2020shock}.
  
  \item 
  Numerical tests demonstrate that including a constant improves the accuracy of derivative approximations. 
  In particular, adding a constant avoids oscillatory representations of constant functions.
  
  \item 
  Including polynomial terms of low order can also improve the accuracy of RBF interpolants near domain boundaries due to regularizing the far-field growth of RBF interpolants \cite{fornberg2002observations}. 
\end{enumerate}

Finally, we note that the set of all RBF interpolants \eqref{eq:RBF-interpol} forms an $N$-dimensional linear space, denoted by $V_{N,m}$. 
This space is spanned, for instance, by the basis elements 
\begin{equation}\label{eq:nodal_basis}
  \psi_k(\boldsymbol{x}) = \sum_{n=1}^N \alpha_n^{(k)} \phi( \norm{ \boldsymbol{x} - \mathbf{x}_n }_2 ) 
    + \sum_{i=1}^Q \beta_i^{(k)} p_i(\boldsymbol{x}), 
    \quad k=1,\dots,N,
\end{equation} 
which (expansion coefficients) are uniquely determined by 
\begin{equation}\label{eq:cond1_psi}
  \psi_k(\mathbf{x}_n) = \delta_{kn} := 
  \begin{cases} 
    1 & \text{if } k=n, \\ 
    0 & \text{otherwise}, 
  \end{cases} 
  \quad n=1,\dots,N,
\end{equation}
and the matching conditions \eqref{eq:matching-cond}. 
The basis $\{ \psi_k \}_{k=1}^N$ can be considered as a nodal basis and comes with the advantage of providing a 
representation of the RBF interpolant in which the expansion coefficients are simply given by the known nodal values of 
$u$: 
\begin{equation}
  u_N(\boldsymbol{x}) = \sum_{k=1}^N u(\mathbf{x}_k) \psi_k(\boldsymbol{x})
\end{equation} 
This representation will be convenient for the implementation of the latter proposed RBF methods.

\subsection{Usual Radial Basis Function Methods} 

For sake of simplicity, we only outline the procedure for usual RBF methods for a scalar linear advection equation of the form 
\begin{equation} 
\begin{aligned}
  \partial_t u + \sum_{r=1}^d a_r \partial_{x_r} u & = 0, \quad && \boldsymbol{x} \in \Omega \subset \R^d, \ t>0, \\ 
  u & = g, \quad && \boldsymbol{x} \in \partial \Omega, \ t>0,
\end{aligned}
\end{equation} 
ignoring the IC for the moment.
Yet, the extension to systems is straightforward and achieved by applying the procedure to every component of 
the system.
Combining the method of lines from \S \ref{sub:method-of-lines} with the above discussed RBF interpolants, the idea 
behind (collocation) RBF methods is to define the spatial discretization $\mathbf{u}$ as the values of $u$ at a set of grid points and the operator $\operatorname{L}(\mathbf{u})$ by using the spatial 
derivative of the RBF interpolant $u_N$. 
Let $X = \{ \mathbf{x}_n \}_{n=1}^N$ be the set of grid points in the computational domain $\Omega \subset \R^d$ and let us assume that some of these grid points lie at the boundary $\partial \Omega$. 
In usual RBF methods, the BC is then enforced by replacing the original interpolation conditions 
\eqref{eq:interpol-cond} by 
\begin{equation}\label{eq:interpol+BC}
  u_N(\mathbf{x}_n) = 
  \begin{cases}
    u(\mathbf{x}_n) & \text{if } \mathbf{x}_n \not\in \partial \Omega, \\ 
    g(\mathbf{x}_n) & \text{if } \mathbf{x}_n \in \partial \Omega,
  \end{cases} 
  \quad n=1,\dots,N.
\end{equation}
This is what is commonly referred to as a \emph{strong enforcement of BCs}. 
Finally,  the \emph{usual RBF method} can be summarized in three simple steps: 

\begin{enumerate}
  \item Determine the RBF interpolant $u_N \in V_{N,m}$ satisfying \eqref{eq:interpol+BC}. 
  
  \item Define $\operatorname{L}(\mathbf{u})$ in the semidiscrete equation \eqref{eq:semidis-eq} by  
  \begin{equation}
    \operatorname{L}(\mathbf{u}) 
      = \left( - \sum_{r=1}^d a_r \left( \partial_{x_r} u_N \right) (\mathbf{x}_n) \right)_{n=1}^N.
  \end{equation}
  
  \item Use some time integration method, e.\,g.\ SSPRK(3,3), to evolve \eqref{eq:semidis-eq} in time. 
\end{enumerate}

Even though a strong enforcement of BCs, as described in \eqref{eq:interpol+BC}, yields a pleasantly simple scheme, it 
is known that this approach can result in RBF methods that are unstable in time; 
see \cite{platte2006eigenvalue,glaubitz2019stability}. 
In fact, this can already be observed for a simple scalar linear advection equation in one spatial dimension and is 
demonstrated, for instance, 
in \cite{platte2006eigenvalue,glaubitz2019stability} as well as in \S \ref{sec:num}.
\section{Linear Stable RBF Methods I: Flux Reconstruction} 
\label{sec:RBF_FR} 

In this section, we propose and investigate the first approach to construct RBF methods which are strongly energy stable for linear advection problems. 
Henceforth, we will refer to such RBF methods simply as \emph{linear stable RBF methods}. 

\subsection{Basic Idea}

For sake of simplicity, we start by considering the one-dimensional scalar linear advection equation 
\begin{equation}\label{eq:FR_linear-adv}
\begin{aligned}
	\partial_t u + \partial_x \left( a u \right) & = 0, \quad && x_L < x < x_R, \ t > 0, \\ 
	u(t,x_L) & = g(t), \quad && t > 0, 
\end{aligned}
\end{equation}
on $\Omega = [x_L,x_R]$ with constant velocity $a > 0$ (the case $a<0$ can be treated analogously).  
In the context of (nonlinear) hyperbolic conservation laws, $f(u) = au$ is usually referred to as the \emph{flux function}. 
Let $X = \{x_n\}_{n=1}^N$ be a set of distinct centers in $\Omega$, $m \in \N$, and let $u_N$ and $f_N$ respectively denote the corresponding RBF interpolant of $u$ and $f(u)$ including polynomials of degree less than $m$; see \eqref{eq:RBF-interpol}. 
Note that $X$ does not necessarily include the boundary points $x_L$ and $x_R$. 
The idea behind this first approach to construct linear stable RBF methods is to introduce \emph{correction functions} ${c_L,c_R: \Omega \to \R}$ and to consider the reconstructed flux function 
\begin{equation}\label{eq:cor-flux}
	f_N^{\text{rec}}(x) = f_N(x) + c_L(x) \left[ \fnum_L - f_N(x_L) \right] + c_R(x) \left[ \fnum_R - f_N(x_R) \right].
\end{equation} 
Here, $\fnum_L$ and $\fnum_R$ are the values of a numerical flux,
\begin{equation} 
\begin{aligned}
	\fnum_L & = \fnum\left( g_L(t), u_N(x_L) \right), \\ 
	\fnum_R & = \fnum\left( u_N(x_R), g_R(t) \right), 
\end{aligned}
\end{equation}
which is chosen to be consistent ($\fnum(u,u) = f(u)$), Lipschitz continuous, and monotone. 
Examples of commonly used numerical fluxes can be found in \cite{cockburn1989tvb} and \cite{toro2013riemann}. 
For the linear advection equation \eqref{eq:FR_linear-adv} a usual choice is the \emph{upwind flux} 
\begin{equation}\label{eq:upwind-flux}
	\fnum(u,v) = 
	\begin{cases} 
		a u \quad & \text{ if } a \geq 0, \\ 
		a v  \quad & \text{ if } a < 0.
	\end{cases}
\end{equation}
Substituting the solution $u$ by $u_N$ and the flux function  $au$ by \eqref{eq:cor-flux}, at a fixed time $t$, we end up with the spatial semidiscretization 
\begin{equation}\label{eq:FR-RBF}
  \partial_t u_N = - f_N' - c_L' \left[ \fnum_L - f_N(x_L) \right] - c_R' \left[ \fnum_R - f_N(x_R) \right].
\end{equation}
Henceforth, we call this semidiscretization and the resulting numerical scheme the \emph{flux reconstruction
RBF (FR-RBF) method}. 
It should be noted that the idea behind the FR-RBF method stems from FR schemes \cite{huynh2007flux} which were first introduced by Huynh in 2007 in the context of polynomial FE methods. 
In what follows, we show how certain conditions for the correction functions $c_L,c_R$ yield provable conservative and linear stable RBF methods. 
Furthermore, we will address the construction of these correction functions.

\subsection{Conservation and Linear Stability}

A fundamental property of linear advection equations---and general hyperbolic conservation laws---is conservation. 
That is, the rate of change of the total amount of the conserved variable $u$ should be equal to the flux across the domain boundaries. 
Any numerical method should satisfy this property on a discrete level; see \cite{leveque1992numerical} and references therein. 
For the above proposed FR-RBF method, indeed, this can be ensured by requiring the correction functions $c_L,c_R$ to satisfy 
\begin{equation}\label{cons:conservation}
	c_L(x_R) = c_R(x_L) = 0, \quad c_R(x_R) = c_L(x_L) = 1.
\end{equation} 
This is summarized and proven in the following lemma.   

\begin{lemma}[Conservation of the FR-RBF method]\label{lem:FR-cons}
	Given is the FR-RBF method \eqref{eq:FR-RBF} and correction functions satisfying \eqref{cons:conservation}. 
	Then, the FR-RBF method is conservative; that is,
	\begin{equation}
		\frac{\d}{\d t} \int_\Omega u_N \intd x  = - \left( \fnum_R - \fnum_L \right) 
	\end{equation} 
	holds for all times $t$. 
\end{lemma}

\begin{proof}
	Integrating \eqref{eq:FR-RBF} over $\Omega$, we get 
	\begin{equation}
	\begin{aligned}
  		\frac{\d}{\d t} \int_\Omega u_N \intd x 
    			= & - \int_\Omega f_N' \intd x 
      			- \left[ \fnum_L - f_N(x_L) \right] \int_\Omega c_L' \intd x 
      			- \left[ \fnum_R - f_N(x_R) \right] \int_\Omega c_R' \intd x \\ 
    			= & - \left[ f_N(x_R) - f_N(x_L) \right] 
      			- \left[ \fnum_L - f_N(x_L) \right] \left[ c_L(x_R) - c_L(x_L) \right] \\ 
      		& - \left[ \fnum_R - f_N(x_R) \right] \left[ c_R(x_R) - c_R(x_L) \right].
	\end{aligned}
	\end{equation} 
	Since the correction functions satisfy \eqref{cons:conservation}, this yields 
	\begin{equation} 
	\begin{aligned}
 		\frac{\d}{\d t} \int_\Omega u_N \intd x  
    			& = - \left[ f_N(x_R) - f_N(x_L) \right] 
				+ \left[ \fnum_L - f_N(x_L) \right] 
      			- \left[ \fnum_R - f_N(x_R) \right] \\ 
    			& = \fnum_L - \fnum_{R},
	\end{aligned}
	\end{equation}
	and therefore the assertion. 
\end{proof}

In a similar manner, linear stability can be ensured for the FR-RBF method by a clever choice of the correction functions $c_L,c_R$. 
This time, we need $c_L,c_R$ to satisfy 
\begin{equation}\label{cons:stability}
  	\int_\Omega v \, c_L' \intd x = - v(x_L), \quad 
  	\int_\Omega v \, c_R' \intd x = v(x_R)
\end{equation} 
for all $v \in V_{N,m}$. 
Remember that $V_{N,m}$ denotes the space of all RBF interpolants \eqref{eq:RBF-interpol}, from which also the approximations $u_N$ and $f_N$ are chosen. 

\begin{theorem}[Linear stability of the FR-RBF method]\label{thm:FR-stab}
	Given is the FR-RBF method \eqref{eq:FR-RBF} with correction functions satisfying \eqref{cons:stability} and the upwind flux $\fnum$ as in \eqref{eq:upwind-flux}. 
	Then, the FR-RBF method is strongly energy stable for the linear advection equation \eqref{eq:FR_linear-adv}; that is, \eqref{es:energy_stable} holds. 
\end{theorem}

\begin{proof}
	In case of the linear advection equation \eqref{eq:FR_linear-adv}, we have $f_N = a u_N$. 
	Hence, we can note that for the $L^2$-norm 
	\begin{equation}
	\begin{aligned}
		\frac{\d}{\d t} \norm{u_N(t)}^2 
			= & 2 \int_\Omega u_N \left( \partial_t u_N \right) \intd x \\ 
			= & - 2a \int_\Omega u_N u_N' \intd x 
      			- 2 \left[ \fnum_L - a u_N(x_L) \right] \int_\Omega u_N c_L' \intd x \\
      		& - 2 \left[ \fnum_R - a u_N(x_R) \right] \int_\Omega u_N c_R' \intd x
	\end{aligned}
	\end{equation} 
	holds. 
	Thus, when $c_L$ and $c_R$ satisfy the constraint \eqref{cons:stability}, we get 
	\begin{equation}
	\begin{aligned}
		\frac{\d}{\d t} \norm{u_N(t)}^2 
			= & - a \left[ u_N(x_R)^2 - u_N(x_L)^2 \right] 
				+ 2 \left[ \fnum_L - a u_N(x_L) \right] u_N(x_L) \\
			& - 2 \left[ \fnum_R - a u_N(x_R) \right] u_N(x_R).
	\end{aligned}
	\end{equation} 
	Furthermore, note that $\fnum_L$ and $\fnum_R$ are given by the upwind flux \eqref{eq:upwind-flux} as
	\begin{equation} 
	\begin{aligned}
		\fnum_L & = \fnum\left( g(t), u_N(x_L) \right) = a g(t), \\
		\fnum_R & = \fnum\left( u_N(x_R), u_N(x_R) \right) = a u_N(x_R),
	\end{aligned}
	\end{equation} 
	since the BC only applies to the left boundary and there is no BC at the right boundary. 
	Finally, this yields 
	\begin{equation}\label{eq:stab-FR-proof}
	\begin{aligned}
		\frac{\d}{\d t} \norm{u_N(t)}^2 
			& \leq - a u_N(x_L)^2 + 2 a u_N(x_L) g(t) \\
			& = - a \left[ u_N(x_L) - g(t) \right]^2 + a  g(t)^2 \\ 
			& \leq a g(t)^2
	\end{aligned} 
	\end{equation} 
	and therefore 
	\begin{equation}
	\begin{aligned}
		\norm{u_N(t)}^2 
			& \leq \norm{\left(u_{\text{init}}\right)_N}^2 + a \int_0^t g(s)^2 \intd s \\
			& \leq K(t) \left( \norm{\left(u_{\text{init}}\right)_N}^2 + \max_{s \in [0,t]} |g(s)|^2 \right)
	\end{aligned} 
	\end{equation} 
	with $K(t) = \max\{ 1, a t\}$.
\end{proof}

\subsection{On the Construction of $c_L$ and $c_R$} 

Following Lemma \ref{lem:FR-cons} and Theorem \ref{thm:FR-stab}, the proposed FR-RBF method was proven  to be conservative and linear stable if $c_L, c_R$ satisfy the conditions \eqref{cons:conservation} and \eqref{cons:stability}. 
Here, we show how such correction functions can be constructed. 
Thereby, we only address the construction of $c_L$ and note that $c_R$ can be constructed completely analogously. 
Collecting the conditions \eqref{cons:conservation} and \eqref{cons:stability}, $c_L$ should satisfy 
\begin{align}
  c_L(x_L) & = 1, \quad c_L(x_R) = 0, \label{eq:cond1} \\ 
  \int_\Omega v \, c_L' \intd x & = - \varphi(x_L) \quad \forall v \in V_{N,m}. \label{eq:cond2}
\end{align} 
In \S \ref{sub:RBF-interpol} we have already noted that $V_{N,m}$ is spanned by the nodal basis $\{ \psi_k \}_{k=1}^N$ given in \eqref{eq:nodal_basis}. 
Hence, condition \eqref{eq:cond2} is equivalent to 
\begin{equation}\label{eq:cond2-2}
	\int_\Omega \psi_k \, c_L' \intd x = - \psi_k(x_L), \quad k=1,\dots,N.
\end{equation} 
Together, the conditions \eqref{eq:cond1} and \eqref{eq:cond2-2} therefore impose $N+2$ conditions on the correction function $c_L$. 
It should be stressed that, so far, we have not restricted $c_L$ to belong to any specific function space. 
For example in FR methods, which are the motivation for the FR-RBF method, $c_L$ lies in certain polynomial spaces. 
Regarding the FR-RBF method, however, it seems natural to let $c_L$ be an RBF interpolant.\footnote{Of course, other choices are possible as well. Yet, in this work, we will only consider correction functions which are RBF interpolants themselves.} 
Thus, we follow the approach of $c_L$ being an RBF interpolant using $N+2$ centers in $\Omega$ and including polynomials of degree less than $m$; that is, $c_L \in V_{N+2,m}$. 
Theoretically, the corresponding set of centers $\tilde{X} = \{ \tilde{x}_n \}_{n=0}^{N+1}$ can be chosen completely independent of the original set of centers $X$. 
In our numerical tests we always chose $X$ and $\tilde{X}$ respectively to be the set of $N$ and $N+2$ equidistant points in $[x_L,x_R]$ including the boundary points. 
Next, letting $\{ \tilde{\psi}_j \}_{j=0}^{N+1}$ be a basis of $V_{N+2,m}$, $c_L$ can be represented as 
\begin{equation}\label{eq:rep-c}
	c_L(x) = \sum_{j=0}^{N+1} \gamma_j \tilde{\psi}_j(x). 
\end{equation} 
Then, the conditions \eqref{eq:cond1} and \eqref{eq:cond2-2} become 
\begin{align} 
	\sum_{j=0}^{N+1} \gamma_j \tilde{\psi}_j(x_L) = & 1, \quad 
	\sum_{j=0}^{N+1} \gamma_j \tilde{\psi}_j(x_R) = 0, \\ 
  	\sum_{j=0}^{N+1} \gamma_j \scp{\psi_k}{\tilde{\psi}_j'} = & - \psi_k(x_L), \quad k=1,\dots,N,
\end{align} 
where 
\begin{equation} 
	\scp{\psi_k}{\tilde{\psi}_j'} = \int_\Omega \psi_k \, \tilde{\psi}_j' \intd x.
\end{equation}
Hence, the coefficients $\gamma_j$ can be determined by solving the systems of linear equations 
\begin{equation}\label{eq:LSE}
 	\underbrace{
  	\begin{pmatrix} 
    		\tilde{\psi}_0(x_L) & \dots & \tilde{\psi}_{N+1}(x_L) \\ 
    		\tilde{\psi}_0(x_R) & \dots & \tilde{\psi}_{N+1}(x_R) \\ 
    		\scp{\psi_1}{\tilde{\psi}_0'} & \dots & \scp{\psi_1}{\tilde{\psi}_{N+1}'} \\ 
    		\vdots & & \vdots \\ 
    		\scp{\psi_N}{\tilde{\psi}_0'} & \dots & \scp{\psi_N}{\tilde{\psi}_{N+1}'}
  	\end{pmatrix}
  	}_{=: A}
  	\begin{pmatrix} 
    		\gamma_ 0 \\ \gamma_1 \\ \gamma_2 \\ \vdots \\ \gamma_{N+1}
  	\end{pmatrix}
  	= 
  	\begin{pmatrix} 
    		1 \\ 0 \\ - \psi_1(x_L) \\ \vdots \\ - \psi_N(x_L)
  	\end{pmatrix}.
\end{equation} 
Substituting these coefficients into \eqref{eq:rep-c} we can also recover the correction function $c_L$.

\subsection{Advantages and Pitfalls} 
\label{sub:pitfalls}

An obvious advantage of the FR-RBF method discussed above is that it is provably conservative and linear stable. 
For usual RBF methods this is not the case, which is demonstrated in \S \ref{sec:num}. 
Hence, unlike usual RBF methods, the FR-RBF method is ensured to be stable and produce physically reasonable solutions. 
A disadvantage of this method lies in the construction of the correction functions $c_L$ and $c_R$, however. 
In particular, these can only be recovered from the system of linear equations \eqref{eq:LSE} if the matrix $A$ is invertible (nonsingular). 
Note that $A$ being invertible depends on the kernel $\varphi$, the sets $X$ and $\tilde{X}$, and the polynomial degree $m-1$. 
Unfortunately, we do not see how $A$ being invertible can be ensured. 
Even though we observe $A$ to be invertible in all our numerical tests, we also note that it often has a fairly high condition number. 
This is reported in Table \ref{tab:cond} for the cubic and quintic kernel on an increasing set of equidistant points. 

\begin{table}[htb]
  \renewcommand{\arraystretch}{1.3}
  \centering 
  \begin{adjustbox}{width=.95\textwidth}
  \begin{tabular}{l c c c c c c c c c c}
     & & \multicolumn{4}{c}{cubic kernel $\varphi$} & & \multicolumn{4}{c}{quintic kernel $\varphi$} \\ \hline 
     & & \multicolumn{9}{c}{$N$} \\ \hline
     & & $10$ & $20$ & $40$ & $80$ & & $10$ & $20$ & $40$ & $80$ \\ \hline 
    $A$ regular? & & \cmark & \cmark & \cmark & \cmark  
	         & & \cmark & \cmark & \cmark & \cmark \\
    $\cond(A)$ 	 & & 8.3E+11 & 4.0E+10 & 5.4E+12 & 5.5E+11  
		 & & 3.8E+10 & 2.6E+11 & 3.3E+11 & 2.2E+8 \\ \hline
  \end{tabular}
  \end{adjustbox}
  \caption{Condition numbers of matrix $A$ in \eqref{eq:LSE} for the cubic and quintic kernel.}
  \label{tab:cond}
\end{table}

We can note from Table \ref{tab:cond} that the proposed construction of $c_L$ and $c_R$ can be expected to result in a numerically unstable---or at least only pseudo stable---method. 
This is also demonstrated in the later numerical tests. 
 
\section{Linear Stable RBF Methods II: Simultaneous Approximation Terms} 
\label{sec:RBF_SAT} 

In the previous section, the FR approach was proposed to construct linear stable RBF methods. 
Unfortunately, we also noted some problems for this approach related to numerical instabilities in the construction of the correction functions $c_L,c_R$.  
In this section, we therefore follow a different approach, considering SATs in combination with a RBF methods. 
By now, SATs are an often used tool in the context of FD methods. 
Together with SBP operators they are applied to retrieve energy stable (SBP-SAT) schemes; see \cite{fernandez2014review,svard2014review} and references therein. 
Recently, the SAT approach has also been combined with continuous FE methods; see 
\cite{abgrall2019analysis,abgrall2019analysis_2,hicken2020entropy}. 
Here, we propose to use SATs in combination with RBF methods for determining boundary correction operators that yield provable linear stability. 
To the best of our knowledge, this is the first time SATs are investigated in the context of RBF methods.

\subsection{Basic Idea in One Dimension}

The basic idea of the SAT approach is to weakly implement the BCs in the spatial semidiscretization. 
For a scalar equation in one dimension, this yields the spatial semidiscretization 
\begin{equation}\label{semi_RBF_SAT}
	\partial_t u_N = - f_N' + SAT_L + SAT_R,
\end{equation}
where $SAT_L$ and $SAT_R$ represent the SATs.
The SATs are defined on the inflow part of the boundary through
\begin{equation}\label{eq:SAT}
\begin{aligned}
  SAT_L & = \tau_L a^+ \delta_{x_L} (u_N-g), \\
  SAT_R & =  \tau_R a^- \delta_{x_R} (u_N-g),
\end{aligned}
\end{equation} 
where $a^+=\max\{a,0\}$ and $a^-=\min\{a,0\}$. 
The terms $\tau_L$  and $\tau_R$ denote the boundary operators, while $\delta$ is the Dirac delta. 
That is, the generalized function with the property that 
\begin{equation}
	\int_\Omega v(x) \delta_{x_L}(x) \intd x = v(x_L). 
\end{equation} 
The boundary operators can be seen as some penalty terms and the basic idea of the SAT approach is to weakly impose the BCs. 
Regarding linear stability of \eqref{semi_RBF_SAT}, let us consider the linear advection equation \eqref{eq:FR_linear-adv} with constant velocity $a>0$.  
An inflow condition is set on the left boundary while an outflow is prescribed on the right boundary.
Hence, the only important boundary term in the semidiscretization \eqref{semi_RBF_SAT} is $SAT_L$. 
In this case, requiring $\tau_L < -\frac{1}{2}$ already suffices to ensure linear stability.

\begin{lemma}[Linear stability of the SAT-RBF method]\label{thm:SAT-stab} 
	Given is the SAT-RBF method \eqref{semi_RBF_SAT} with SATs \eqref{eq:SAT} and $\tau_L < -\frac{1}{2}$. 
	Then, the SAT-RBF method is strongly energy stable for the linear advection equation \eqref{eq:FR_linear-adv}; that is, \eqref{es:energy_stable} holds.
\end{lemma}

\begin{proof} 
From \eqref{semi_RBF_SAT}, we get the conservation relation:
\begin{equation}\label{eq:RBF_SAT_STABI}
\begin{aligned}
	\frac{\d}{\d t} \int_\Omega u_N^2 \intd x 
    	= & - 2 a\int_\Omega u_N u_N' \intd x  +2 u_N SAT_L \\
    	= & - a\left[ u_N(x_R)^2 - u_N(x_L)^2 \right] +2 u_N(x_L) \tau_L a^+ (u_N(x_L)-g)\\
    	= & -au_N(x_R)^2 + au_N(x_L)^2  +2 \tau_L a u_N(x_L) (u_N(x_L)-g).
\end{aligned}
\end{equation}
As obviously $au_N(x_R)^2>0$, we can bound the above equation by 
\begin{equation}\label{eq:RBF_SAT_STAB_RB_2}
\begin{aligned}
  	\frac{\d}{\d t} \int_\Omega u_N^2 \intd x 
    	\leq & \,au_N(x_L)^2 +2 \tau_L a u_N(x_L)(u_N(x_L)-g).
\end{aligned}
\end{equation}
Extending this equation under the condition $\tau_L <- \frac{1}{2}$ yields
\begin{equation}\label{eq:RBF_SAT_STAB_RB_3}
\begin{aligned}
  	\frac{\d}{\d t} \int_\Omega u_N^2 \intd x 
    	\leq & \, (1+2\tau_L)au_N(x_L)^2 -  2\tau_La g u_N(x_L) +\frac{a\tau_L^2}{1+2\tau_L} g^2 -\frac{a\tau_L^2}{1+2\tau_L} g^2 \\
    	\leq &\, -\frac{\tau_L^2 a g^2}{1+2\tau_L},
\end{aligned}
\end{equation}
which implies energy stability in accordance to Definition \ref{def:energy_stable}.
\end{proof}

\begin{remark}
The above described SAT approach can be interpreted as the use of a numerical flux which is determined to guarantee the energy estimate. 
Hence, regarding the results obtained in \S \ref{sec:RBF_FR}, the correction terms in the FR approach can be interpreted as SATs, and vice-versa.
\end{remark}

\begin{remark}\label{SBP_Stability}
Here, we only consider stability for continuous $L^2$ norms,  
rather than discrete $L^2$ norms that are commonly considered in the SBP-SAT framework. 
In particular, this requires us to assume exact integration of the volume and the SAT terms. 
In general, considering arbitrary multidimensional domains, this can be a fairly strong numerical constraint. 
In case of linear problems with constant coefficients, the impact of this limitation is damped by the use of exact quadrature rules in the discretization, and by the existence of a SBP-like condition when using RBFs. 
However, this also opens the possibility for future research by determining point distributions 
and RBFs that satisfy an SBP property (at least up to machine precision).
One can find a first approach in this direction in \cite{platte2006eigenvalue}, even though the authors do not explicitly draw the connection to SBP operators. 
It might be argued that this demonstrates a missing communication between the different communities (FE/FD methods and RBF methods). 
This work can therefore be seen as an attempt to build a bridge between these different perspectives on numerical methods for PDEs. 
Furthermore, future works could address the extension of the proposed methods to nonlinear problems. 
For these, for instance, entropy correction as suggested in \cite{abgrall2018general,abgrall2018connection,abgrall2019reinterpretation,abgrall2019analysis_2}
might be of interest, cf.\ section \ref{sec:extensions}.
\end{remark}

\subsection{Extension to Two Dimensional Problems}
\label{sub:SAT-2d}

Next, we aim to extend SAT-RBF methods to multidimensional problems. 
As the extension to multidimensional problems can be done in the same way, for sake of simplicity, we only focus on two-dimensional model problems.
Hence, let us consider the linear problem  
\begin{equation}\label{eq:model_problem_2d}
\begin{aligned}
	\partial_t u + a_1 \partial_{x} u + a_2 \partial_y u &= 0, \quad 
		&& \bm{x} \in \Omega, \ t > 0, \\ 
  	u & = u_{\text{init}}, \quad 
		&& \bm{x} \in \Omega, \ t = 0, \\ 
  	u & = g , \quad 
		&& \bm{x} \in \partial \Omega, \ t > 0,
\end{aligned}
\end{equation}
with $\Omega \subset \R^2$. 
By $\partial \Omega^-$ and $\partial \Omega^+$ we respectively denote the part of the boundary where inflow and outflow conditions are applied. 
Moreover, we assume that the imposed BCs yield to the well-posedeness of \eqref{eq:model_problem_2d}; see \cite{nordstrom2017roadmap}.
Discretizing \eqref{eq:model_problem_2d} by the RBF-SATs approach yields
\begin{equation}\label{eq:semi_RBF_SAT_2d}
	\partial_t u_N = -a_1 \partial_x u_N -a_2 \partial_y u_N + SAT^{\text{in}}.
\end{equation}
Here, $SAT^{\text{in}}$ represents the inflow BC. 
That is,
\begin{equation}
SAT^{in} =\Pi \delta_{\partial \Omega^-}(u_N-g)
\end{equation}
with $\Pi = \Pi(\bm{x}) \in \R^2$ being the boundary operator that remains to be characterized such that linear stability is guaranteed. 
Moreover, $\delta_{\partial \Omega^-}$ denotes the generalized function with the property that 
\begin{equation}
	\int_\Omega \Pi \delta_{\partial \Omega^-} v \intd \bm{x} = 
		\int_{\partial \Omega^-} \left( \Pi \cdot \bm{n} \right) v \intd s, 
\end{equation} 
where $\bm{n}$ represents the outward normal to the boundary.
In order to derive a stability estimate, we use the Gauss theorem to obtain 
\begin{equation}\label{eq:RBF_SAT_STAB_RB_2_two_d}
\begin{aligned}
  \frac{\d}{\d t} \int_\Omega u_N^2 \intd \bm{x} 
    	& = -  \int_{\partial \Omega^-} \left[ u_N^2 \left( (a_1,a_2)^T \cdot \bm{n} \right) - 2u_N \left( \Pi \cdot \bm{n} \right) (u_N-g) \right] \intd s.
  \end{aligned}
\end{equation}
Thus, linear stability is guaranteed by the following admissibility condition on $\Pi$:
\begin{equation}\label{eq:condition_2}
	u_N^2 \left( (a_1,a_2)^T \cdot \bm{n} \right) - 2u_N \left( \Pi \cdot \bm{n} \right) (u_N-g) \geq 0
\end{equation} 

This expression highly depends on the domain and the BCs themselves. 
Assuming no incoming wave ($g|_{\partial \Omega^-}=0$) condition \eqref{eq:condition_2} reduces to $2 \Pi\cdot \bm{n}\leq (a_1,a_2)^T \cdot \bm{n}$. 
We summarize this in the following lemma. 

\begin{lemma}[Linear stability of the SAT-RBF method in two dimensions]\label{thm:SAT-stab_two}
	Given is the SAT-RBF method \eqref{eq:semi_RBF_SAT_2d} with boundary operator $\Pi$
	satisfying \eqref{eq:condition_2}. 
	Then, the SAT-RBF method is strongly energy stable for the two dimensional 
advection equation \eqref{eq:model_problem_2d}.
\end{lemma}

\subsection{Extension to Linear Systems in One Space Dimension}

Finally, we demonstrate how the SAT-RBF approach can be extended to systems of linear advection equations. 
Let us consider the following one-dimensional model problem: 
\begin{equation}\label{eq:1}
\aligned
 \partial_t \bm{u} +A \partial_x \bm{u} & = 0,\quad &&  x\in (0,1), \ t> 0, \\
 L_0(\bm{u}) & = g_0(t), \quad &&  x=0, \ t > 0, \\
 L_1(\bm{u}) & = g_1(t), \quad  &&  x=1, \ t > 0, \\
 \bm{u}(x,0) & = \bm{u}_{\text{init}}, \quad && x \in [0,1], \ t=0
\endaligned
\end{equation} 
Here, $\bm{u} = (u_1,\dots,u_m)^T$ is the vector of conserved variables and $A \in \R^{m \times m}$ is a symmetric matrix. 
Furthermore, $L_0$ and $L_1$ are assumed to be linear operators. 
We denote by $n_0$ the number of incoming characteristics at $x=0$ and by $n_1$ the number of incoming characteristics at $x=1$. 
Thus, the rank of $L_0$ and $L_1$ respectively is $n_0$ and $n_1$.
We start by noting that multiplying the continuous differential equation in \eqref{eq:1} from the left hand side by $\bm{u}^T$, we get 
\begin{equation}
  \int_0^1 \bm{u}^T  \partial_t  \bm{u} \intd x = - \int_0^1 \bm{u}^T A \partial_x \bm{u} \intd x.
\end{equation} 
As the energy is given by $E=\tfrac{1}{2}\int_0^1 \bm{u}^2 \intd x$, it satisfies through the use of the Gauss theorem 
\begin{equation}\label{eq:energy_Extension}
  \frac{\d}{\d t} E+\bm{u}(1,t)^T A \bm{u}(1,t)- \bm{u}(0,t)^T A \bm{u}(0,t)= 0.
\end{equation}
We then apply the SAT term corresponding to the incoming characteristics. 
Imposing the BC weakly and using the RBF formulation for the problem \eqref{eq:1} in \eqref{eq:energy_Extension}, we obtain
\begin{equation}\label{eq:three} 
\begin{aligned}
  	\frac{\d}{\d t} E_N 
		= \ & \bm{u}_N(1,t)^T \left(-A \bm{u}_N(1,t)+ \Pi_1(L_1(\bm{u}_N)-g_1)\right) \\ 
  		& + \bm{u}_N(0,t)^T\left(A\bm{u}_N(0,t)-\Pi_0(L_0(\bm{u}_N-g_0)) \right), 
\end{aligned}
\end{equation}
where $\Pi_0$ and $\Pi_0$ are the boundary operators. 
These have to be chosen such that the change of energy yields an estimate as described in Definition \ref{def:energy_stable}. 
Moreover, for any $t$, the image of the boundary operator $L_0$ and $L_1$ respectively has to be the same as the image of $\Pi_0L_0$ and $\Pi_1L_1$.
This ensures that there is no loss of information at the boundary.

This can be solved by decomposing $A$ in its diagonal form and rewriting the boundary procedure by using characteristic variables instead of $\bm{u}$. 
However, to specify precisely $\Pi_0$ and $\Pi_1$, one has to consider the problem itself, e.\,g.\ the wave equation. 
More details can be found in \cite{abgrall2019analysis}. 
If we set the BC to zero at the inflow part and suppose that $L_0$ and $L_1$ are simply given by the identity matrix, we can derive from \eqref{eq:three} the following result.

\begin{lemma}[Linear stability of the SAT-RBF method for system]\label{thm:SAT-stab_system}
	Assume that the boundary operator $\Pi$ can be selected such that the matrix
$\Pi-A+(\Pi-A)^T$ is negative semi-definite. 
	Then, the SAT-RBF method is strongly energy stable for the linear symmetric system  \eqref{eq:1}.
\end{lemma}
 
\section{Potential Extensions}
\label{sec:extensions} 

This section briefly addresses some possible extensions of the proposed approach to construct linearly energy stable RBF methods. 
These extensions include local RBF-FD methods (\S \ref{sub:local-RBF-methods}), (entropy) stability for nonlinear problems (\S \ref{sub:nonlinear-problems}), and numerical integration (\S \ref{sub:num-int}).

\subsection{Local Radial Basis Function Methods} 
\label{sub:local-RBF-methods}

Thus far, we have only considered global RBF methods. 
Yet, an obvious concern with these are their computational costs. 
In fact, finding a global RBF interpolant or calculating a differentiation matrix each cost $\mathcal{O}(N^3)$ operations for $N$ nodes. 
While in the discussed methods this can be done a priori once---before time stepping\footnote{Assuming the nodes do not change over time.}---there are additional $\mathcal{O}(N^2)$ operations each time a differentiation matrix is applied---during time stepping. 
Local RBF-FD are considered as one of the leading options to remedy this problem.\footnote{The conference presentation \cite{tolstykh2000using} by Tolstykh in 2000 seems to be the earliest reference to RBF-FD methods.} 
Conceptually, these methods can be interpreted as an extreme case of overlapping domain decomposition, with a separate domain surrounding each node. 
The basic idea is to center a local RBF-FD stencil at each of the $N$ global nodes, and let it include the $n-1$ nearest neighbors, where $n \ll N$. 
For every node, and based on its surrounding stencil, a local FD formula that is exact for all RBF interpolants on that stencil---potentially including polynomilas---is derived then from a system of linear equations similar to \eqref{eq:LS-RBF-interpol}. 
The main difference is that the right hand side is replaced by the nodal values of a linear differentiation operator.
For more details, see \cite[Chapter 5]{fornberg2015primer} and references therein. 
Unfortunately, the framework discussed in the present manuscript is not immediately applicable to local RBF-FD methods. 
Yet, we think that it is mediately transferable by replacing exact integrals and differentiation operators by their discrete counterparts as long as these satisfy certain SBP properties. 
In this case, many stability properties which are based on integration by parts (the continuous analogue of SBP) would still be satisfied in a discrete norm.
We intend to address these extensions in future works.

\subsection{Variable Coefficients and Nonlinear Problems} 
\label{sub:nonlinear-problems}

Up to this points, only linear problems with constant coefficients  have been considered.
In what follows, we provide some comments regarding linear problems with variable coefficients as well as nonlinear problems, namely hyperbolic conservation laws.
% Variable coefficients 
Let us consider a linear advection equation
\begin{equation}\label{eq:Model_problem2}
\begin{aligned}
   	\partial_t u(t,x)+\partial_x \left( a(x) u(t,x) \right) & = 0, \quad 
   		&& x \in (x_L, x_R), \ t>0, \\
   	u(t,x_L) & = g_L(t), \quad 
   		&& t \geq 0, \\
   	u(0,x) & = u_{\text{init}}(x), \quad 
		&& x \in [x_L,x_R]
\end{aligned}
\end{equation}
with variable velocity $a(x)>0$ as well as compatible IC $u_{\text{init}}$ and BC $g_L$.
As described \emph{inter alia} in \cite{nordstrom2006conservative,offner2019error},
the energy is bounded for a fixed time interval if $\partial_x a$ is bounded. 
In fact, one has 
\begin{equation}
\begin{aligned}
 	\|u(t)\|^2 \leq & \exp \left(t \|\partial_x a\|_{L^\infty} \right) \cdot \\ 
	& \left( \|u_{\text{init}}\|^2 +
 \int_0^t \exp \left( -\tau \| \partial_x a \|_{L^\infty} \right)
  \left[ a(x_L) g_L(\tau)^2 - a(x_R)u(\tau,x_R)^2 \right]  \intd \tau \right);
\end{aligned}
\end{equation}
see \cite[Section 2]{nordstrom2017conservation}. 
To obtain a similar bound for the semidiscretization of \eqref{eq:Model_problem2} by a numerical method, often a skew-symmetric formulation has to be used. 
This is, for instance, a well-established technique for SBP-based methods discussed in Remark 
 \ref{SBP_Stability}. 
 Assuming an SBP property is holding for the RBF method, the transfer of stability results would be straightforward.
In the one-dimensional setting, a suitable skew-symmetric formulation of \eqref{eq:Model_problem2} might be
\begin{equation}
\partial_x(a(x)u(t,x))= \alpha \partial_x(a(x)u(t,x))+ (1-\alpha) \left( u(t,x) \partial_x a(x) +a(x) \partial_x u(t,x) \right)  
\end{equation}
with $\alpha=0.5$. 
A more detailed investigation of stability for variable coefficient problems and the usage of skew-symmetric formulations in RBF methods will be provided in future works. 
% Nonlinear problems 
There also nonlinear problems of the form 
\begin{equation}
\begin{aligned}
   \partial_t u(t,x)+\div f(u(t,x))&=0, &t>0,\; x\in \Omega \subset \R^n,\\
   u(0,x)&=u_0(x), & x\in  \Omega, \\
   L(u(t,x))&= G(t,x)  & t>0,  x\in  \partial \Omega
 \end{aligned}
\end{equation} 
will be considered. 
Here, $L$ and $G$ are boundary operators which are assumed to yield a well-posed problem; see \cite{svard2014review, nordstrom2017roadmap} and references therein.  
In the context of nonlinear problems it is often argued---especially in the computational fluid dynamics community---that entropy stability, rather than energy stability, is an appropriate stability concept. 
See \cite{hesthaven2019entropy,chen2017entropy,abgrall2018general,abgrall2019reinterpretation, fjordholm2012arbitrarily} and references therein.
Unfortunately, entropy stability alone is not necessarily implying convergence of a numerical scheme. 
Yet, experience has demonstrated that entropy stable methods are often more reliable and robust than comparable methods which violate entropy conditions
A crucial tool to ensure entropy stability are numerical (two-point) fluxes developed by Tadmor 
\cite{tadmor2003entropy} together with a proper splitting of volume terms \cite{chen2017entropy}. 
Another possible approach might be entropy correction as proposed in \cite{abgrall2018general,abgrall2019reinterpretation}. 
Of course, also the imposed BCs themselves play an essential role \cite{svard2020entropy}.
An investigation of these techniques in the context of RBF methods will be part of future research as well.

\subsection{Numerical Integration}
\label{sub:num-int}

When discretizing the SAT terms discussed in \S \ref{sec:RBF_SAT}, but also for numerically determining suitable correction functions in \S \ref{sec:RBF_FR}, we are in need of computing certain integrals. 
Depending on the number of degrees of freedom and the dimension, this can be computationally costly. 
Hence, future works might also address the advantage---and potential pitfalls---of replacing continuous integrals in the RBF method by a discrete quadrature (in one dimension) or cubature (in higher dimensions) formula. 
In one dimension, reasonable candidates are the trapezoidal or Gauss--Legendre/Lobatto formula. 
In more then one dimension, and assuming a rectangular domain, their tensor products of these rules could be used; see some of the excellent monographs \cite{haber1970numerical,stroud1971approximate,engels1980numerical,davis2007methods,trefethen2017cubature}. 
Yet, these would not be available in the notoriously difficult case of non-standard, in particular non-rectangular, domains. 
Here, an alternative might be given by classical (quasi-)Monte Carlo methods \cite{metropolis1949monte,niederreiter1992random,caflisch1998monte,dick2013high} or the recently developed high-order least squares cubature formulas \cite{glaubitz2020stableCF,glaubitz2020constructing}, based on the one dimensional works \cite{wilson1970necessary,wilson1970discrete,huybrechs2009stable,glaubitz2020stable}.
\section{Numerical Results} 
\label{sec:num}

In what follows, we numerically investigate the two proposed approaches to construct linear stable RBF methods and compare them to the usual RBF method. 
All tests are performed for polyharmonic splines (PHS) using a cubic ($k=2$) and a quintic ($k=3$) kernel. 
Moreover, polynomials of degree less than $m = 2$ and $m=3$ have been included, respectively. 
This ensures that the resulting RBF interpolant is always well-defined; see \cite{iske2003radial} or \cite[Chapter 3.1.5]{glaubitz2020shock}. 
While the proposed RBF methods could also be applied to other kernels, such as Gaussian or compactly supported ones (Wendland functions), PHS allow us to circumvent the (sometimes difficult) choice of the shape parameter $\varepsilon$. 
Of course, a discussion and comparison of different kernels would be highly interesting. 
Yet, it would also exceed the scope of this work. 
The same holds for the distribution of the centers. 
It should be noted that the stability results discussed in \S \ref{sec:RBF_FR} and \S \ref{sec:RBF_SAT} hold independently of the kernel $\varphi$ and centers $X$ (and $\tilde{X}$). 
Hence, for sake of simplicity, we only consider equidistant centers in the subsequent numerical tests.

\subsection{One-Dimensional Scalar Problems} 

Let us start by considering the scalar one-dimensional advection problem 
\begin{equation}\label{eq:test_linear-adv}
\begin{aligned}
	\partial_t u + \partial_x \left( a u \right) & = 0, \quad && 0 < x < 1, \ t > 0, \\ 
	u(t,x_L) & = g(t), \quad && t > 0, \\
	u(0,x) & = u_{\text{init}}(x), \quad && 0 \leq x \leq 1, 
\end{aligned}
\end{equation}
with constant velocity $a=1$. For the SAT-RBF method we used $\tau = -1$ unless otherwise stated.

\subsubsection{Inflow Boundary Condition} 

Given are the BC and IC 
\begin{equation}\label{eq:BC_inflow_bump} 
	g(t)=u_{\mathrm{init}}(0.5-t), \quad 
  	u_{\mathrm{init}}(x) = 
  	\begin{cases}
    		e^{16} \exp \left( \frac{-16}{1-(4x-1)^2} \right) & \text{if } 0 < x < 0.5, \\ 
    		0 & \text{otherwise}. 
  	\end{cases}
\end{equation}
That is, we have a smooth IC and an inflow BC at the left boundary $x=0$. 
The results for the usual RBF, FR-RBF and SAT-RBF method for $N=40$ grid points at time $t=0.5$ are displayed in Figure \ref{fig:1d_inflow_bump}. 

\begin{figure}[!htb]
  \centering
  \begin{subfigure}[b]{0.4\textwidth}
    \includegraphics[width=\textwidth]{%
      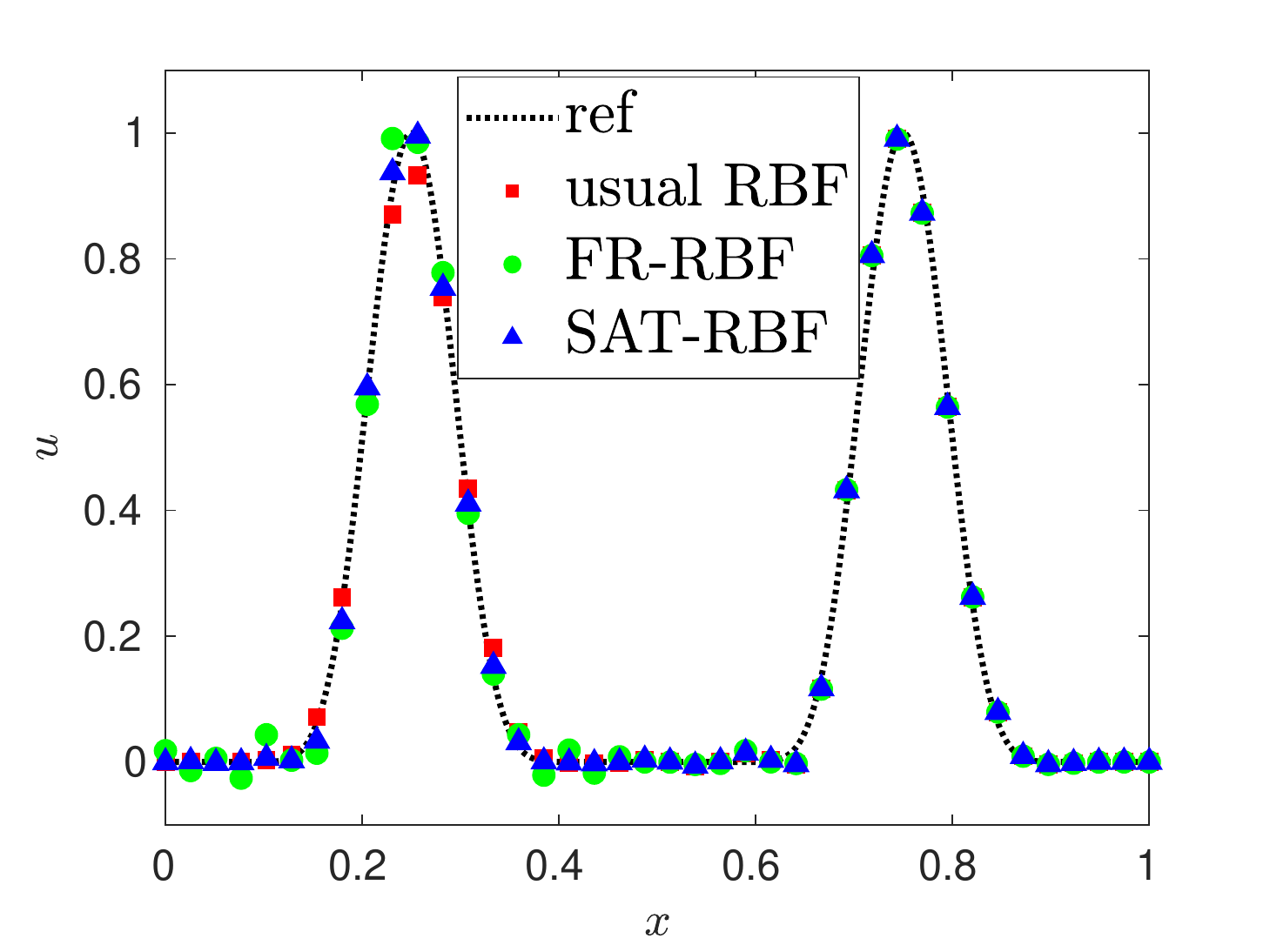}
    \caption{Cubic kernel}
    \label{fig:1d_inflow2_bump_sol_cubic} 
  \end{subfigure}%
  ~ 
  \begin{subfigure}[b]{0.4\textwidth}
    \includegraphics[width=\textwidth]{%
      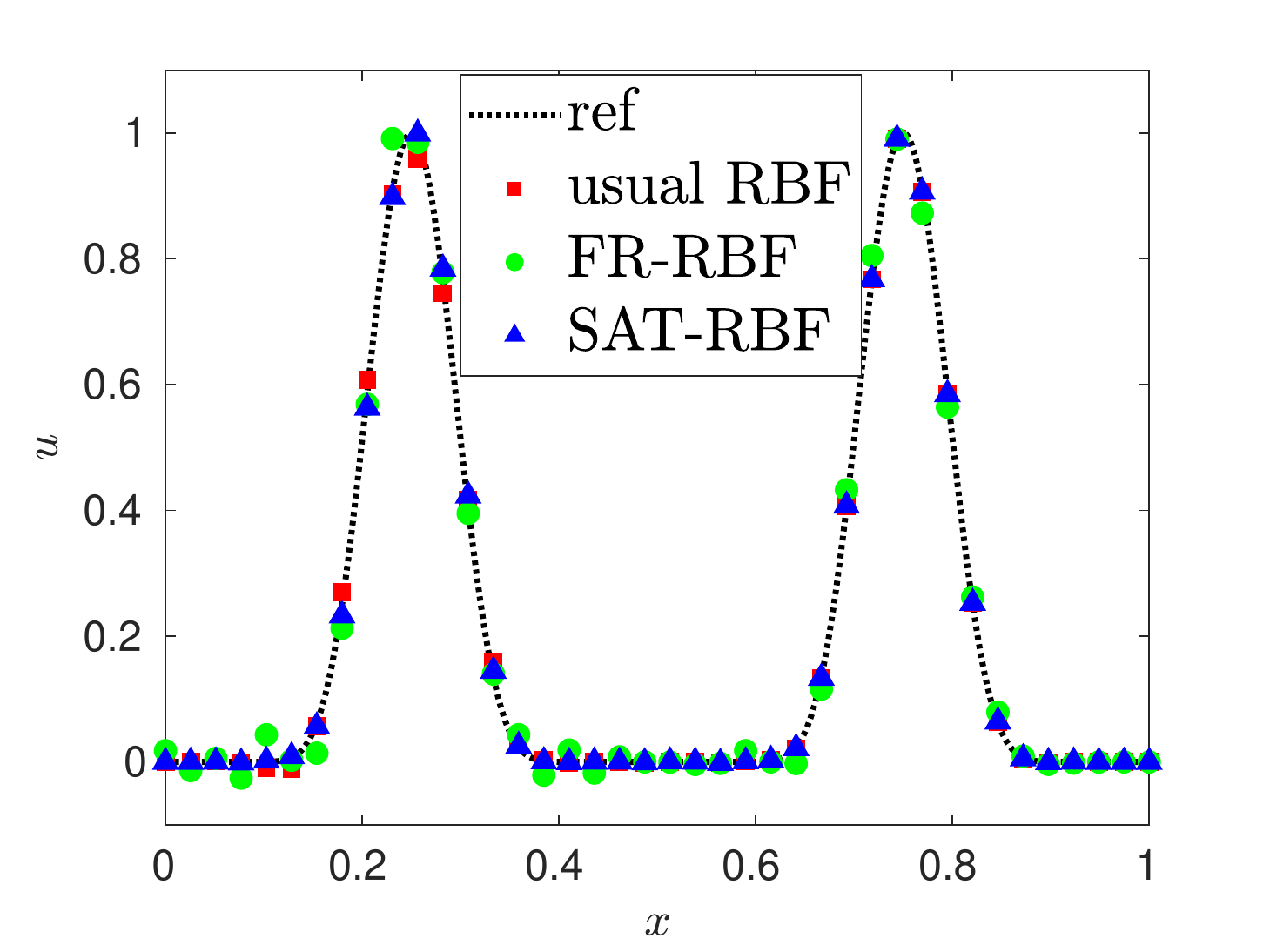}
    \caption{Quintic kernel}
    \label{fig:1d_inflow2_bump_sol_quintic} 
  \end{subfigure}%
  \caption{Numerical solutions of \eqref{eq:FR_linear-adv} with IC and inflow BCs as in \eqref{eq:BC_inflow_bump} for $N=40$ at $t=0.5$}
  \label{fig:1d_inflow_bump}
\end{figure}

From Figure \ref{fig:1d_inflow_bump}, we can already note the numerical stability issues for the FR-RBF method which were discussed in \S \ref{sub:pitfalls}. 
These stem from high condition numbers of the matrix $A$ in the linear system \eqref{eq:LSE} that has to be solved to recover the correction functions $c_L$ and $c_R$. 
In fact, this lack of numerical stability for the FR-RBF method can also be noted from increased errors in Table \ref{tab:1d_inflow}.

\begin{table}[htb]
  \renewcommand{\arraystretch}{1.3}
  \centering 
  \begin{adjustbox}{width=0.95\textwidth}
  \begin{tabular}{l l c c c c c c c c c c c c }
    \multicolumn{2}{c}{} 
      & \multicolumn{6}{c}{$\ell^1$-errors} 
      & \multicolumn{6}{c}{$\ell^\infty$-errors} \\ \hline 
    \multicolumn{2}{c}{} 
      & & \multicolumn{4}{c}{$N$} & $\mathcal{O}$  
      & & \multicolumn{4}{c}{$N$} & $\mathcal{O}$ \\ \hline
    kernel & Method & & $10$ & $20$ & $40$ & $80$ & 
	            & & $10$ & $20$ & $40$ & $80$ & \\ \hline 
    cubic & usual   & & 2.0E-1 & 1.1E-1 & 1.2E-2 & 2.1E-3 & 2.1 
		    & & 4.9E-1 & 3.2E-1 & 5.6E-2 & 1.5E-2 & 1.6 \\
	  & FR      & & 1.7E-1 & 1.5E-1 & 1.6E-2 & 2.0E-3 & 2.1 
		    & & 4.4E-1 & 6.9E-1 & 4.8E-2 & 1.3E-2 & 1.6 \\
	  & SAT     & & 1.5E-1 & 1.0E-1 & 9.8E-3 & 1.5E-3 & 2.2 
		    & & 4.3E-1 & 2.4E-1 & 3.6E-2 & 8.9E-3 & 1.8 \\ \hline
    quintic & usual & & 2.7E-1 & 6.5E-2 & 5.2E-3 & 9.7E-4 & 2.7 
		    & & 8.5E-1 & 2.1E-1 & 3.0E-2 & 4.3E-3 & 2.5 \\
	  & FR      & & 1.9E-1 & 7.4E+1 & 1.8E-2 & 3.0E-3 & 1.9 
		    & & 3.5E-1 & 1.8E+2 & 8.9E-2 & 1.5E-2 & 1.5 \\
	  & SAT     & & 2.0E-1 & 4.4E-2 & 3.6E-3 & 1.2E-3 & 2.4 
		    & & 3.5E-1 & 1.2E-1 & 2.3E-2 & 9.7E-3 & 1.7 \\ \hline
  \end{tabular}  
  \end{adjustbox}
  \caption{Errors for \eqref{eq:FR_linear-adv} with IC and inflow BCs as in \eqref{eq:BC_inflow_bump} at $t=0.5$}
  \label{tab:1d_inflow}
\end{table}

Table \ref{tab:1d_inflow} lists the $\ell^1$ and $\ell^\infty$ errors, 
\begin{equation} 
\begin{aligned}
	\norm{ u - u_N }_{\ell^1} & = \frac{1}{N} \sum_{n=1}^N |u(x_n) - u_N(x_n)|, \\
	\norm{ u - u_N }_{\ell^\infty} & = \max_{n=1,\dots,N} |u(x_n) - u_N(x_n)|, 
\end{aligned}
\end{equation}
and the resulting \emph{average orders}, $\mathcal{O} = \frac{1}{3} \sum_{j=1}^3 o_j$, where 
\begin{equation}
  o_j = \log_{2}\left( \frac{||u-u_{N_{j+1}}||}{||u-u_{N_j}||} \right)
\end{equation}
with $N_j = 10 \cdot 2^j$.

\subsubsection{Long Time Simulation with Periodic BCs} 
\label{sub:long-time}

Next, let us consider \eqref{eq:test_linear-adv} with periodic BCs and a smooth IC given by   
\begin{equation}\label{eq:BC_periodic} 
	u(t,0) = u(t,1), \quad 
  	u_{\mathrm{init}}(x) = \sin^2( 2 \pi x ).
\end{equation}
The results for $N=40$ for increasing end times are displayed in Figure \ref{fig:1d_periodic_sin^2}. 
Moreover, Figure \ref{fig:1d_periodic_sin^2_energy} illustrates the corresponding energy profiles of the different methods over time. 
Corresponding errors at time $t=100$ are presented in Table \ref{tab:1d_periodic_sin^2}.

\begin{figure}[!htb]
	\centering
	\begin{subfigure}[b]{0.32\textwidth}
		\includegraphics[width=\textwidth]{%
			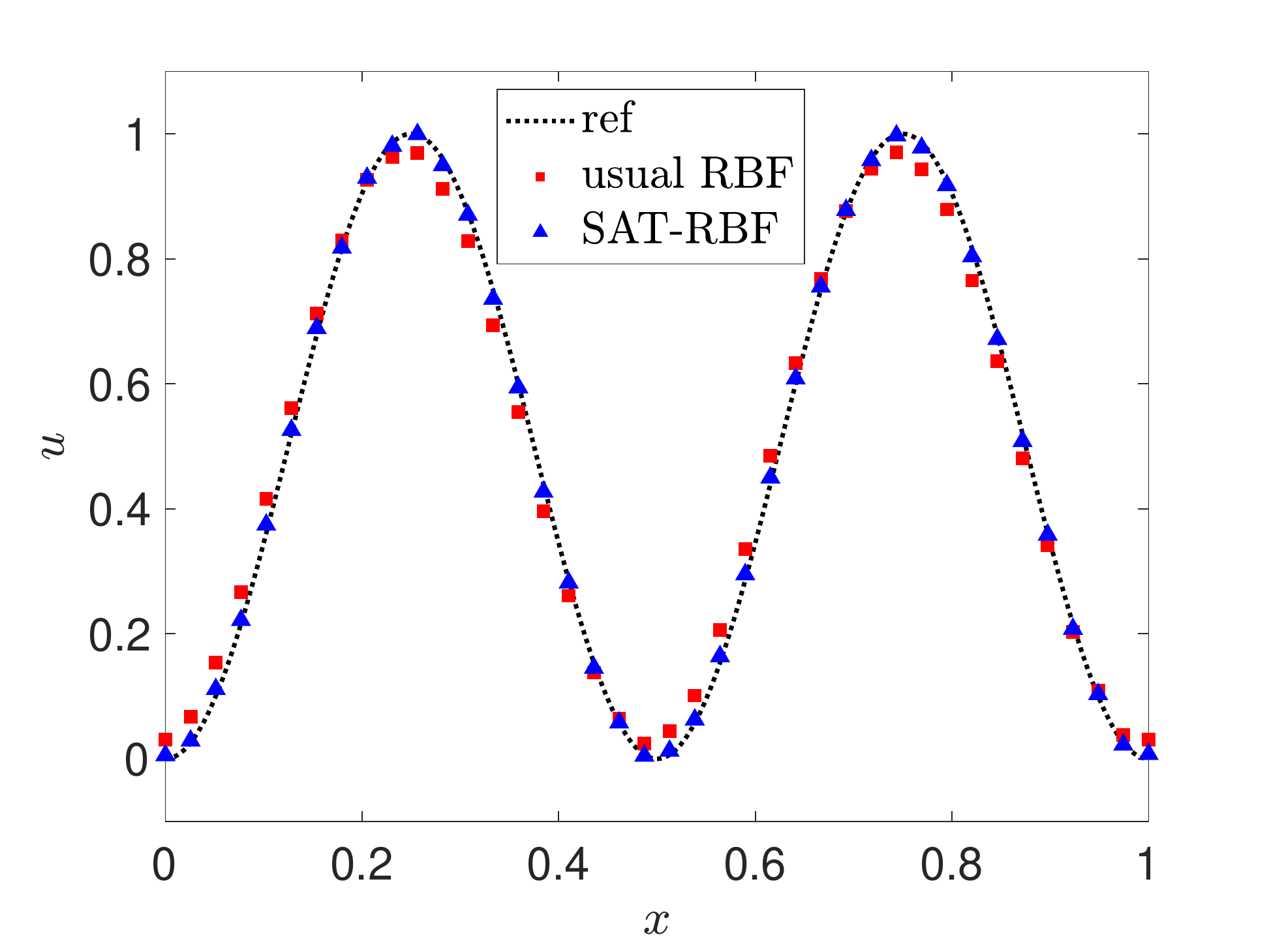}
		\caption{Cubic kernel, $t=4$}
		\label{fig:1d_periodic_sin^2_cubic_t4} 
	\end{subfigure}%
	~ 
	\begin{subfigure}[b]{0.32\textwidth}
		\includegraphics[width=\textwidth]{%
			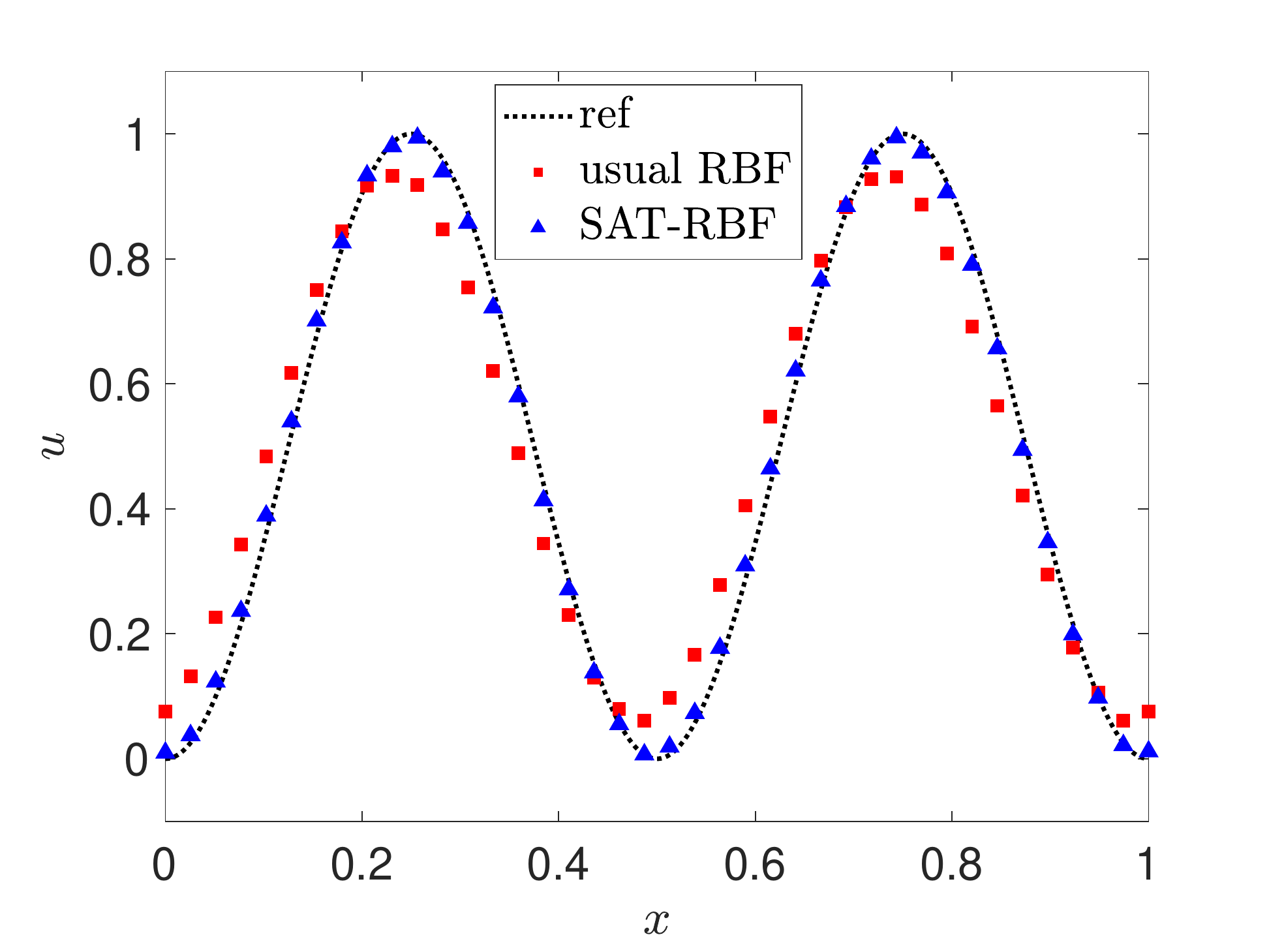}
		\caption{Cubic kernel, $t=10$}
		\label{fig:1d_periodic_sin^2_cubic_t10}
	\end{subfigure}% 
	~ 
	\begin{subfigure}[b]{0.32\textwidth}
		\includegraphics[width=\textwidth]{%
			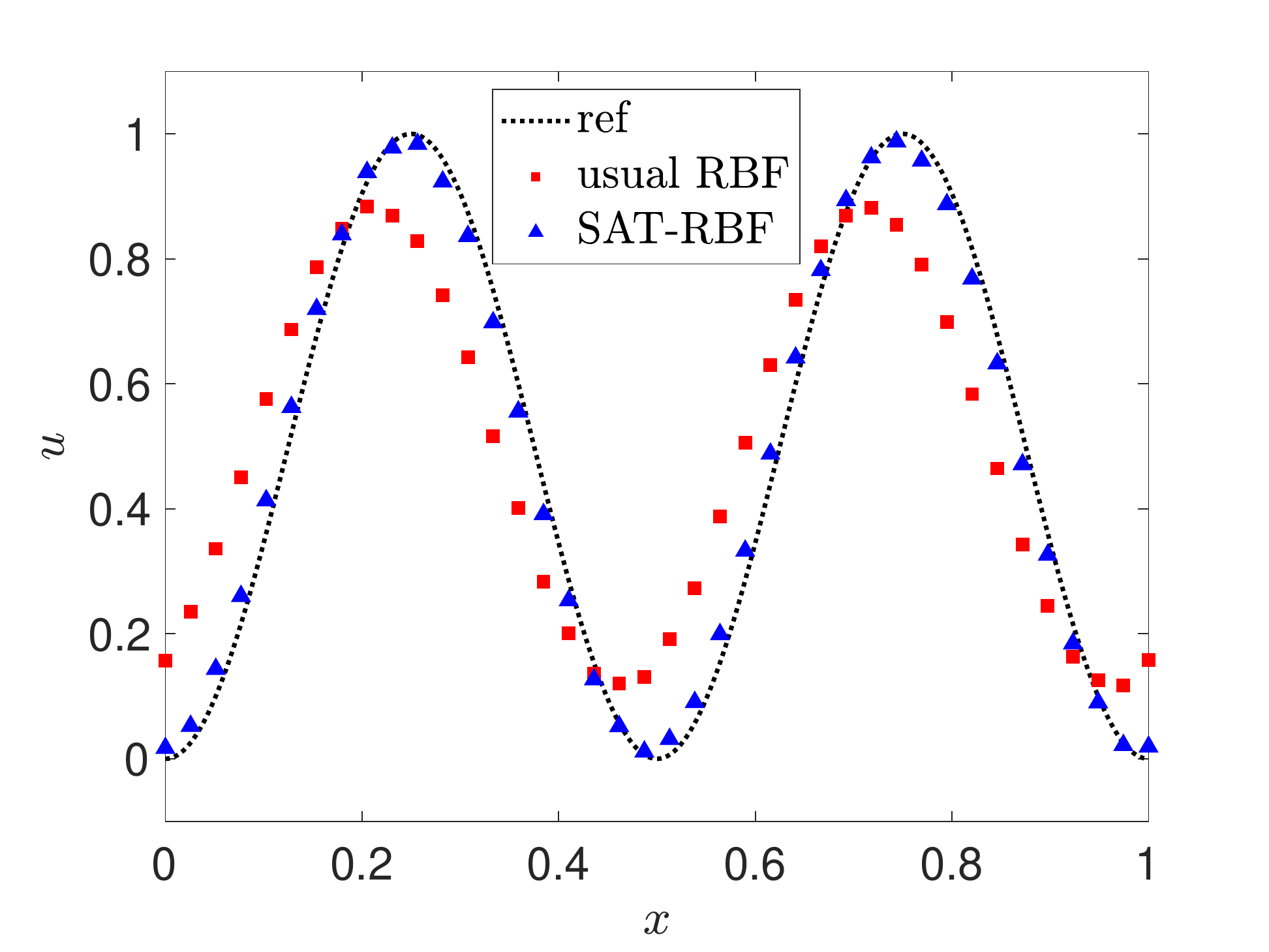}
		\caption{Cubic kernel, $t=20$}
		\label{fig:1d_periodic_sin^2_cubic_t20}
	\end{subfigure}% 
	\\  
	\begin{subfigure}[b]{0.32\textwidth}
		\includegraphics[width=\textwidth]{%
			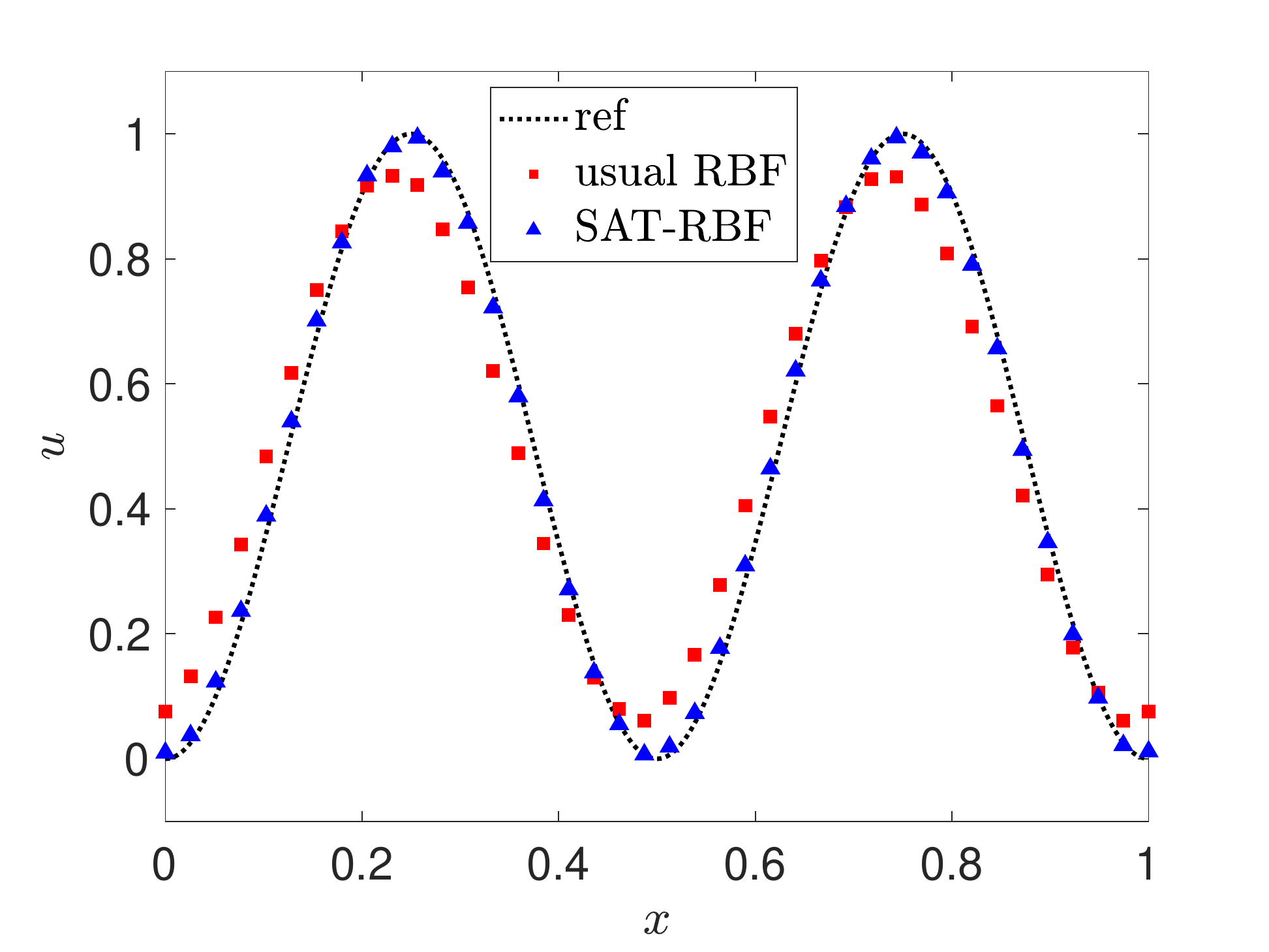}
		\caption{Quintic kernel, $t=10$}
		\label{fig:1d_periodic_sin^2_quintic_t10} 
	\end{subfigure}%
	~ 
	\begin{subfigure}[b]{0.32\textwidth}
		\includegraphics[width=\textwidth]{%
			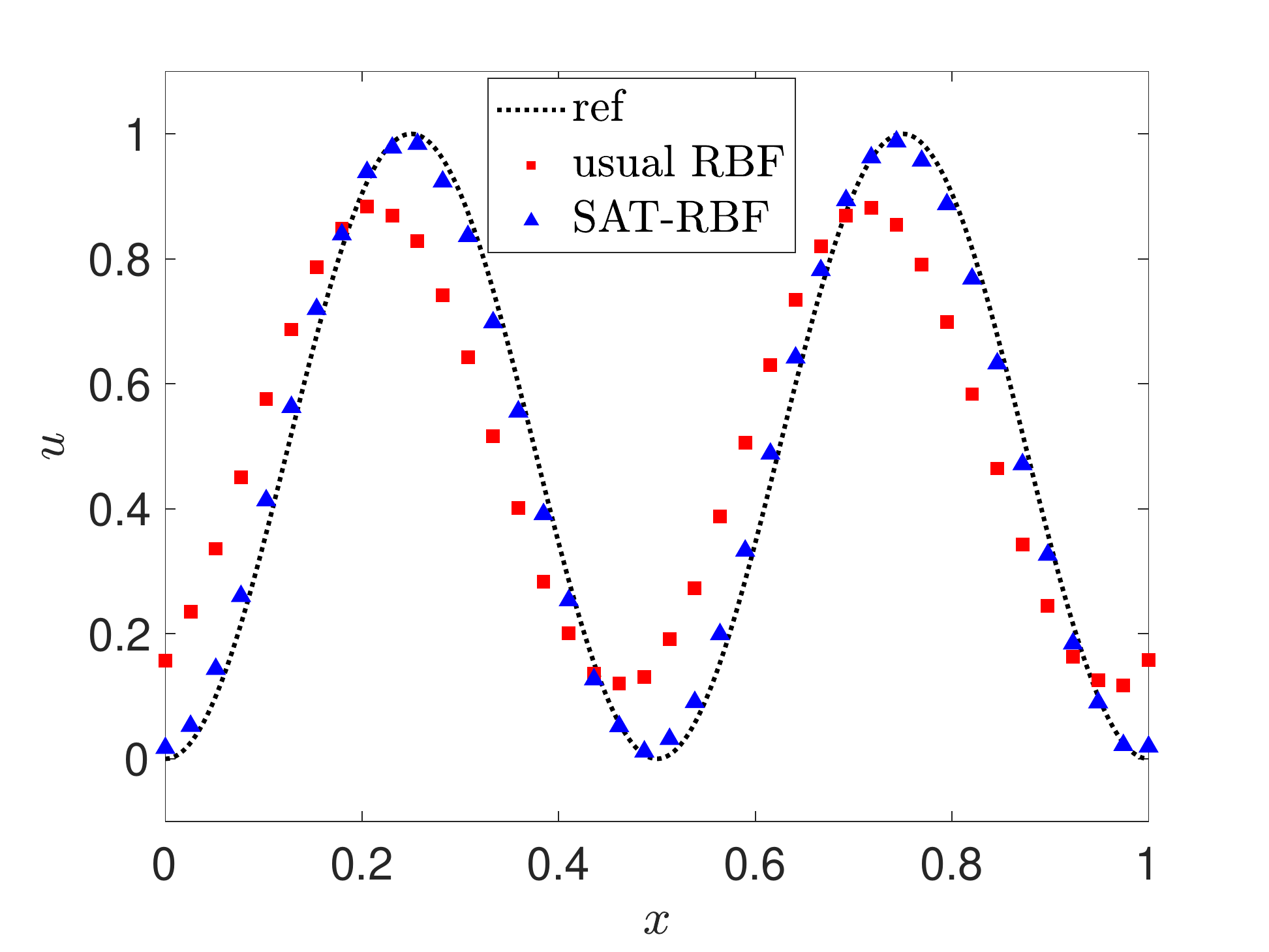}
		\caption{Quintic kernel, $t=20$}
		\label{fig:1d_periodic_sin^2_quintic_t20}
	\end{subfigure}% 
	~ 
	\begin{subfigure}[b]{0.32\textwidth}
		\includegraphics[width=\textwidth]{%
			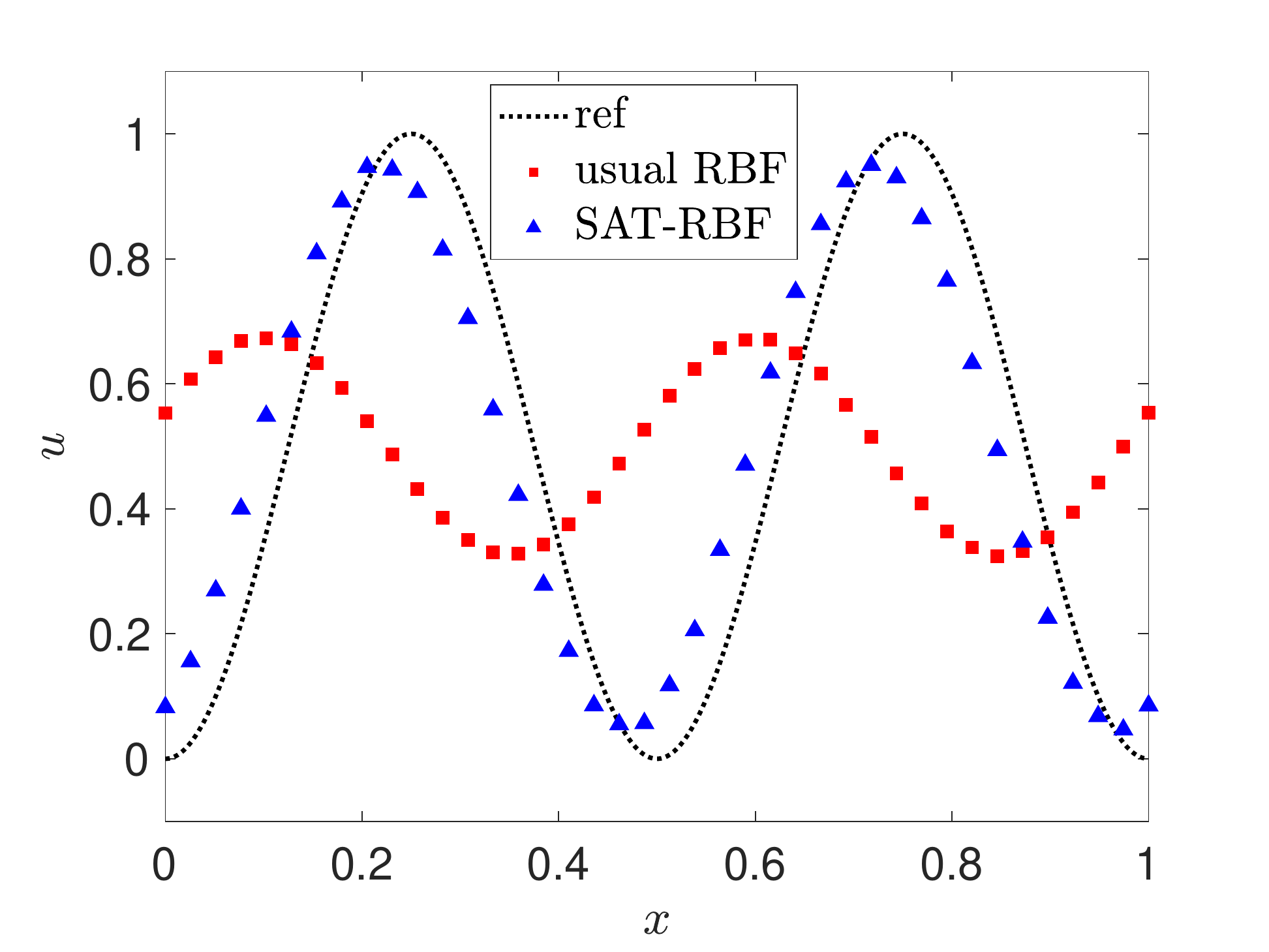}
		\caption{Quintic kernel, $t=80$}
		\label{fig:1d_periodic_sin^2_quintic_t80}
	\end{subfigure}% 
	\caption{Numerical solutions for \eqref{eq:FR_linear-adv} with periodic BCs and initial condition as in \eqref{eq:BC_periodic} for $N=40$}
	\label{fig:1d_periodic_sin^2}
\end{figure}

\begin{figure}[!htb]
	\centering
	\begin{subfigure}[b]{0.4\textwidth}
		\includegraphics[width=\textwidth]{%
			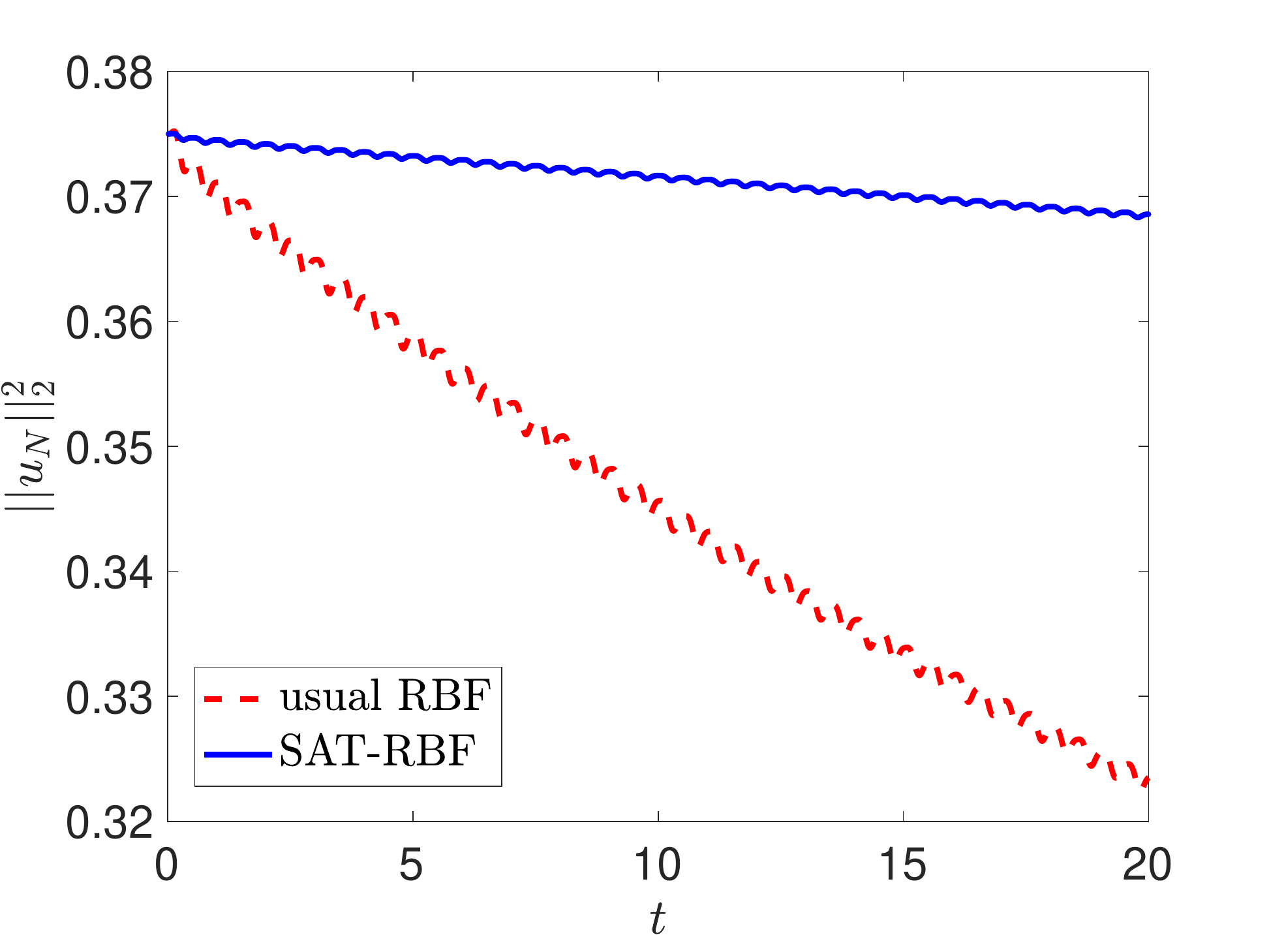}
		\caption{Cubic kernel}
		\label{fig:1d_periodic_sin^2_energy_cubic} 
	\end{subfigure}%
	~ 
	\begin{subfigure}[b]{0.4\textwidth}
		\includegraphics[width=\textwidth]{%
			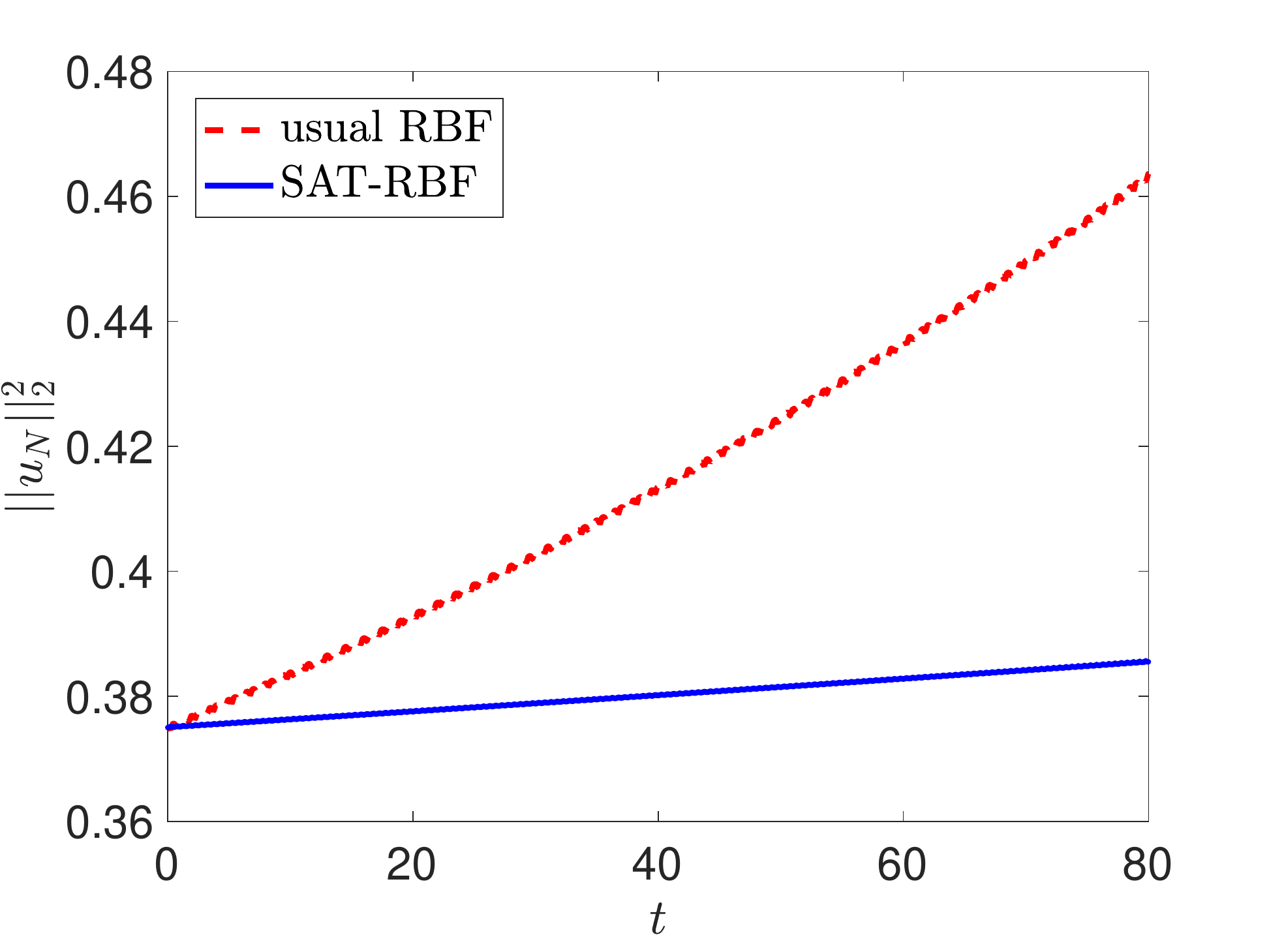}
		\caption{Quintic kernel}
		\label{fig:1d_periodic_sin^2_energy_quintic}
	\end{subfigure}% 
	\caption{Energy profiles over time corresponding to \eqref{eq:FR_linear-adv} with periodic BCs and initial condition as in \eqref{eq:BC_periodic} for $N=40$}
	\label{fig:1d_periodic_sin^2_energy}
\end{figure}

\begin{table}[htb]
	\renewcommand{\arraystretch}{1.3}
	\centering 
	\begin{adjustbox}{width=0.95\textwidth}
		\begin{tabular}{l l c c c c c c c c c c c c }
			\multicolumn{2}{c}{} 
			& \multicolumn{6}{c}{$\ell^1$-errors} 
			& \multicolumn{6}{c}{$\ell^\infty$-errors} \\ \hline 
			\multicolumn{2}{c}{} 
			& & \multicolumn{4}{c}{$N$} & $\mathcal{O}$  
			& & \multicolumn{4}{c}{$N$} & $\mathcal{O}$ \\ \hline
			kernel & Method  & & $10$ & $20$ & $40$ & $80$ & 
			& & $10$ & $20$ & $40$ & $80$ & \\ \hline
			cubic 
			& usual & & 3.4E-01 & 3.3E-01 &	3.9E-01 & 1.7E-01 & 0.3 & & 4.7E-01
			& 4.9E-01&6.0E-01& 2.7E-01 & 0.2  \\
			& SAT   & &	2.9E-01 & 5.6E-01 & 1.5E-01 & 3.2E-02 & 1.0 & & 4.2E-01
			& 8.5E-01 & 2.3E-01 & 5.1E-02 & 1.1\\\hline
			
			quintic 
			&usual & & 2.4E+02 & 3.2E+00 & 3.2E-01 & 3.6E-02 & 4.2 & & 3.3E+02 & 4.8E+00 & 5.3E-01 & 5.7E-02 & 4.1 \\
			& SAT   & & 6.4E+02 & 3.5e+02 & 5.8E-02 & 4.9E-02 & 4.5& & 9.2E+02 & 6.6E+02 & 1.4E-01 & 1.1e-01 & 4.3 \\ \hline
		\end{tabular}  
	\end{adjustbox}
	\caption{Errors for \eqref{eq:FR_linear-adv} with periodic BCs and initial condition as in \eqref{eq:BC_periodic} at $t=100$}
	\label{tab:1d_periodic_sin^2}
\end{table}

Note that for this problem the energy of the solution should actually be constant over time. 
Yet, in both cases, using the cubic as well as the quintic kernel, we respectively observe the usual RBF method to result in clearly decreasing and increasing energy profiles.  
For the SAT-RBF method, on the other hand, we observe this decrease/increase to be significantly smaller. 
Unfortunately, the FR-RBF method blew up before the final time could be reached in some tests. 
This demonstrates once more the numerical instability of the FR-RBF method. 
Henceforth, we therefore only consider the usual and SAT-RBF method. 
The SAT-RBF method, on the other hand, yielded stable computations in both cases. 
It should be noted that somewhat surprisingly the energy profile of the SAT-RBF method is also observed to increase over time in Figure \ref{fig:1d_periodic_sin^2_energy_quintic}. 
This seems to be a contradiction to our previous results that the SAT-RBF method is linear energy stable. 
This is caused by numerical rounding errors and time integration rather than by failure of the spatial semidiscretization, cf.\ \cite{nordstrom2008error,offner2019error,offner2018error}. 
Still, the SAT-RBF method yields visibly more accurate energy profiles than the usual RBF method. 
This is also reflected in the quality of their numerical solutions; see Figure \ref{fig:1d_periodic_sin^2}.  
Note that while the usual RBF starts to becomes visibly less accurate already at $t=10$ for both kernels, the SAT-RBF method is only observed to do so at $t=20$ for the cubic kernel and at $t=80$ for the quintic kernel.

\subsubsection{Equidistant vs.\ Nonequidistant Points} 
\label{sub:points}

We consider the same test case as in \S \ref{sub:long-time}, yet for nonequidistant points. 
Below, the potential implication of scattered points that are obtained by adding white uniform noise to a set of equidistant points is investigated. 
That is, the scattered points are given by 
\begin{equation}\label{eq:noise}
  \tilde{x}_0 = 0, \quad 
  \tilde{x}_N = 1, \quad
  \tilde{x}_n = x_n +  Z_n 
  \quad \text{with} \quad 
  x_n = \frac{n}{N} 
  \quad \text{and} \quad 
  Z_n \in \mathcal{U}\left(-\frac{1}{\sigma N},\frac{1}{\sigma N}\right),  
\end{equation}
for $n=1,\dots,N-1$. 
Here, the $Z_n$ are independent, identically distributed, and further assumed to not be correlated with the $x_n$.
Note that the scattered points are closer to equidistant points for larger $\sigma$ and less close to them for small $\sigma$, where $\sigma >0$ in all cases. 

\begin{figure}[!htb]
	\centering
	\begin{subfigure}[b]{0.33\textwidth}
		\includegraphics[width=\textwidth]{%
			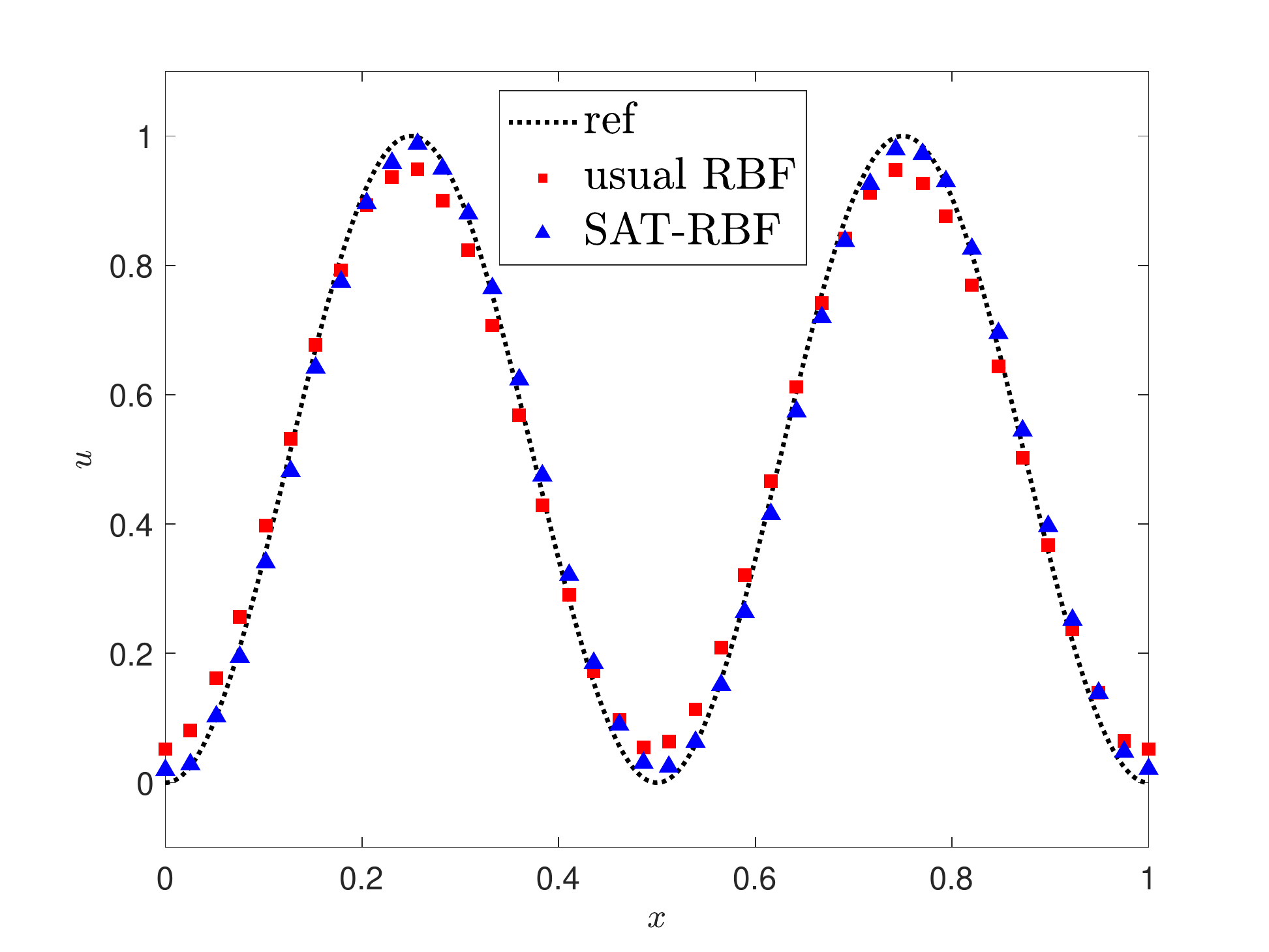}
		\caption{Cubic kernel, $\sigma = 20$}
		\label{fig:1d_periodic_sin^2_cubic_noneq_20}
	\end{subfigure}% 
	~ 
	\begin{subfigure}[b]{0.33\textwidth}
		\includegraphics[width=\textwidth]{%
			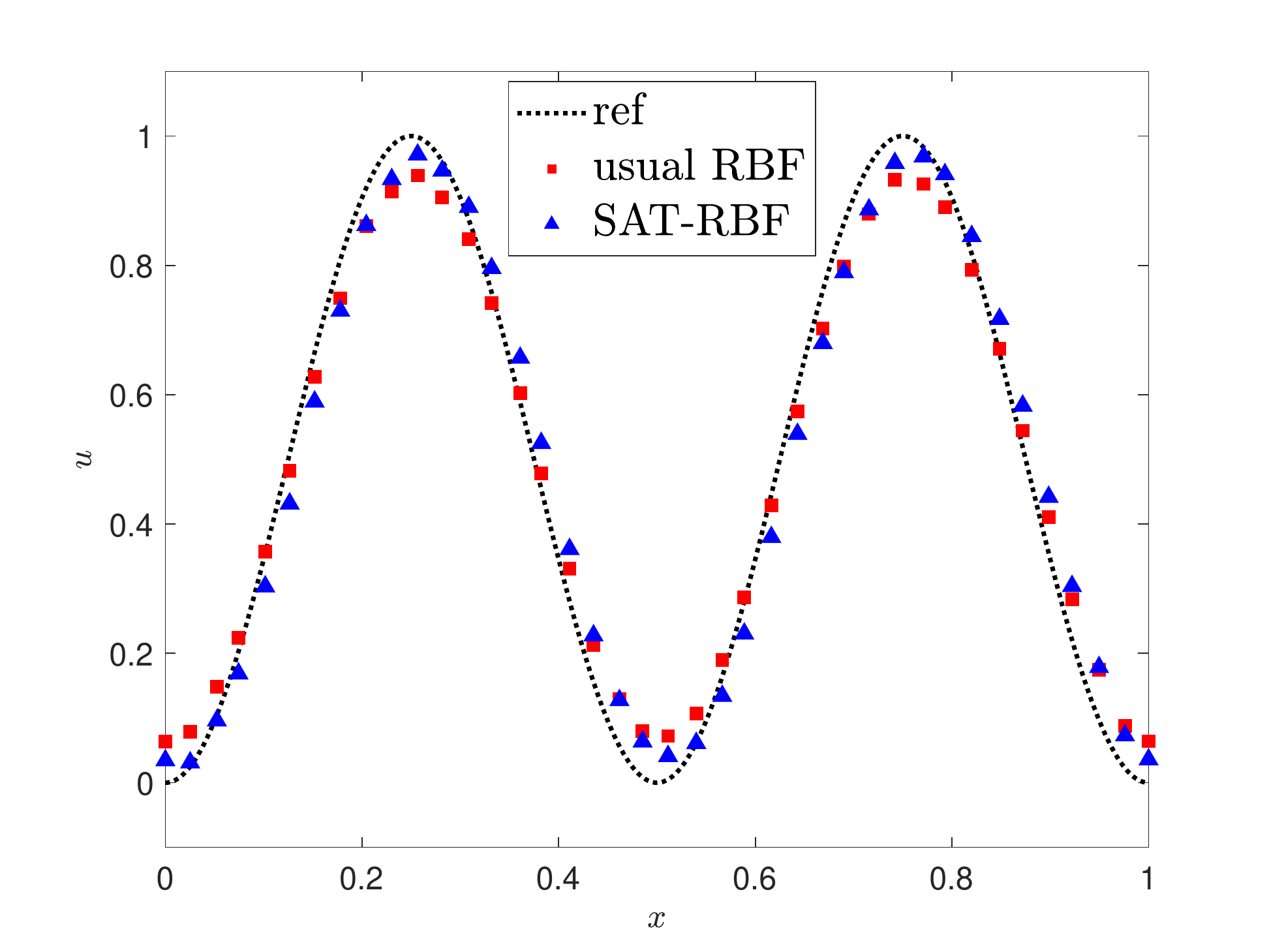}
		\caption{Cubic kernel, $\sigma = 10$}
		\label{fig:1d_periodic_sin^2_cubic_noneq_10}
	\end{subfigure}% 
	~ 
	\begin{subfigure}[b]{0.33\textwidth}
		\includegraphics[width=\textwidth]{%
			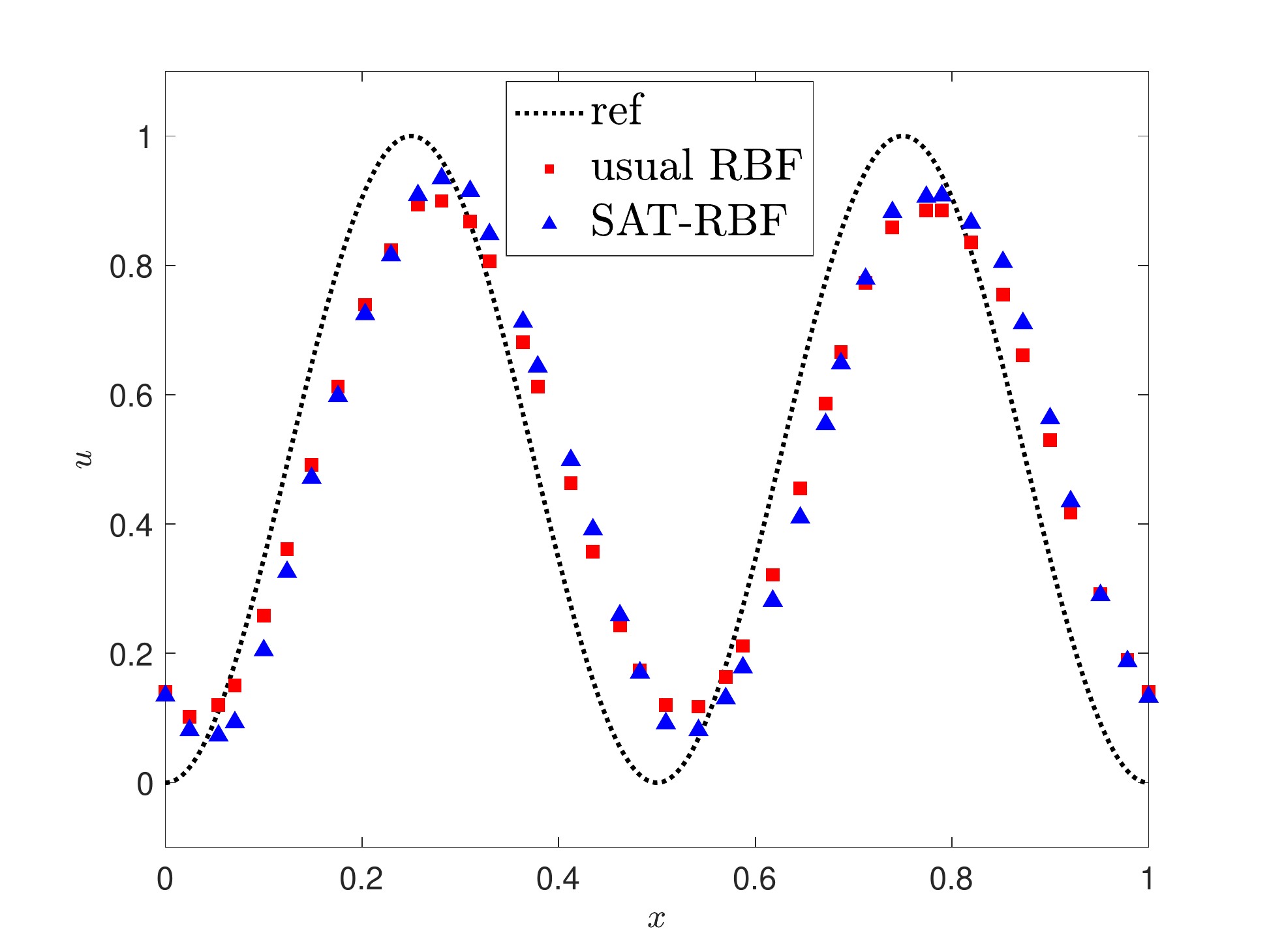}
		\caption{Cubic kernel, $\sigma = 4$}
		\label{fig:1d_periodic_sin^2_cubic_noneq_4}
	\end{subfigure}% 
	\\  
	\begin{subfigure}[b]{0.33\textwidth}
		\includegraphics[width=\textwidth]{%
			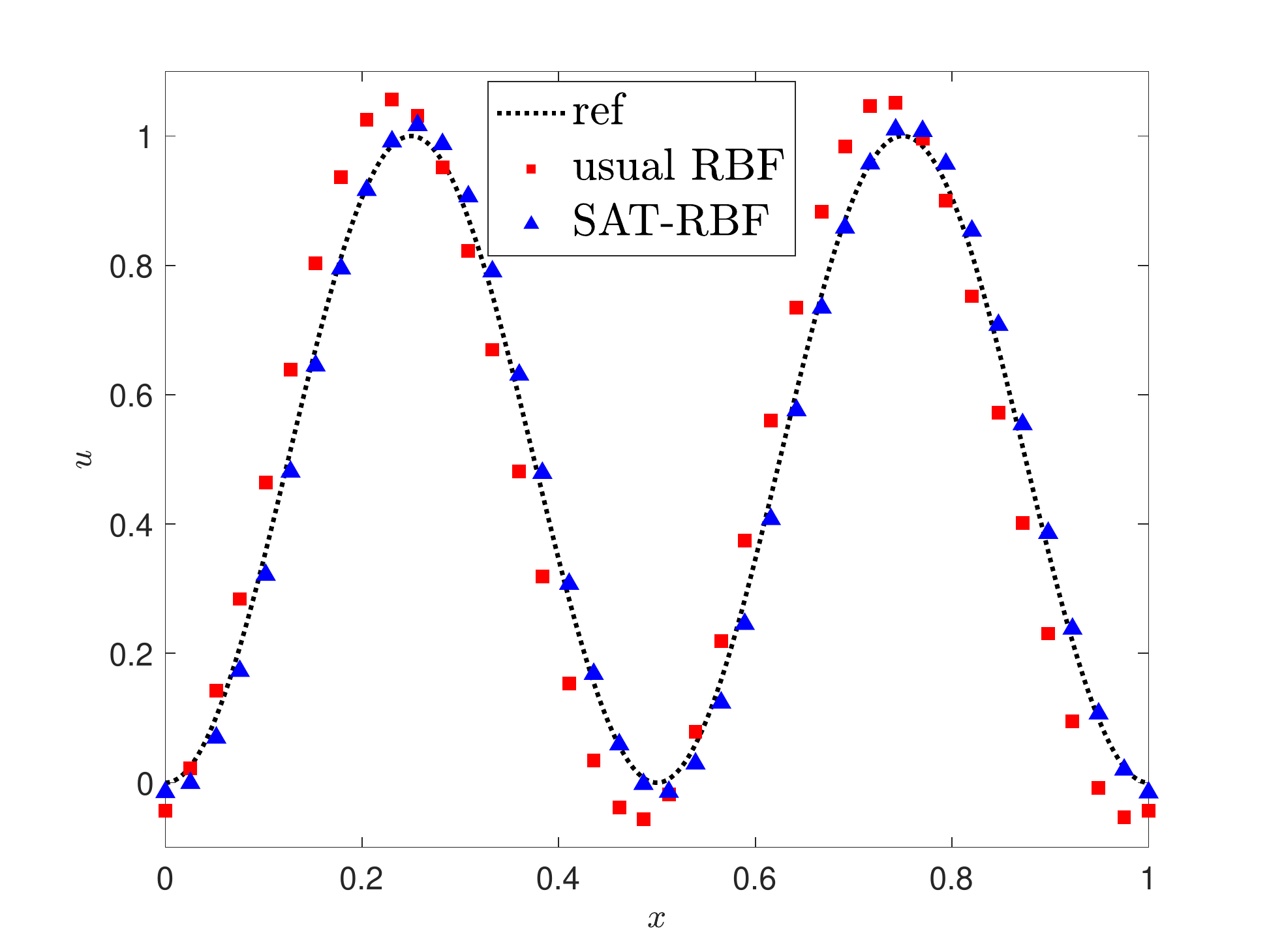}
		\caption{Quintic kernel, $\sigma = 20$}
		\label{fig:1d_periodic_sin^2_quintic_noneq_20}
	\end{subfigure}% 
	~ 
	\begin{subfigure}[b]{0.33\textwidth}
		\includegraphics[width=\textwidth]{%
			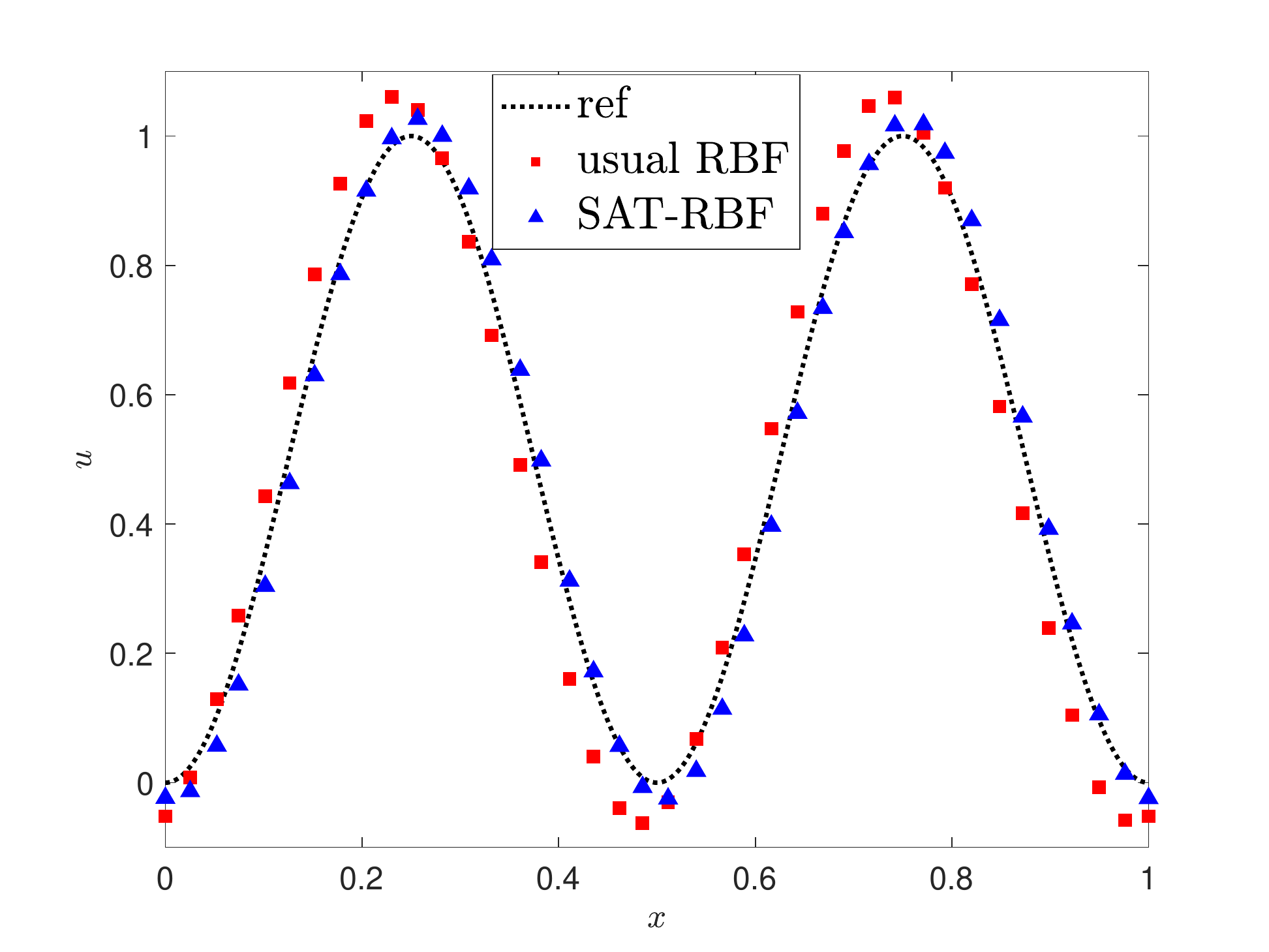}
		\caption{Quintic kernel, $\sigma = 10$}
		\label{fig:1d_periodic_sin^2_quintic_noneq_10}
	\end{subfigure}% 
	~ 
	\begin{subfigure}[b]{0.33\textwidth}
		\includegraphics[width=\textwidth]{%
			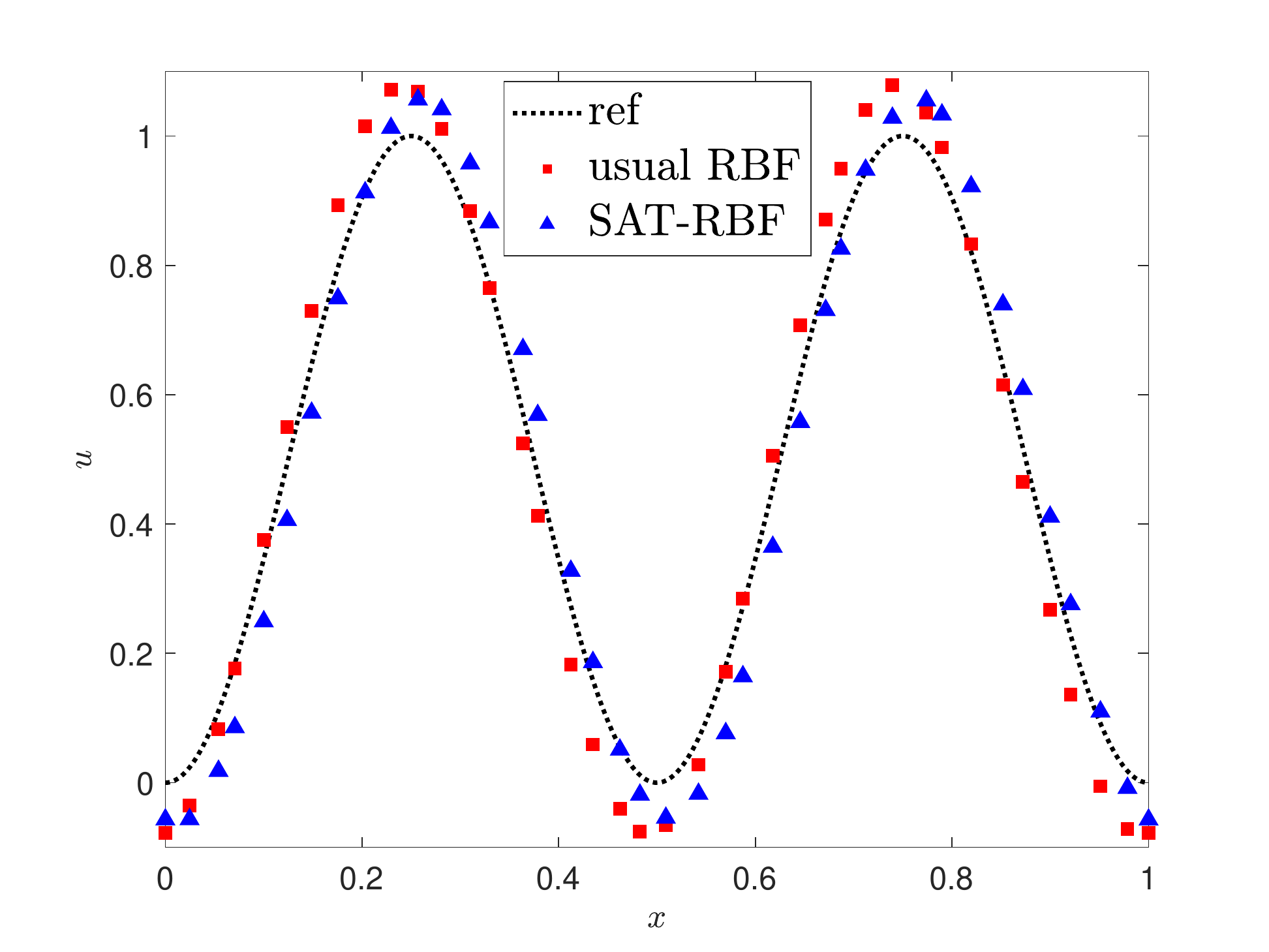}
		\caption{Quintic kernel, $\sigma = 4$}
		\label{fig:1d_periodic_sin^2_quintic_noneq_4}
	\end{subfigure}% 
	\caption{Numerical solutions for \eqref{eq:FR_linear-adv} with periodic BCs and initial condition as in \eqref{eq:BC_periodic} at $t=4$ (for the cubic kernel) as well as $t=10$ (for the quintic kernel) for $N=40$ nonequidistant points}
	\label{fig:1d_periodic_sin^2_noneq}
\end{figure}

Figure \ref{fig:1d_periodic_sin^2_noneq} illustrates the influence of going over from equidistant to scattered data points for the usual as well as SAT-RBF method. 
We observe that both methods suffer in accuracy as the points becomes increasingly nonequidistant. 
In particular, it can be noted the advantage of weakly enforcing BCs by the SAT approach, compared to the usual RBF method, decreases for scattered points. 
This behavior will be considered more in future investigations.

\subsubsection{Variable Coefficients} 

Given is the following advection problem from \cite{offner2019error}:
\begin{equation}\label{eq:test_linear-adv-variable-coeff}
\begin{aligned}
	\partial_t u + \partial_x \left(a(x)  u \right) & = 0, \quad && 0 < x < 2\pi , \ t > 0, \\ 
	u(t,0) & = 0, \quad && t > 0, \\ 
	u(0,x) & = u_{\text{init}}(x) , \quad && 0 \leq x \leq 2 \pi, \\ 
\end{aligned}
\end{equation}
where $u_{\text{init}}(x) = \sin(12(x-0.1))$ and $a(x) = x$ is a variable coefficient. 
The exact solution of \eqref{eq:test_linear-adv-variable-coeff} is given by
\begin{equation}
	u(t,x)= \exp(-t) u_{\text{init}}(x\exp(-t)).
\end{equation}

\begin{figure}[!htb]
	\centering
	\begin{subfigure}[b]{0.4\textwidth}
		\includegraphics[width=\textwidth]{%
			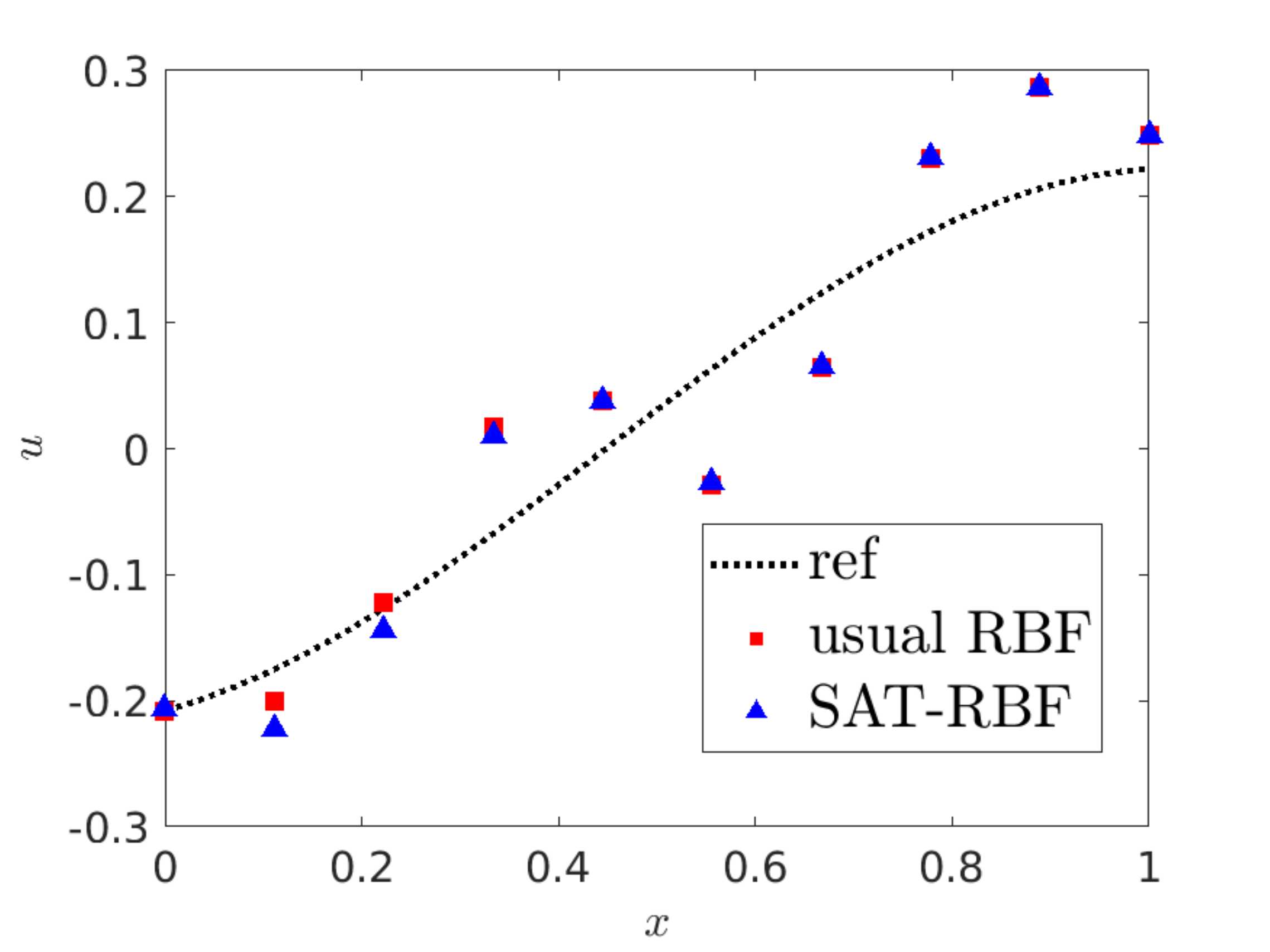}
		\caption{Solution, cubic kernel, $N=10$}
		\label{fig:1d_variableCoeff_cubic_10} 
	\end{subfigure}%
	~ 
	\begin{subfigure}[b]{0.4\textwidth}
		\includegraphics[width=\textwidth]{%
			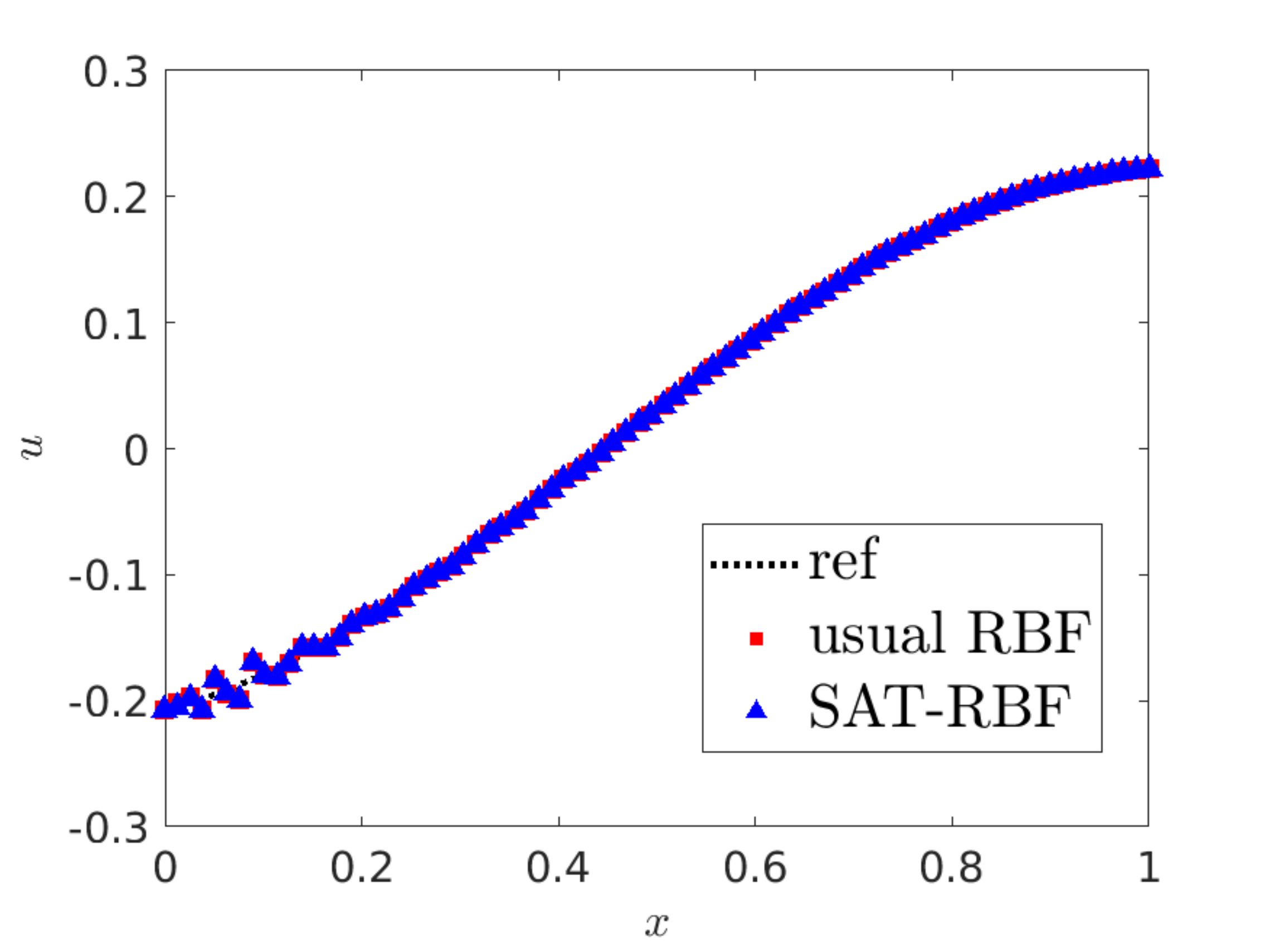}
		\caption{Cubic kernel, $N=80$}
		\label{fig:1d_variableCoeff_cubic_80} 
	\end{subfigure}%
	\\  
	\begin{subfigure}[b]{0.4\textwidth}
		\includegraphics[width=\textwidth]{%
			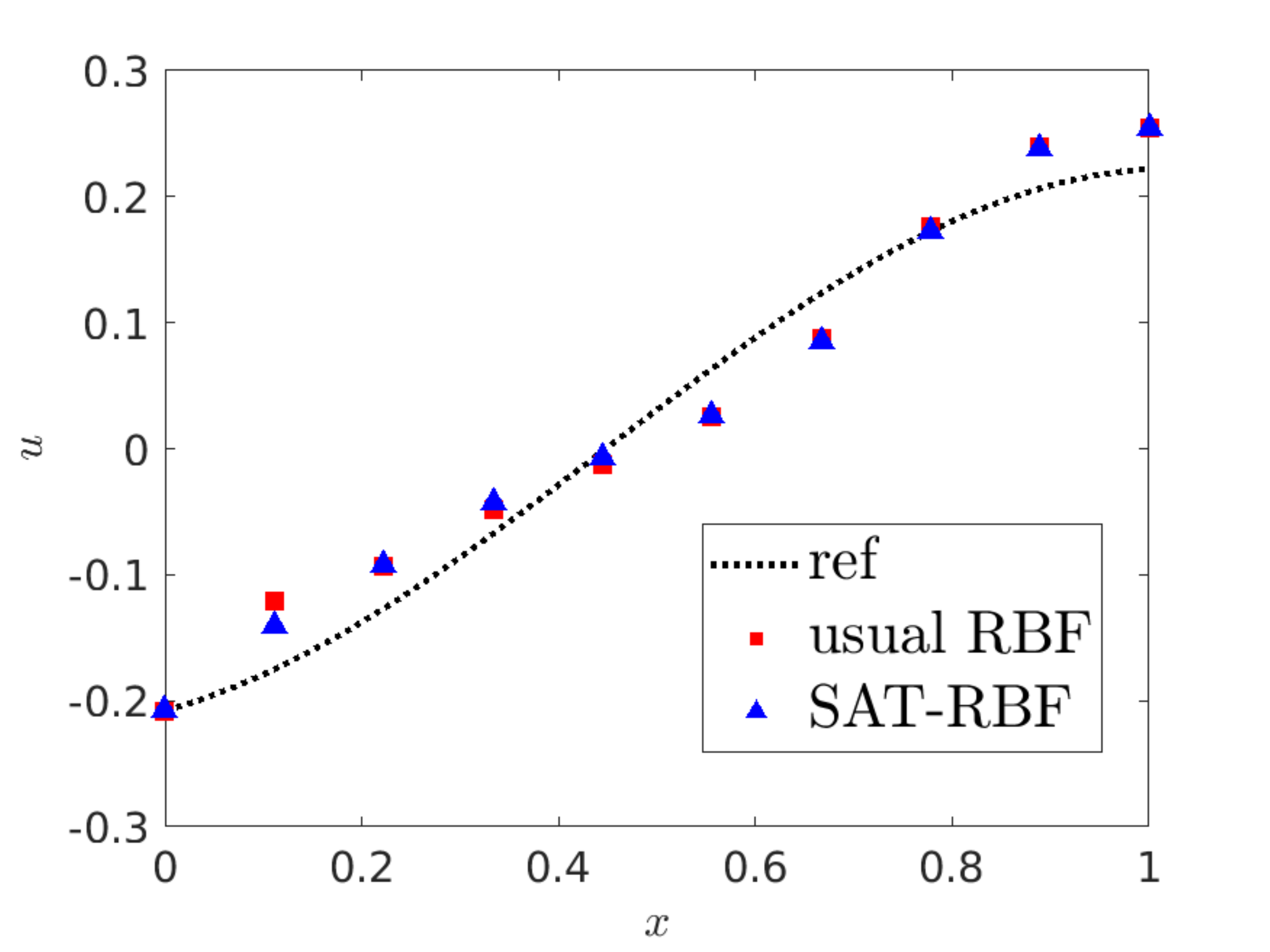}
		\caption{Quintic kernel, $N=10$}
		\label{fig:1d_variableCoeff_quintic_10} 
	\end{subfigure}%
	~ 
	\begin{subfigure}[b]{0.4\textwidth}
		\includegraphics[width=\textwidth]{%
			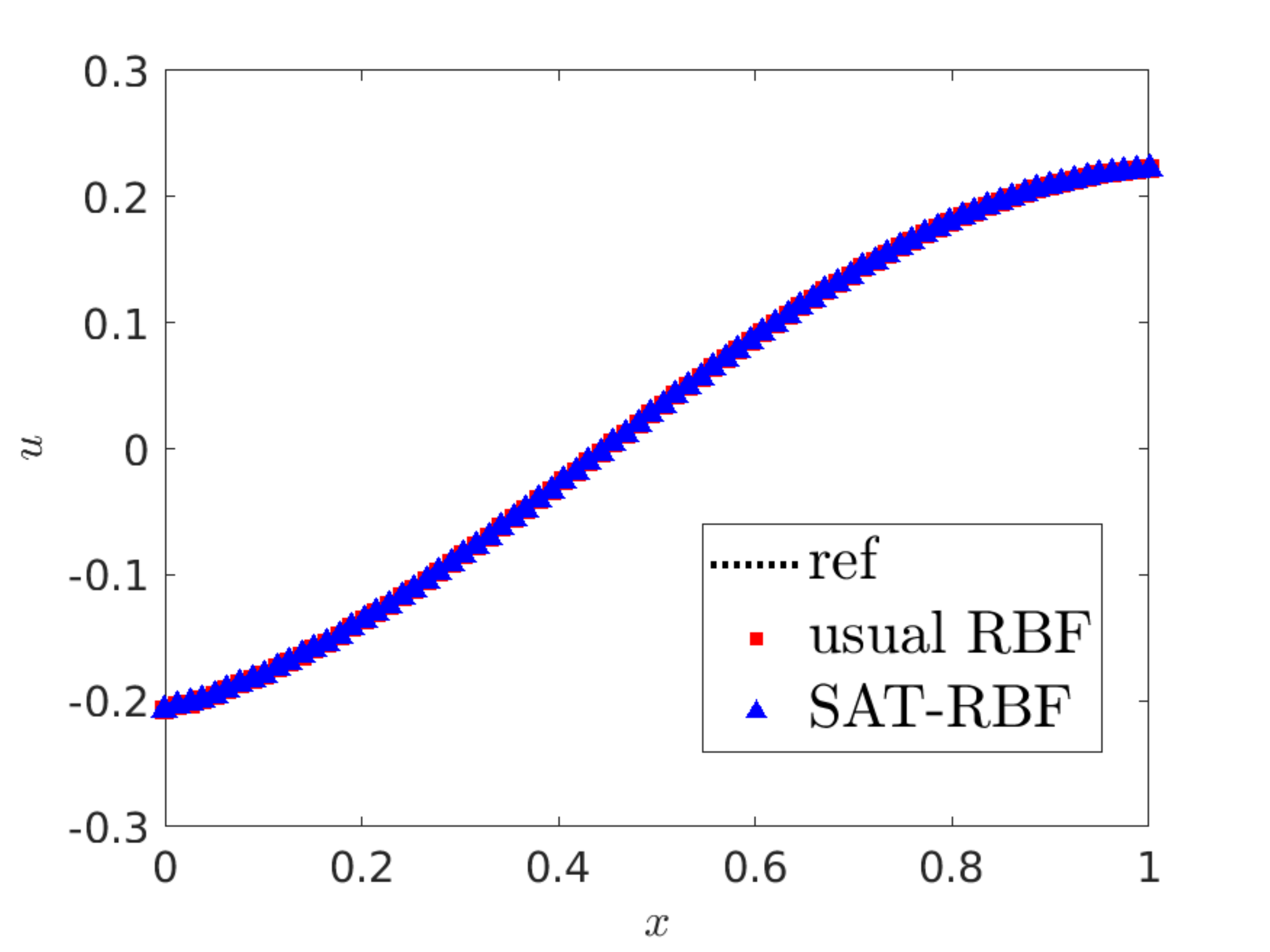}
		\caption{Quintic kernel, $N=80$}
		\label{fig:1d_variableCoeff_quintic_80} 
	\end{subfigure}%
	\caption{Numerical solutions of \eqref{eq:test_linear-adv-variable-coeff} at $t=1.5$}
	\label{fig:1d_variableCoeff}
\end{figure}

\begin{table}[htb]
	\renewcommand{\arraystretch}{1.3}
	\centering 
	\begin{adjustbox}{width=0.95\textwidth}
		\begin{tabular}{l l c c c c c c c c c c c c }
			\multicolumn{2}{c}{} 
			& \multicolumn{6}{c}{$\ell^1$-errors} 
			& \multicolumn{6}{c}{$\ell^\infty$-errors} \\ \hline 
			\multicolumn{2}{c}{} 
			& & \multicolumn{4}{c}{$N$} & $\mathcal{O}$  
			& & \multicolumn{4}{c}{$N$} & $\mathcal{O}$ \\ \hline
			kernel & Method  & & $10$ & $20$ & $40$ & $80$ & 
			& & $10$ & $20$ & $40$ & $80$ & \\ \hline
			cubic 
			& usual & & 4.7E-02 & 1.8E-02 & 5.3E-03 & 1.4E-03 & 1.6
			& & 9.2E-02 & 4.7E-02 & 2.1E-02 & 1.1E-02 & 1.0\\
			& SAT   & &	4.9E-02 & 1.8E-02 & 5.4E-03 & 1.4E-03 & 1.7
			& & 8.9E-02 & 4.7E-02 & 2.1E-02 & 1.1E-02 & 0.9\\\hline
			
			quintic 
			&usual & & 2.6E-02 & 4.6E-03 & 6.5E-04 & 9.2E-05 & 2.7	
			& & 3.8E-02 & 1.2E-02 & 3.6E-03 & 1.1E-03 & 1.7\\
			& SAT  & & 2.4E-02 & 4.4E-03 & 6.4E-04 & 8.7E-05 & 2.7			
			& & 3.8E-02 & 9.9E-03 & 3.4E-03 & 7.9E-04 & 1.7\\ \hline
		\end{tabular}  
	\end{adjustbox}
	\caption{Errors for \eqref{eq:test_linear-adv-variable-coeff} with inflow BCs at $t=1.5$.}
	\label{tab:1d_varcoeff_errors}
\end{table}

The results of the usual RBF and SAT-RBF method for the variable coefficient problem are illustrated in Figure \ref{eq:test_linear-adv-variable-coeff}. 
As also already observed before, the cubic kernel yields less reliable results than the quintic kernel. 
Furthermore, in the quintic case the SAT-RBF method tends to be slightly more accurate than the usual RBF method, as opposed to the contrary behavior for the cubic kernel. 
This can be noted from Table \ref{tab:1d_varcoeff_errors}. 
There it is also demonstrated that, once we increase the number of nodes, the RBF-SAT method provides slightly more accurate results.

\subsection{Extension to Systems: The Wave Equation with Sinusoidal BCs}

Extending our investigation to one-dimensional systems of advection equations, let us consider the wave equation 
\begin{equation} 
	\partial_{tt} u - c^2 \partial_{xx} u = 0, \quad 0 < x < 1, \ t > 0
\end{equation}
with sinusoidal BC and zero IC. 
Note that the wave equation can be rewritten as a system of linear advection equations, 
\begin{equation} \label{eq:acoustic-problem}
\begin{aligned}
	\partial_t u + c \partial_x v & = 0, \\
	\partial_t v + c \partial_x u & = 0,
\end{aligned}
\end{equation} 
which is sometimes referred to as the one-dimensional \emph{acoustic problem}; see \cite{gelb2008discrete,glaubitz2019stableDG}. 
In this formulation, the IC reads $u(0,x) = v(0,x) = 0$, while the sinusoidal BC is given by 
\begin{equation}\label{eq:sinusBCWave}
	\frac{1}{\sqrt{2}}\begin{pmatrix}1&1\\1&-1\end{pmatrix}\begin{pmatrix}u(t,0)\\v(t,0)\end{pmatrix} = \begin{pmatrix}\sin t\\0\end{pmatrix}, \quad 
	\frac{1}{\sqrt{2}}\begin{pmatrix}1&1\\1&-1\end{pmatrix}\begin{pmatrix}u(t,1)\\v(t,1)\end{pmatrix} = \begin{pmatrix}0\\\sin t\end{pmatrix}.
\end{equation}
Furthermore, we choose $c=1$.
The  boundary  operators are determined from the  eigenvalues and eigenvectors associated to the steady matrix of the system \eqref{eq:acoustic-problem}, given by
\begin{equation}
	X = \frac{1}{\sqrt{2}}\begin{pmatrix}1&1\\1&-1\end{pmatrix}.
\end{equation}
This yields the left and right SAT operators  
\begin{equation}\label{eq:bdoperatorsin}
	\Pi_0 = \begin{pmatrix}-R_0 & 1\\0 & 0\end{pmatrix}X^T
	\begin{pmatrix}u\\v\end{pmatrix}-\begin{pmatrix}\sin t\\0\end{pmatrix}, \quad 
	\Pi_1 = \begin{pmatrix}0&0\\1&-R_1\end{pmatrix}X^T\begin{pmatrix}u\\v\end{pmatrix}-\begin{pmatrix}0\\\sin t\end{pmatrix},
\end{equation}
where $R_0$ and $R_1$ both belong to $(0,1)$.

\begin{figure}[!htb]
	\centering
	\begin{subfigure}[b]{0.32\textwidth}
		\includegraphics[width=\textwidth]{%
			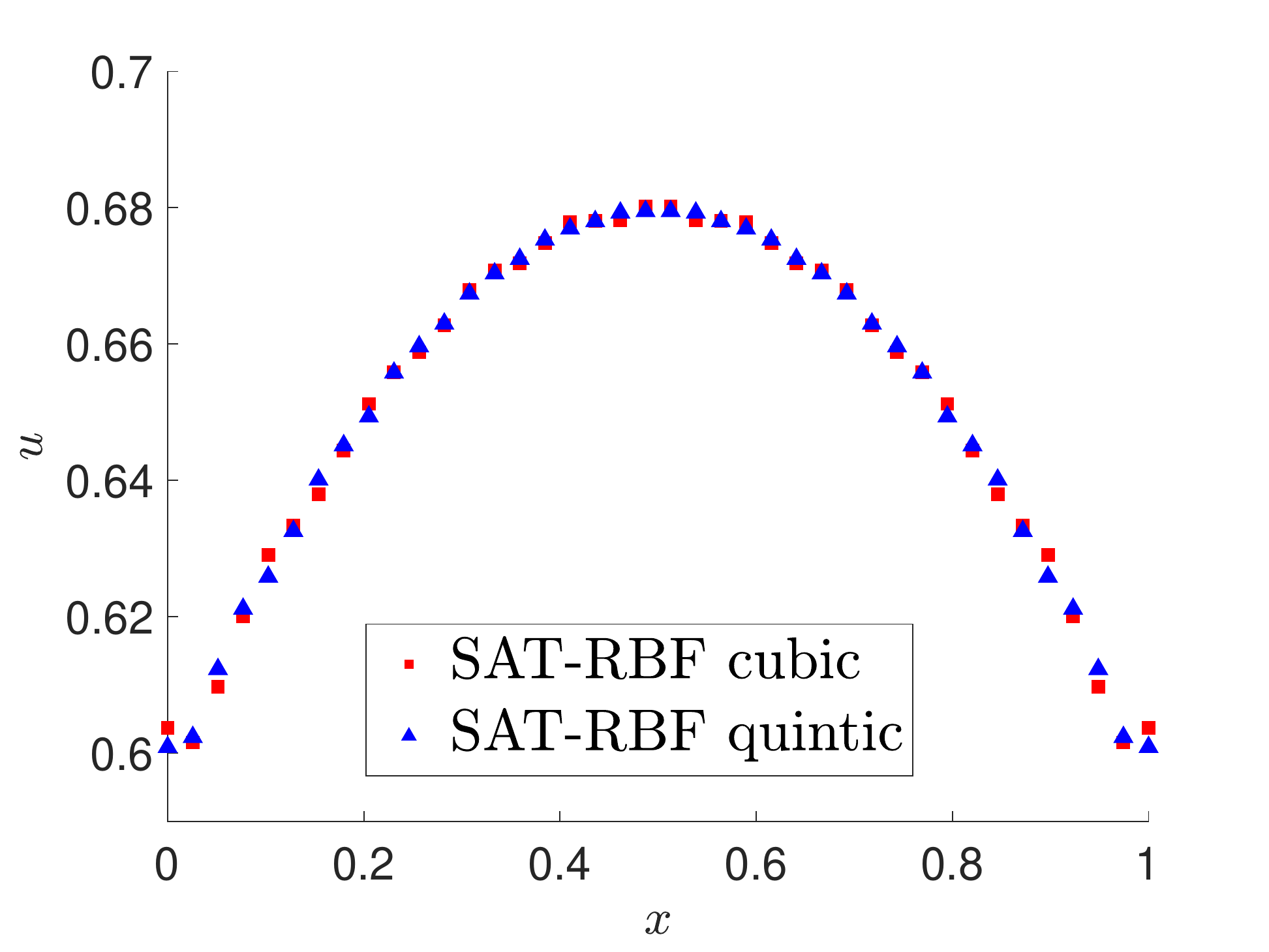}
		\caption{Solution $u$, $t=1$}
		\label{fig:wave_u_1} 
	\end{subfigure}%
	~
	\begin{subfigure}[b]{0.32\textwidth}
		\includegraphics[width=\textwidth]{%
			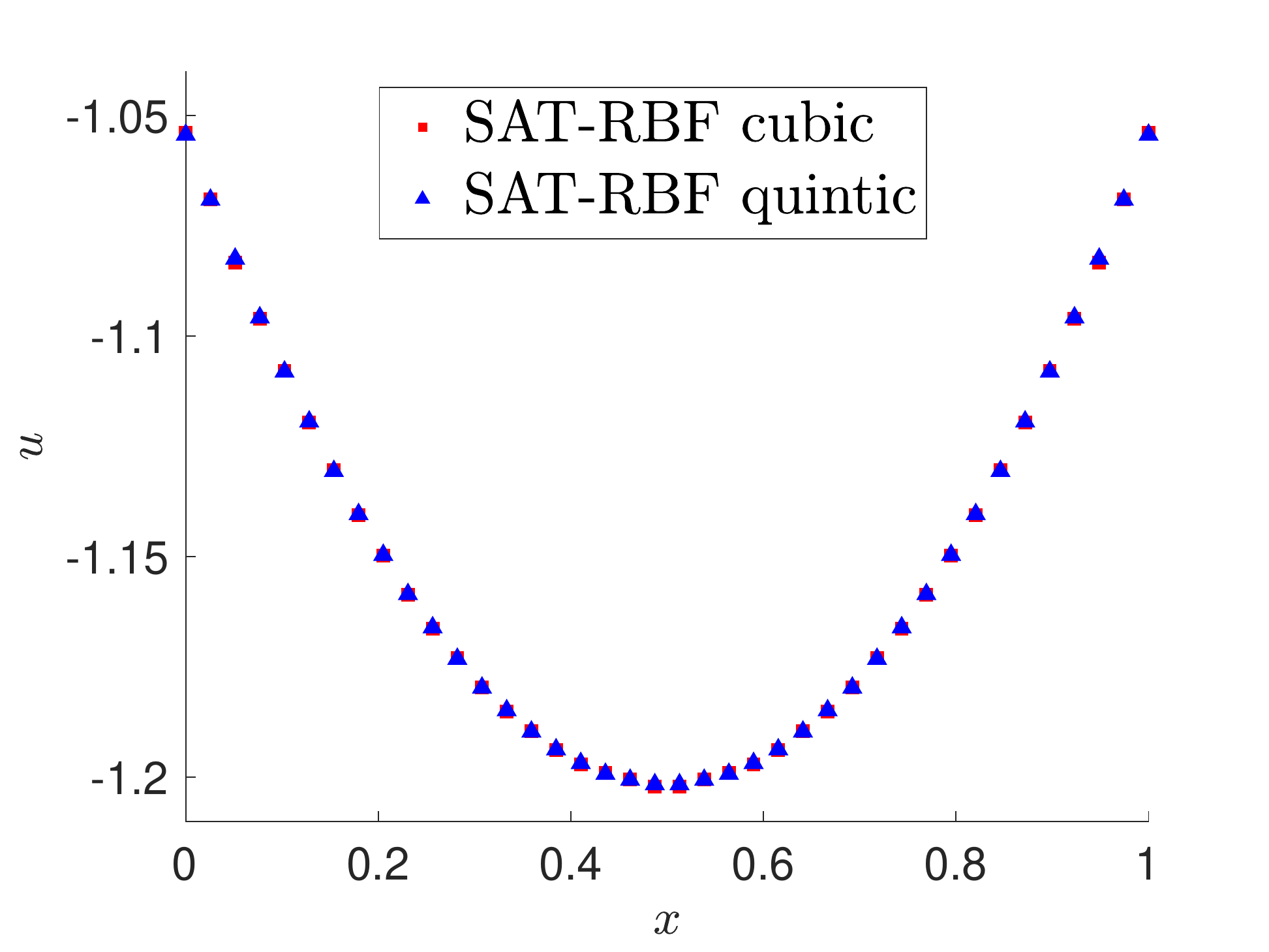}
		\caption{Solution $u$, $t=5$}
		\label{fig:wave_u_5} 
	\end{subfigure}
	~
	\begin{subfigure}[b]{0.32\textwidth}
		\includegraphics[width=\textwidth]{%
			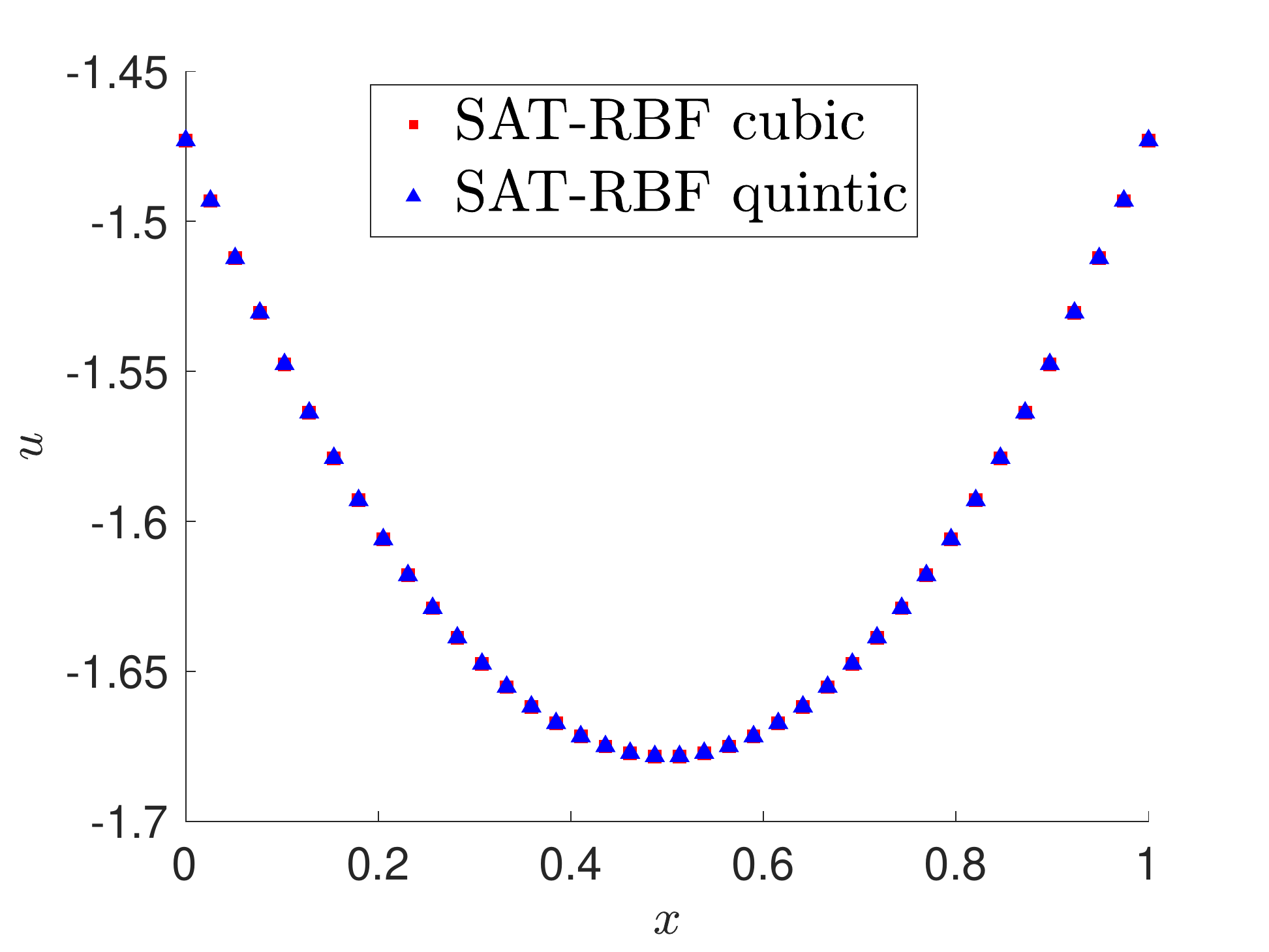}
		\caption{Solution $u$, $t=100$}
		\label{fig:wave_u_100} 
	\end{subfigure}
	\\  
	\begin{subfigure}[b]{0.32\textwidth}
		\includegraphics[width=\textwidth]{%
			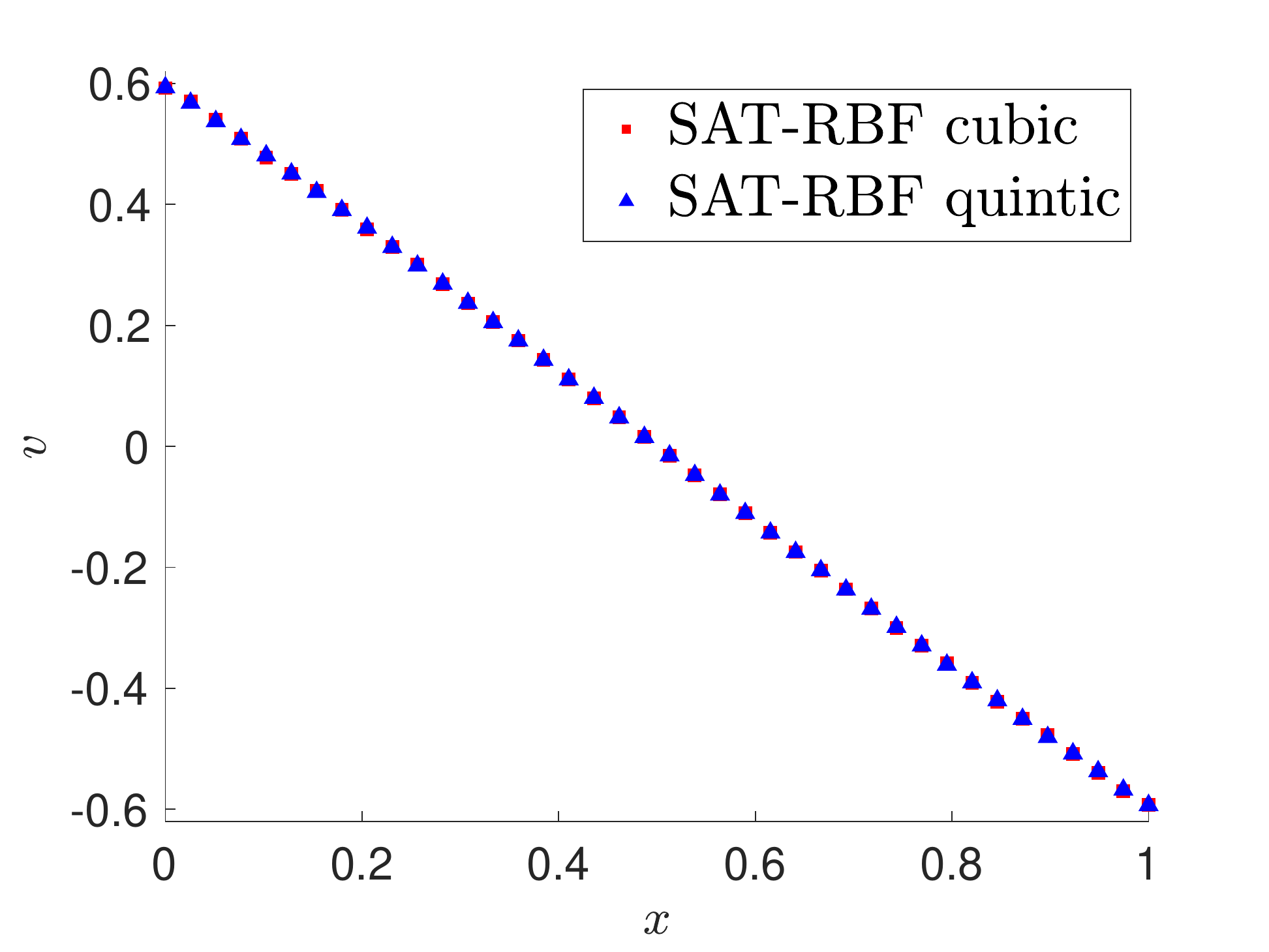}
		\caption{Solution $v$, $t=1$}
		\label{fig:wave_v_1} 
	\end{subfigure}%
	~
	\begin{subfigure}[b]{0.32\textwidth}
		\includegraphics[width=\textwidth]{%
			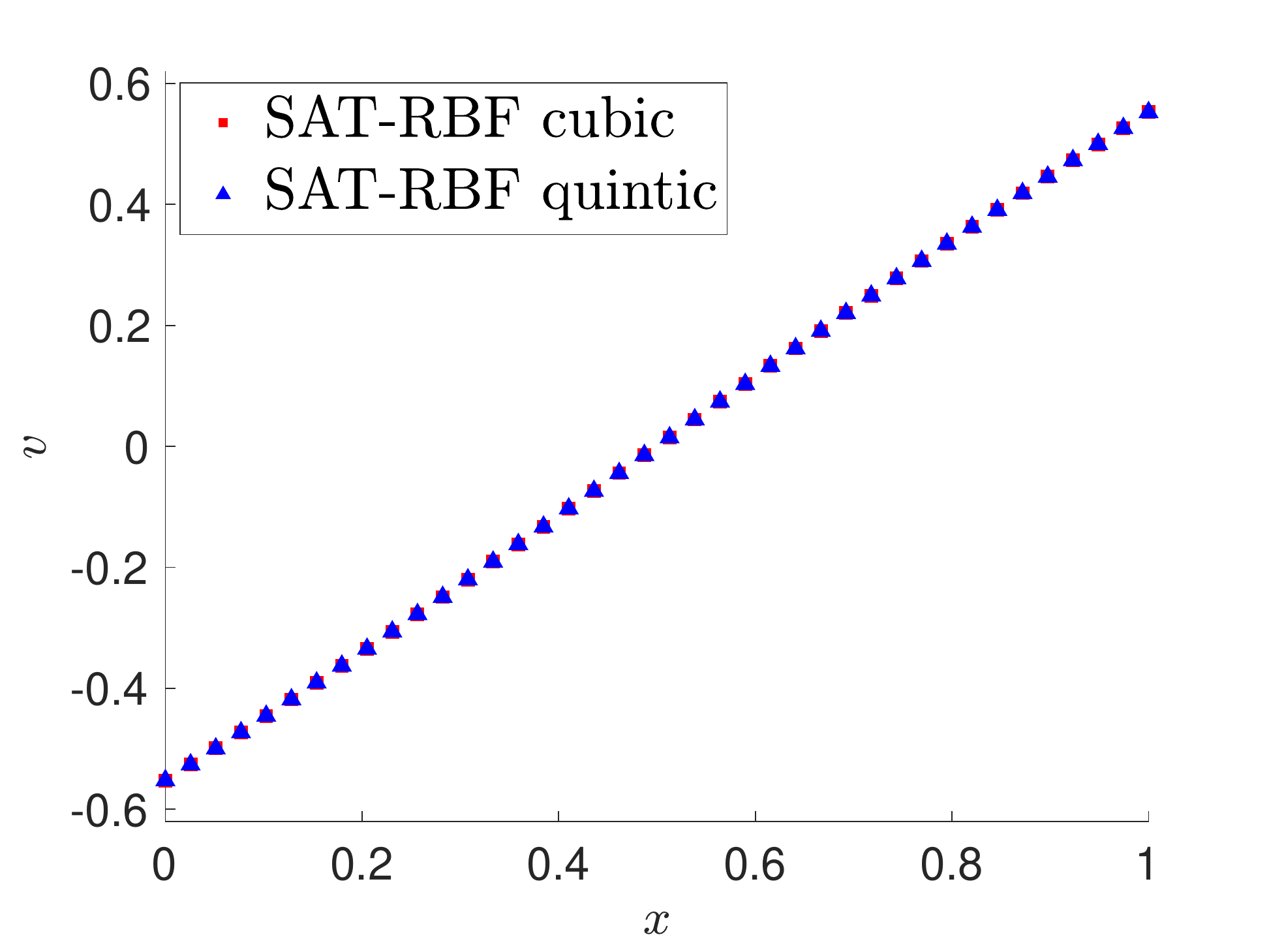}
		\caption{Solution $v$, $t=5$}
		\label{fig:wave_v_5} 
	\end{subfigure}%
	~
	\begin{subfigure}[b]{0.32\textwidth}
		\includegraphics[width=\textwidth]{%
			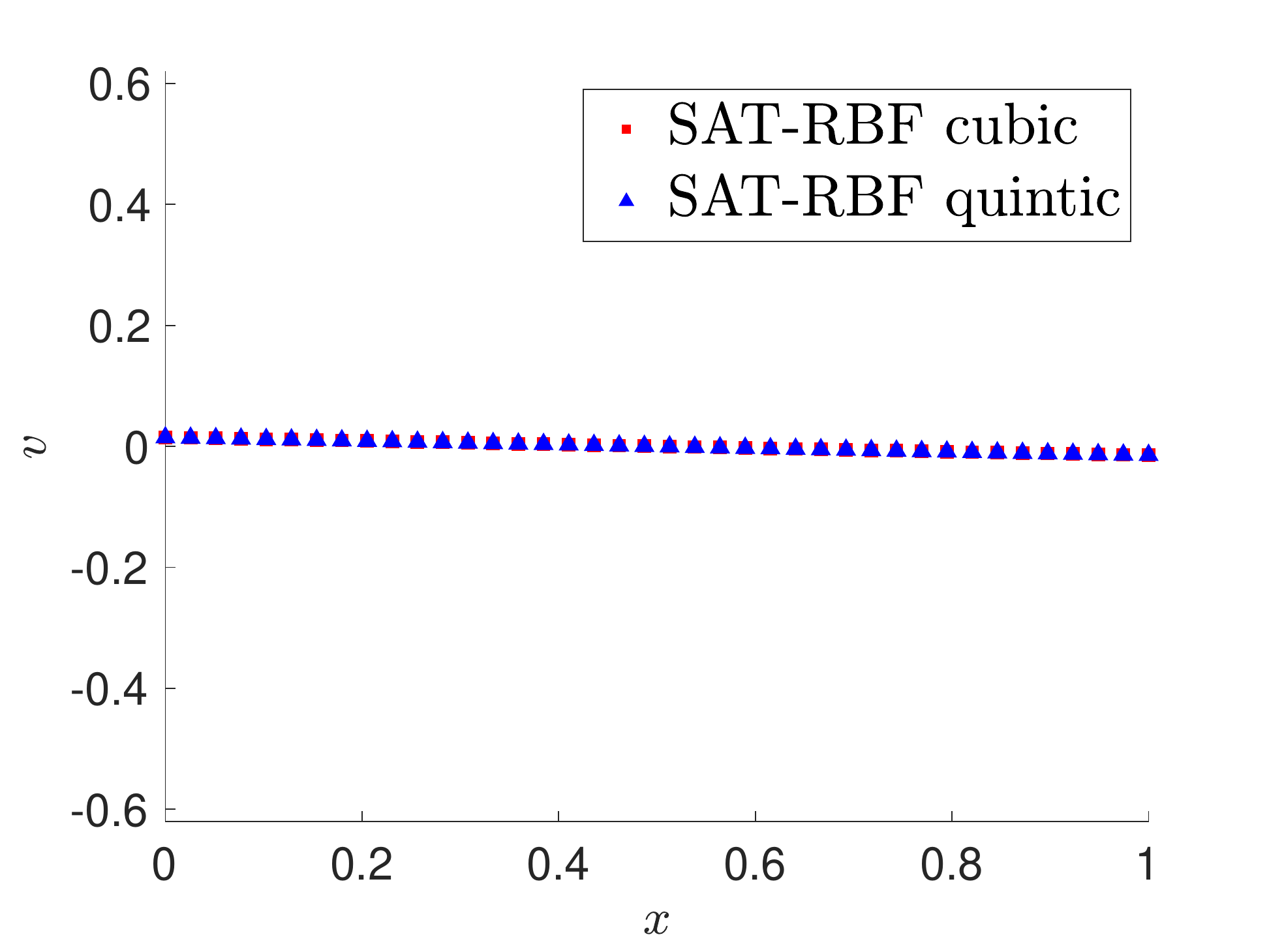}
		\caption{Solution $v$, $t=100$}
		\label{fig:wave_v_100} 
	\end{subfigure}
	\caption{Numerical solutions for the acoustic problem \eqref{eq:acoustic-problem} with zero IC and sinusoidal BCs}
	\label{fig:wave_equation_t}
\end{figure}

Figure \ref{fig:wave_equation_t} reports the numerical results for the two components $u$ and $v$ for the cubic and quintic kernel using $N=40$ nodes at different times. 
For this test case, the usual RBF method yielded quite dubious and inaccurate results. 
Hence, Figure \ref{fig:wave_equation_t} only illustrates the SAT-RBF method, which is demonstrated to also yield accurate results in this test case.
In particular, it can be noted that the SAT-RBF method remains stable even for long times, regardless of the chosen kernel. 
It is also interesting to observe that, though the quintic kernel seems to perform better for short time intervals, the impact of the chosen kernel decreases as we go forward in time.

\subsection{Extension to Two-Dimensions: Scalar Advection With Zero-Inflow Boundary Conditions}

As a last test case, we address the extension of the SAT-RBF method by considering the scalar two-dimensional linear advection problem 
\begin{equation}\label{eq:2d-problem}
\begin{aligned}
  	\partial_t u + \partial_x u  & = 0, \\ 
	u(0,x,y) & = \sin(4\pi x) \left( 1 - \frac{1}{2} \sin(2\pi y) \right), \\
	u(t,0,y) & = 0, 
\end{aligned}
\end{equation}
on ${\Omega=[0,1]^2 \subset \R^2}$. 
The exact solution is given by
\begin{equation}\label{eq:ref-sol} 
\begin{aligned}
  u(t,x,y) 
    \begin{cases}
    		0 & \text{ if } x \leq t, \\ 
		u(0,x-t,y) & \text{ otherwise}.
    \end{cases}
\end{aligned}
\end{equation}
Following the discussion in \S \ref{sub:SAT-2d}, an appropriate discretized SAT term is given by 
\begin{equation}
	SAT = - \frac{1}{2} H^{-1} E_w \mathbf{u}.  
\end{equation} 
Here, $H$ denotes the diagonal \emph{mass matrix} given by 
\begin{equation}
	H = \diag{h_1,\dots,h_N}, \quad 
	h_n = \int_{\Omega} \psi_n(\boldsymbol{x}) \intd \boldsymbol{x},
\end{equation} 
with $\{ \psi_n \}_{n=1}^N$ being the nodal basis that spans the approximation space $V_{N,m}$; 
see \S \ref{sub:RBF-interpol}. 
Furthermore, $\mathbf{u}$ is the vector containing the nodal values of $u_N$ at the centers $\mathbf{x}_n$, $n = 1,\dots,N$, and $E_w$ is the matrix defined by 
\begin{equation}
	E_w(n_1,n_2) = 
		\begin{cases} 
			1 & \text{ if } \ \mathbf{x}_n = ( 0, y_{n} )^T, \\ 
			0 & \text{ otherwise}.
		\end{cases}
\end{equation}

\begin{figure}[!htb]
	\centering
	\begin{subfigure}[b]{0.31\textwidth}
		\includegraphics[width=\textwidth]{%
			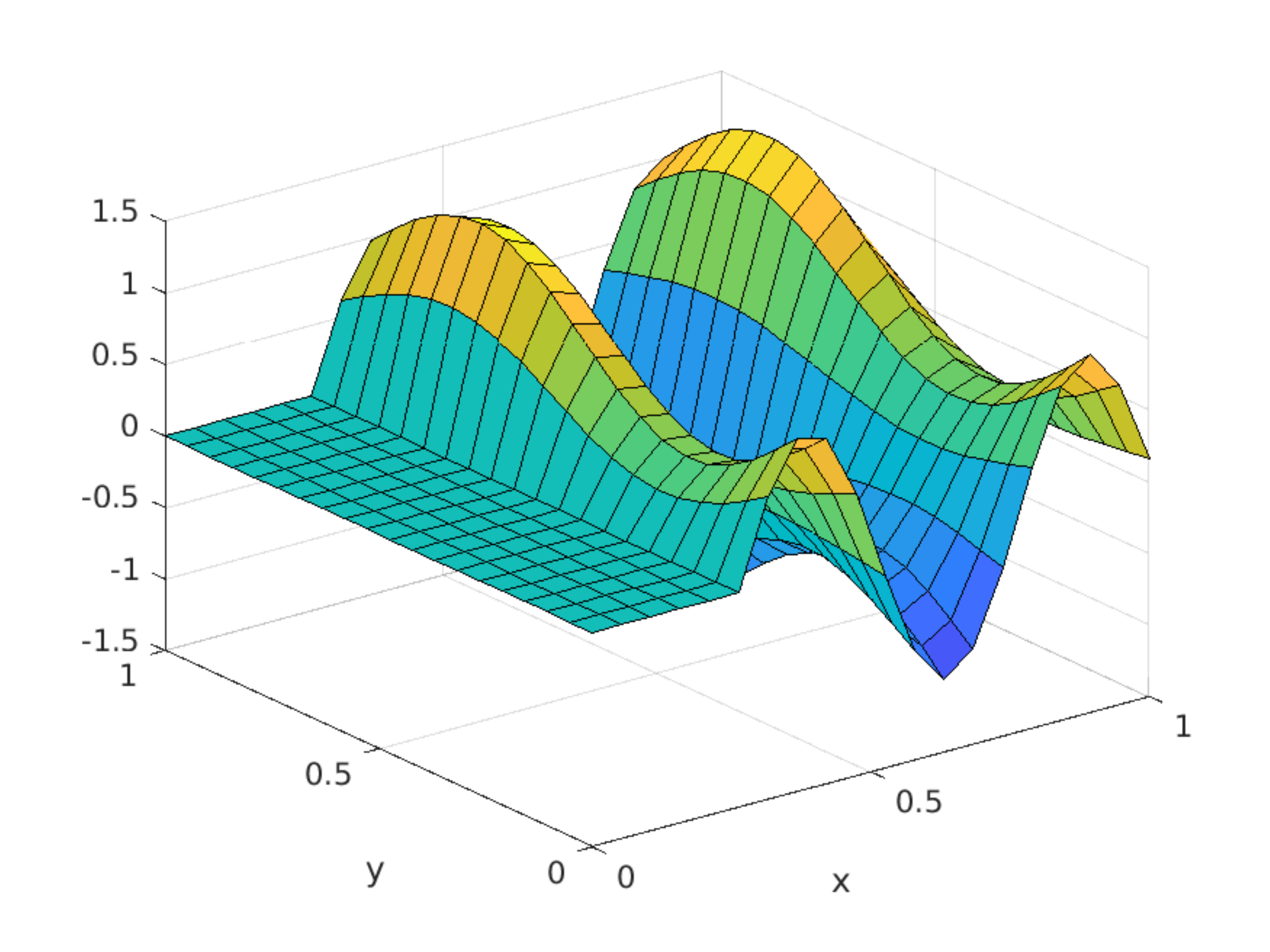}
		\caption{Exact solution, $t=0.2631$}
	\end{subfigure}
	~ 
	\begin{subfigure}[b]{0.31\textwidth}
		\includegraphics[width=\textwidth]{%
			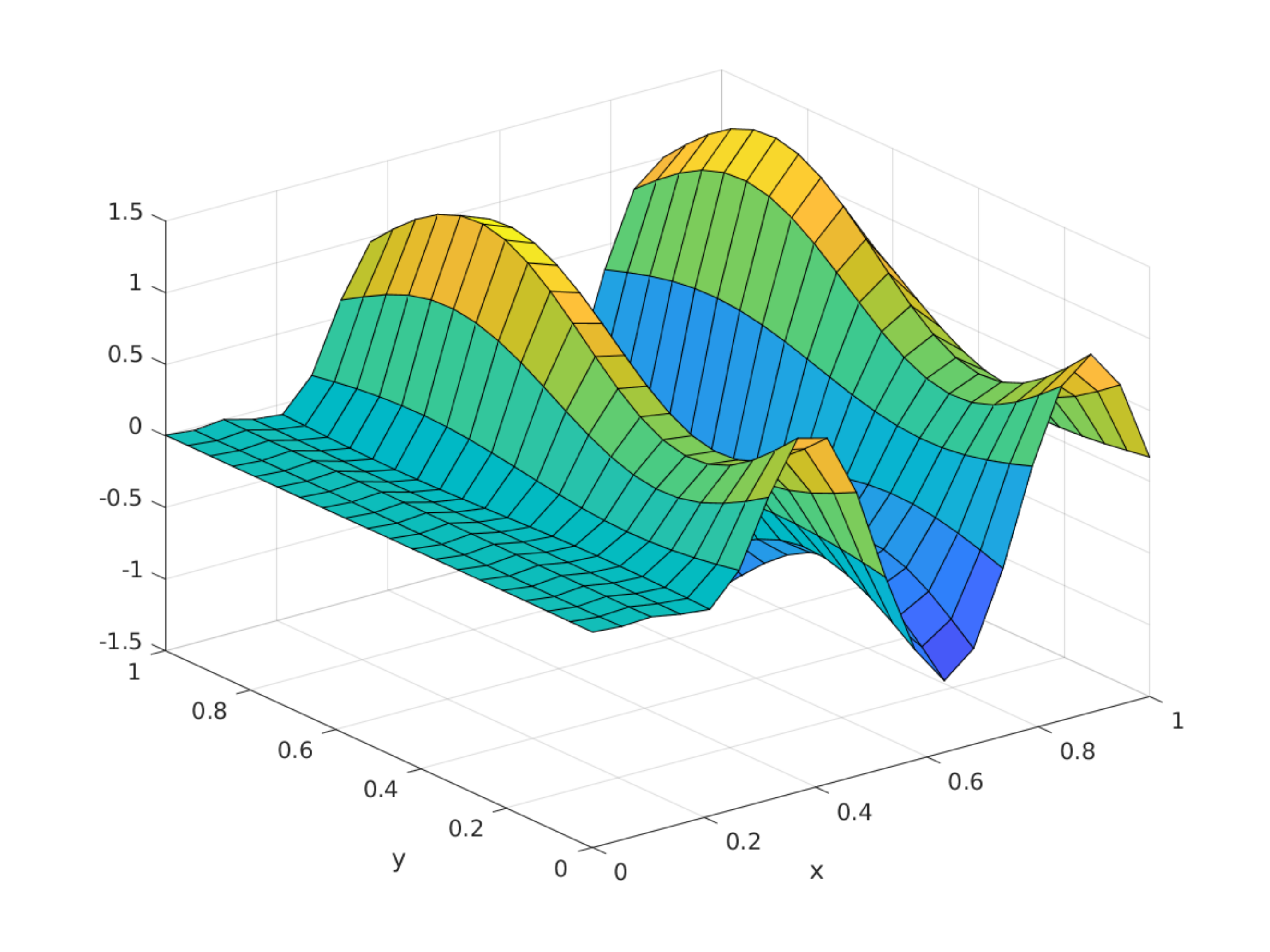}
		\caption{Usual RBF, $t=0.2631$}
	\end{subfigure}
	~ 
	\begin{subfigure}[b]{0.31\textwidth}
		\includegraphics[width=\textwidth]{%
			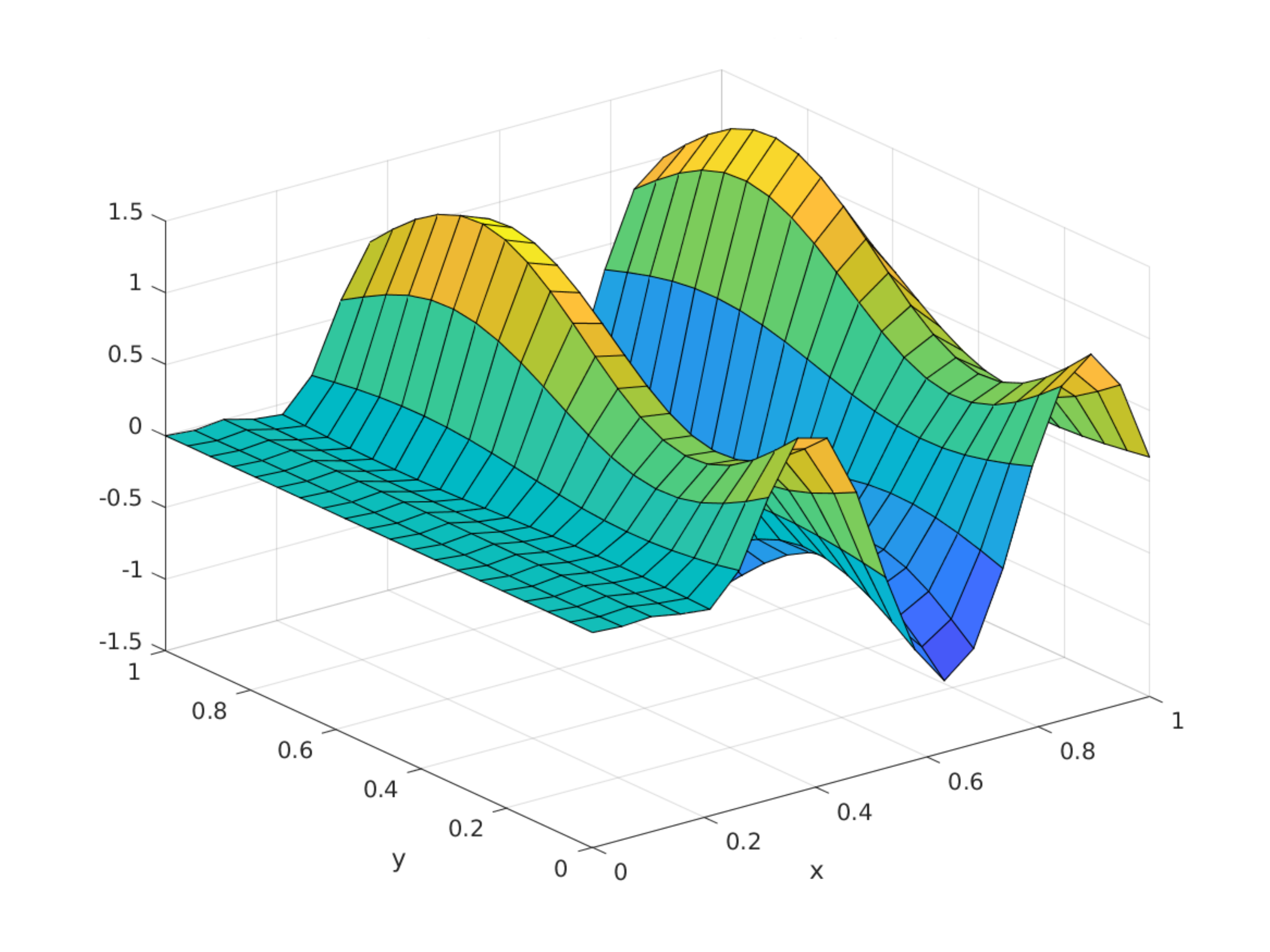}
		\caption{SAT-RBF, $t=0.2631$}
	\end{subfigure}
	\\ 
	\begin{subfigure}[b]{0.31\textwidth}
		\includegraphics[width=\textwidth]{%
			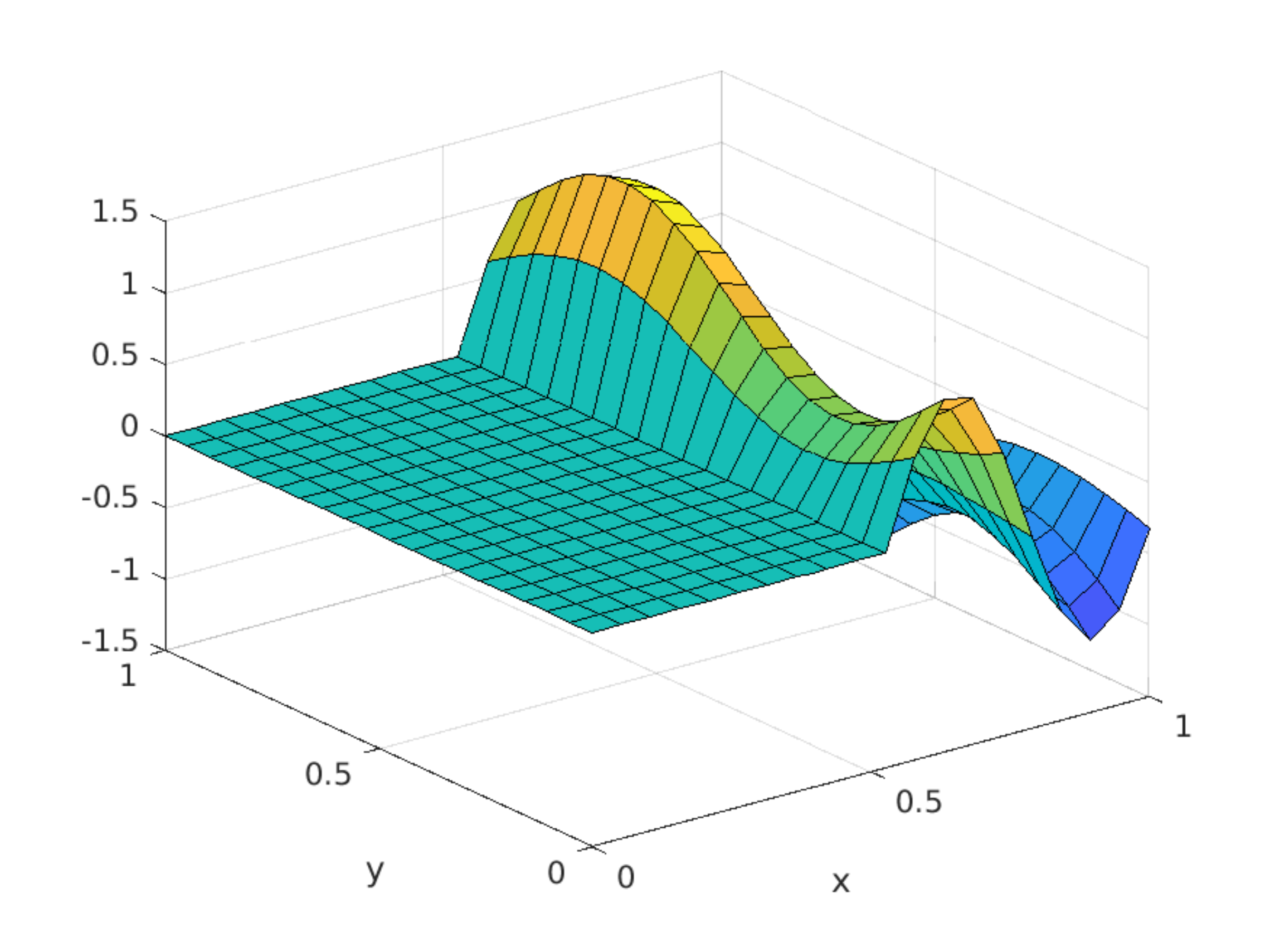}
		\caption{Exact solution, $t=0.5263$}
	\end{subfigure}
	~
	\begin{subfigure}[b]{0.31\textwidth}
		\includegraphics[width=\textwidth]{%
			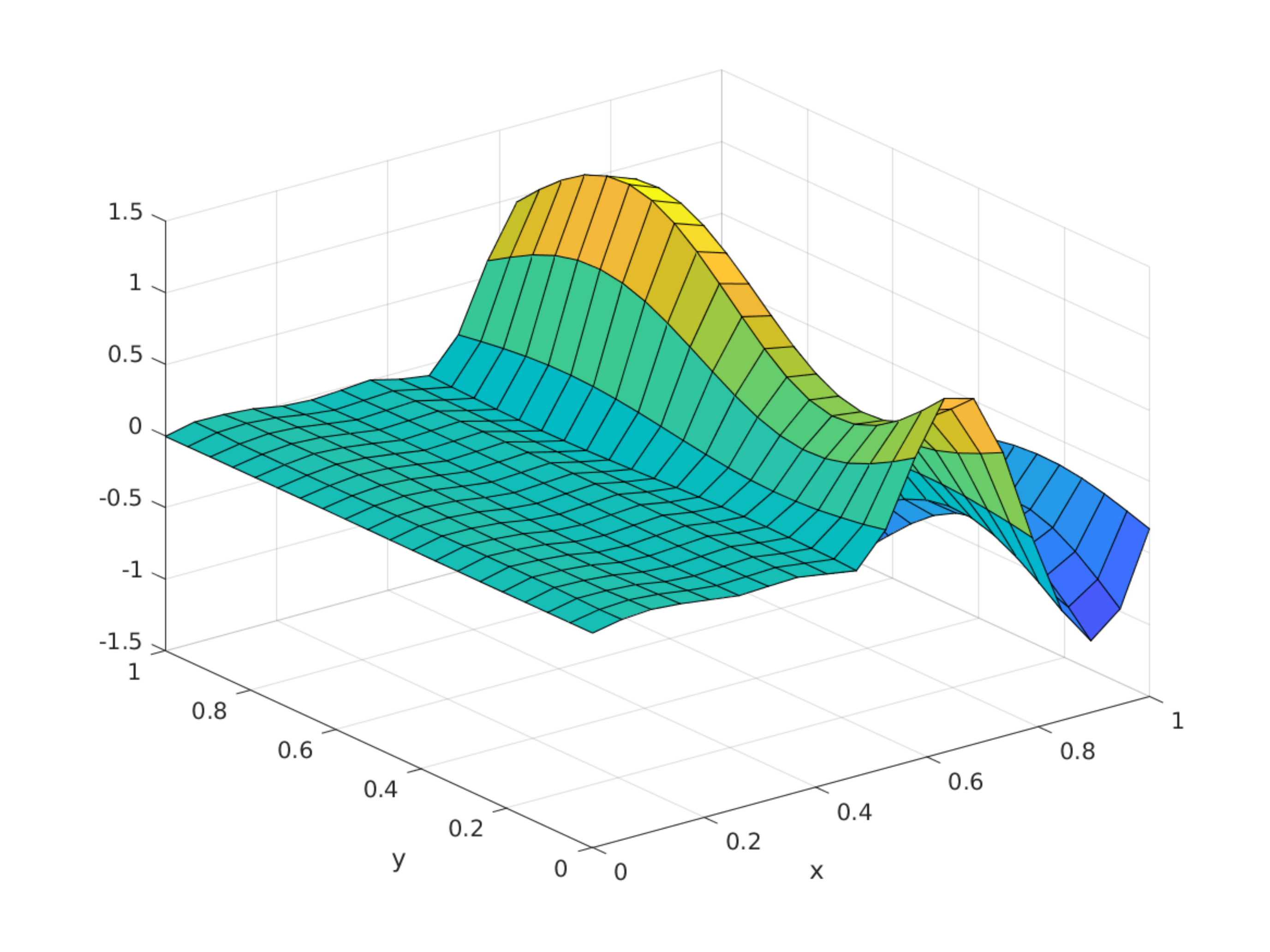}
		\caption{Usual RBF, $t=0.5263$}
	\end{subfigure} 
	~ 
	\begin{subfigure}[b]{0.31\textwidth}
		\includegraphics[width=\textwidth]{%
			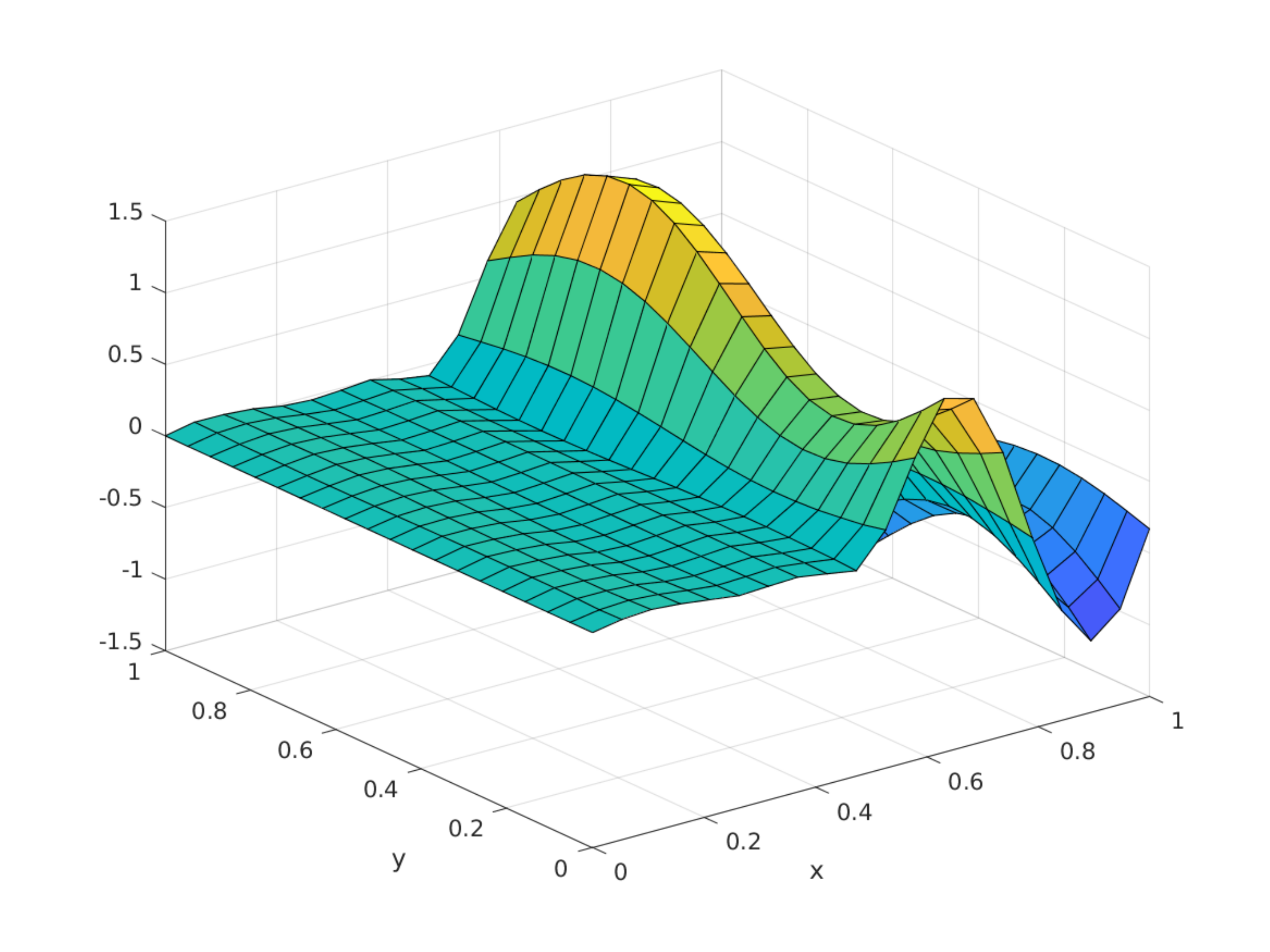}
		\caption{SAT-RBF, $t=0.5263$}
	\end{subfigure}
	\caption{(Numerical) solutions of the usual RBF and SAT-RBF method at two different times}
	\label{fig:SlidingSinus}
\end{figure}

Figure \ref{fig:SlidingSinus} illustrates the results of the usual RBF and SAT-method at times $t=0.2631$ and $t=0.5263$. 
In all cases a quintic kernel has been used. 
We performed the same test cases also for a cubic and a multiquadric kernel. 
The results were essentially the same and are therefore not illustrated here. 
Yet, the $L^2$-norms of the error at time $t=0.52632$ for all three kernels are reported in Table \ref{tab:errors}. 

\begin{table}[!htb] 
	\renewcommand{\arraystretch}{1.3}
	\centering
	\begin{tabular}{c c c c c} 
		\toprule
		 & & cubic & quintic & multiquadric \\ \hline
		usual RBF & & 1.14 & 1.36 & 8.82E-1 \\ \hline 
		SAT-RBF & & 1.12 & 1.33 & 8.69E-1 \\ \bottomrule 
	\end{tabular}
	\caption{$L^2$-norms of the error at time $t=0.52632$} 
	\label{tab:errors}
\end{table}

From these errors, it can be noted that the SAT-RBF method yields more accurate solutions compared to the usual RBF method in all cases. 
The increase in accuracy by going over to the SAT-RBF method is only small for this problem. 
Still, this test demonstrates that the RBF-SAT approach can also be used to construct provable stable RBF methods in higher dimensions without decreasing the overall accuracy of the usual RBF method. 
Finally, it should be noted that the CFL number had to be decreased (from $0.1$ for the usual RBF method) to $0.01$ for the SAT-RBF method, in order to obtain stable computations. 
This might be considered as a drawback of the proposed SAT-RBF method. 
At the same time, a reduced time step size is a well-known byproduct of adding dissipation (which is done by the SAT term). 
Future works might elaborate on this observation.  
\section{Concluding Thoughts} 
\label{sec:summary} 

In this work, we investigated stability properties of global RBF methods for linear advection problems. 
In particular, stable incorporations of BCs were addressed. 
Classically, in RBF framework, BCs are implemented strongly. 
This can trigger stability issues in the underlying RBF methods, however. 
Here, we demonstrated how such issues can be remedied by adapting some well-known techniques from the FD and FE communities. 
This revealed some, so far, unexplored connections between RBF methods and classical FD and FE schemes.
In the process, we proposed two novel RBF approaches, respectively referred to as  FR-RBF and SAT-RBF methods. 
By weakly imposing the BCs through numerical fluxes (in the FR-RBF method) or SATs 
(in SAT-RBF method), linear stability could be proven for these schemes.
A list of numerical simulations support our theoretical findings and demonstrate the advantages 
of using stable RBF methods compared to usual (unstable) ones. 
Yet, it should also be stressed that particular methods proposed in this work should not be considered as some kind of ’ultimate’ RBF schemes for linear advection equations.
Rather, we hope that our investigation will pave the way towards a more mature stability theory for RBF methods. 

Future investigations might address SBP-like conditions for RBF methods. 
In particular, a discussion on how the kernel and point distribution could be selected (optimized) to ensure such an SBP property would, in our opinion, be of great interest.

%\section*{Acknowledgements}

\bibliographystyle{siamplain}
\bibliography{literature}

\end{document}